\documentclass[a4paper,12pt]{article} 

\usepackage{amsmath,amssymb,amsthm}
\usepackage[usenames]{color}
\usepackage{graphicx}
\usepackage{amsfonts}
\usepackage{comment}
\usepackage{longtable}

\newtheorem{mainthm}{Main Theorem}
\newtheorem{thm}{Theorem}[section]
\newtheorem{prop}[thm]{Proposition}
\newtheorem{lmm}[thm]{Lemma}
\newtheorem{defn}[thm]{Definition}
\newtheorem{remark}[thm]{Remark}

\newtheorem{cry}[thm]{Corollary}

\newtheorem{ntn}{Notation}[section]

\newtheorem{Not}[thm]{Notation}
\numberwithin{equation}{section}

\newcommand{\pr}{\noindent{\bf [Proof]} \qquad}
\newcommand{\prend}{\hfill \qed \medskip}

\newcommand{\Comm}{{\rm Comm}}
\newcommand{\LCM}{{\rm LCM}}
\newcommand{\rank}{{\rm rank}}

\newcommand{\wt}{{\rm wt}}
\newcommand{\Aut}{{\rm Aut}}

\newcommand{\GCD}{{\rm GCD}}

\newcommand{\CG}{{\cal G}} 
\newcommand{\CH}{{\cal H}}

\newcommand{\CP}{{\cal P}} 
 
\newcommand{\CT}{{\cal T}}

\newcommand{\bC}{{\mathbb C}} 
\newcommand{\bZ}{{\mathbb Z}} 
 
\newcommand{\bQ}{{\mathbb Q}} 
 
\newcommand{\bR}{{\mathbb R}} 
 
\newcommand{\g}{\mathfrak{g}}

\newcommand{\clrblue}{\textcolor{blue}}

\definecolor{dark}{rgb}{0.1,0.1,0.1}
\definecolor{skyblue}{rgb}{0.5,0.5,1}
\definecolor{fadegreen}{rgb}{0.9,1,0.9}

\newcommand{\Com}{{\rm Com}}

\begin{document}

\title{A lattice theoretical interpretation of generalized deep holes of the Leech lattice
vertex operator algebra}
\author{
\begin{tabular}{c}
Ching Hung Lam$^{(1)}$
\footnote{Partially supported by grant AS-IA-107-M02 of Academia Sinica  and MoST grant  110-2115-M-001-011-MY3  of Taiwan}  \
and Masahiko Miyamoto$^{(2)}$
\footnote{Partially supported 
by the Grants-in-Aids
for Scientific Research, No.18K18708, The Ministry of Education,
Science and Culture, Japan.}
\cr 
{\small $(1)$  Institute of Mathematics, Academia Sinica, Taipei 10617, Taiwan} \cr
{\small $(2)$ Institute of Mathematics, University of Tsukuba, Tsukuba, Ibaraki 305-8571, Japan }
\end{tabular}
}
\date{}
\maketitle

\begin{abstract}
We give a lattice theoretical interpretation of generalized deep holes of the Leech lattice VOA $V_\Lambda$.  We show that a generalized deep hole defines a ``true" automorphism invariant deep hole of the Leech lattice. We also show that there is  a correspondence between the set of isomorphism classes of holomorphic VOA $V$ of central charge $24$ having non-abelian $V_1$ and 
	the set of equivalence classes  of  pairs $(\tau, \tilde{\beta})$ satisfying certain conditions, where $\tau\in Co_0$ and $\tilde{\beta}$ is a $\tau$-invariant deep hole of squared length $2$.  It provides a new combinatorial approach towards the classification of holomorphic VOAs of central charge $24$. In particular, we give an explanation for an observation of G. H\"ohn, which relates the weight one Lie algebras of holomorphic VOAs of central charge $24$ to certain codewords associated with the glue codes of Niemeier lattices.  
\end{abstract}

\section{Introduction}
The classification 
of strongly regular holomorphic vertex operator algebras (abbreviated as VOA) of central 
charge $24$  with non-trivial weight one space is basically completed. 
It has been shown that 
there are exactly $70$ strongly regular holomorphic VOAs  
of CFT-type with central charge $24$ and non-zero weight one space and their VOA structures are uniquely determined by the Lie algebra structures of their weight one space. Moreover, 
the possible Lie algebra structures for their weight one subspaces have been given 
in Schellekens' list \cite{Schel}. 
The main tool is the so-called orbifold construction \cite{CMi,DGM,EMS,FLM,MC,Mi3,M}. Nevertheless, their constructions and the uniqueness proofs were  done by  case by case analysis \cite{Dong,DGM,DM2,EMS,EMS1, FLM,KLL, Lam,LL,LS,LS2,LS3,LS4, LS19,LS20,LS5, Mi3,SS}. 
A simplified and uniform construction and proof of uniqueness is somehow expected. Recently, it is proved in \cite{ELMS} (see also \cite{CLM}) that for any holomorphic VOA of central charge $24$ with a semisimple weight one Lie algebra, the VOA obtained by an orbifold construction by an inner automorphism $g$ defined by a $W$-element is always isomorphic to the Leech lattice VOA $V_\Lambda$.  
By taking its reverse automorphism $\tilde{g}$ of $V_{\Lambda}$, there 
is a direct orbifold construction from the Leech lattice VOA $V_{\Lambda}$ to $V$. 
As a consequence, a relatively simpler proof for the Schellekens' list is obtained. M\"{o}ller and Scheithauer \cite{MoScheit} called an automorphism $\tilde{g}\in {\rm Aut}(V_{\Lambda})$ \emph{a generalized deep hole} with a slightly different definition.  
They obtained a classification of generalized deep holes and gave a new proof for the classification of holomorphic VOA of central charge 24 with non-trivial weight one spaces;   we do not need these classification results in this article.   

As it is well-known, a deep hole of a lattice $L$ is an element $v$ of $\bR L$ such that 
a distance $\min\{ |\!|\alpha-v|\!|\mid \alpha\in L\}$ from $L$ is the largest (covering radius) among elements $v$ in $\bR L$ and so it has a geometrical meaning.  
Furthermore, deep holes of the Leech lattice $\Lambda$ have many interesting 
geometrical and algebraic meanings. Therefore, it is natural to expect some geometrical properties for generalized deep holes. 

Recall that any automorphism $\tilde{g}\in \Aut(V_{\Lambda})$ can be written as 
\[\tilde{g}=\widehat{\tau}\exp(2\pi i\beta(0)),
\] 
where $\tau\in Co.0=O(\Lambda)$, $\beta\in \bR\Lambda^{\tau}$ and $\widehat{\tau}$ denotes a standard lift of $\tau$ in $O(\widehat{\Lambda})$ \cite{B,DN}. 
In this paper, we will show that 
a holomorphic VOA of central charge  $24$ with non-abelian weight one Lie algebra naturally defines 
a pair $(\tau, \tilde{\beta})$, where $\tau\in Co_0$ is the same isometry as defined by $\tilde{g}$ and $\tilde{\beta}$ is a ($\tau$-invariant) deep hole of the Leech lattice related to $\beta$. In particular, 
a generalized deep hole for the Leech lattice VOA defines a ``true" deep hole of the Leech lattice. It provides a new combinatorial approach towards the classification of holomorphic VOAs of central charge $24$.  The key observation is that there is a ``duality" associated with the lattice $\Lambda^\tau$ which changes the level; we call it $\ell$-\emph{duality}.  

\begin{mainthm}[see Theorem \ref{dual}] \label{ellD}
If $\tilde{g}=\widehat{\tau}\exp(2\pi i\beta(0))$ is a generalized deep hole, 
then there is an isometry: 
$$ \varphi_{\tau}: \sqrt{\ell}(\Lambda^{\tau})^{\ast} \to \Lambda^{\tau}, $$
where $\ell=|\widehat{\tau}|$. 
\end{mainthm}

Using this $\ell$-duality $\varphi$ (extended to $\bC\Lambda^{\tau}$), it is easy to show that
$\langle \varphi_{\tau}(\sqrt{\ell}\beta), \varphi_{\tau}(\sqrt{\ell}\beta) \rangle\in 2\bZ$ and its neighbor lattice  
 $N=\Lambda_{\varphi_{\tau}(\sqrt{\ell}\beta)}+
\bZ\varphi_{\tau}(\sqrt{\ell}\beta)$ defines a Niemeier lattice with $N\ncong \Lambda$.
 
Let $V=(V_{\Lambda})^{[\tilde{g}]}$ be the VOA obtained by an orbifold construction from $V_{\Lambda}$ 
by $\tilde{g}$. Let $\CH$ be a Cartan subalgebra of $V_1$. 
Then there is an even lattice $L$ such that 
$$\Comm(\Comm(M(\CH),V),V)\cong V_L.$$ The lattice $L$ (or $L^*$) encoded the information 
of the root system of $V_1$. 

\begin{mainthm} [see Theorem \ref{duality}]
	Via $\ell$-duality, we have 
	$ N^{\tau}\cong \sqrt{\ell}\varphi_{\tau}(L^{\ast})$, where 
$L^{\ast}$ denotes the dual lattice of $L$. 
\end{mainthm}

In addition, we prove the following theorem.

\begin{mainthm}[see Theorem \ref{deephole} and Proposition \ref{P0}]
Let $\varphi$ be the isometry defined in the Main Theorem \ref{ellD}. Then 
we can choose $\beta\in \bC\Lambda^{\tau}$ such that 
$\tilde{g}=\widehat{\tau}\exp(2\pi i\beta(0))$ and 
$\tilde{\beta}=\sqrt{\ell}\varphi(\beta)$ is a deep hole of $\Lambda$ of squared length $2$. 
\end{mainthm} 

The classification of generalized deep holes is thus equivalent to the classification of the pairs $(\tau, \tilde{\beta})$ with $\tau\in O(\Lambda)$ and  $\tilde{\beta}$ 
a $\tau$-invariant deep hole of squared length $2$ satisfying certain conditions and up to some equivalence. 

Let $\CT$ be the set of pairs $(\tau, \tilde{\beta})$ 
satisfying certain conditions (see Section \ref{S:7}, (C1)--(C3)).
We define a relation $\sim$  on $\CT$ as follows:
$(\tau,\tilde{\beta})\sim (\tau',\tilde{\beta}')$ if and only if \\
(1)  
$\tilde{\beta}$ and $\tilde{\beta}'$ are equivalent deep holes of the Leech lattice $\Lambda$, i.e., there are $\sigma\in O(\Lambda)$ and $\lambda\in \Lambda$ such that $\tilde{\beta}'= \sigma( \tilde{\beta}-\lambda)$;\\
(2) $\tau$ is conjugate to $\sigma^{-1} \tau' \sigma$ in $O(N)$.

The following is another main theorem of this article. 

\begin{mainthm}[see Theorem \ref{Main}]
	There is a one-to-one correspondence between the set of isomorphism classes of holomorphic VOA $V$ of central charge $24$ having non-abelian $V_1$ and 
	the set $\CT/\sim$ of equivalence classes  of  pairs $(\tau, \tilde{\beta})$ by $\sim$. 
\end{mainthm}

Since a deep hole determines a unique Niemeier lattice, up to isometry and there are only $23$ Niemeier lattices  with non-trivial root system, it is straightforward to list all possible choices for $(\tau, \tilde{\beta})$ and to determine the corresponding Lie algebra structures for $V_1$. 
Therefore, one can complete the classification of 
holomorphic VOAs of central charge 24 with non-abelian $V_1$   
by a purely combinatorial method. Indeed, we will provide an explanation for an observation of H\"ohn, which relates the weight one Lie algebras of holomorphic VOAs of central charge $24$ to certain codewords associated with the glue codes of Niemeier lattices \cite[Theorem 3.1 and Table 3]{Ho} (see Section \ref{S:8} for details).


\section{Some previous results}\label{S:2}
We first  recall several results from \cite{CLM}. 
Let $V$ be a holomorphic VOA of 
central charge 24 with $V_1\not=0$. 
Suppose  $V_1=\oplus_{j=1}^t \CG_{j,k_j}$ is a semisimple Lie algebra. We use 
$k_j$ (resp. $h^{\vee}_j$ and $r_j$) to denote the level (resp. the dual Coxeter number and the lace number) of $\CG_j$, $j=1, \dots, t$.  
Let $\CH$ be a Cartan subalgebra of $V_1$. 
It is easy to see that 
$\Comm(\Comm(M(\CH),V),V)$ is isomorphic to a lattice 
VOA $V_L$ for an even lattice $L\subseteq \CH$. 
Let $\rho_j$  be a Weyl vector of $\CG_{j,k_j}$ and set $\alpha=\sum_{j=1}^t \rho_j/h_j^{\vee}$, which we call \textit{a W-element}.  
From the property of lattice VOA, we have 
a decomposition 
$$V=\oplus_{\delta\in L^{\ast}} U(\delta)\otimes M(\CH)e^{\delta},$$
where $U(\delta)$ are $\Comm(M(\CH),V)$-modules and 
$M(\CH)e^{\delta}$ are $M(\CH)$-module. 
Since $V$ is holomorphic, we have $U(\delta)\not=0$ for every $\delta\in L^{\ast}$.

Set $g=\exp(2\pi i\alpha(0))$. 
In \cite{CLM} (see also \cite{ELMS}), it is proved  the VOA $V^{[\alpha]}$ obtained by an orbifold construction 
 from $V$ by the inner automorphism $g$ is isomorphic to the Leech lattice VOA, i.e.,  $V^{[\alpha]}\cong V_{\Lambda}$.  
Namely, the integer weights submodule $T^1_{\bZ}$ 
of the simple $g$-twisted $V$-module $T^1$ is nonzero and 
$V^g\oplus T^1_{\bZ}\oplus  \cdots \oplus (T^1_{\bZ})^{\boxtimes (|g|-1)} \cong V_{\Lambda}$.

In this case, there is an automorphism $\tilde{g}$ of $V_{\Lambda}$ acting on  
$(T^1_{\bZ})^{\boxtimes m}$ as $e^{2\pi i m/n}$, $n=|g|$. Moreover, the orbifold construction  from $V_{\Lambda}$ 
by an automorphism $\tilde{g}$ gives $V$, i.e., $V_{\Lambda}^{[\tilde{g}]}\cong V$.  That $\tilde{g} \in \Aut(V^{[g]})$ is called the reverse automorphism of $g$. 

Since $\tilde{g}\in \Aut(V_{\Lambda})$, we have $\tilde{g}=\hat{\tau}\exp(2\pi i\beta(0))$ for some $\tau\in O(\Lambda)=Co.0$ and $\beta\in \bC\Lambda$ \cite{B,DN}. By choosing a suitable (standard) lift $\hat{\tau}$ of  
$\tau$, we may choose $\beta\in \bC\Lambda^{\tau}$ \cite{EMS,LS}. 
 
\begin{remark}\label{og}
	By the definition of W-element, it is clear that the order of $g$ on $V_1=\oplus_{j=1}^t \CG_{j,k_j}$ is $\LCM(\{r_ih_i^{\vee} \mid i\}$; indeed, one can show that $|g|=\LCM(\{r_ih_i^{\vee} \mid i\}$ on $V$ \cite[Proposition 5.1]{ELMS}.  
\end{remark}	
 
 Set  $n=|g|=|\tilde{g}|$ and let $K_0,N_0\in \bZ$ with $\GCD(K_0,N_0)=1$ such that $\langle \alpha,\alpha\rangle=\frac{2K_0}{N_0}$. 
As we have shown in \cite{CLM}, 
$n=|\tau|N_0$ and $N_0|h^{\vee}_j$; thus, $|\tau|=\LCM(\{r_ih_j^{\vee}/N_0\mid j \})$. 

It is also proved in \cite[Proposition 4.2]{CLM} that $k_i/h_i^{\vee}=\frac{K_0-N_0}{N_0}$; therefore we have 
\[
\frac{2}{r_ik_i}=\frac{2}{r_ih^{\vee}_i}\frac{N_0}{K_0-N_0}=
\frac{2}{(K_0-N_0)r_ih^{\vee}/N_0} 
\]and 
\[
{\LCM(\{r_ik_i\mid i\})}
={(K_0-N_0)\LCM(\{r_ih^{\vee}/N_0\mid i\})}
={(K_0-N_0)|\tau|}. 
\]
As a consequence, we have the following result.

\begin{lmm} \label{rk}
$\LCM(\{r_ik_i\mid i\}) =(K_0-N_0)|\tau|$.  
\end{lmm}


\medskip

Next we recall an important result from \cite{CLM}.  
\begin{prop}[{\cite[Propositions 3.25 and 4.3]{CLM}}]\label{Comm}
We have $N_0\alpha\in L^{\ast}$. Moreover,  \[ 
\Comm(M(\CH),V^{[g]})=\oplus_{j=0}^{|\tau|-1}U(jN_0\alpha)
\]
and 
$\CH+\Comm(M(\CH),V)_1$ is a Cartan subalgebra of $V^{[g]}\cong V_\Lambda$. 
\end{prop}

Since $U(jN_0\alpha)\otimes e^{jN_0\alpha}\in V$, 
$U(jN_0\alpha)$ appears in $g^{-jN_0}$-twisted 
$V$-module and so 
we may choose $\beta$ and $\hat{\tau}$ 
so that 
$\langle \beta,\alpha\rangle\equiv 1/N_0|\tau| \pmod{\bZ}$ and 
$\widehat{\tau}$ acts on $U(jN_0\alpha)$ as a multiple by 
$e^{-2\pi ij/|\tau|}$.

\medskip

As we discussed, a $W$-element of a holomorphic VOA of central charge $24$ with a semisimple weight one Lie algebra defines an automorphism $\tilde{g}=\hat{\tau}\exp(2\pi i\beta(0))$ for some $\tau\in \Aut(\Lambda)=Co.0$ and $\beta\in \bC\Lambda^{\tau}$ such that $V= V_\Lambda^{[\tilde{g}]}$.   One main question is to determine the isometry $\tau\in Co.0$ arising in this manner.

Denote 
$$\CP=\left\{ \tau\in Co.0\, \left | \begin{array}{l}
	{}^\exists \beta\in \bQ\Lambda^{\tau} \mbox{ s.t. }
	\widehat{\tau}\exp(2\pi i\beta(0)) 
	\mbox{ can be realized as the reverse }\\
	\mbox{ automorphism of  some orbifold construction 
		given by a $W$-element} 
\end{array}\right. \right\}. $$
In Section \ref{sec:6}, we will show  that $\CP \subset \CP_0 = \{1A,2A,2C,3B,4C,5B,6E,6G,7B,8E,10F\}$. 
The main idea is to analyze the conformal weights of the irreducible $\hat{\tau}$-twisted modules and $\tilde{g}$-twisted modules.

\section{Irreducible twisted modules for lattice VOAs}\label{Sec:twist}       

First  we review some basic properties about the irreducible twisted modules for lattice VOAs. 
Let $P$ be an even unimodular lattice.
Let $\tau\in O(P)$ be of order $n$ and $\hat{\tau} \in O(\hat{P})$ be a standard lift of $\tau$.
We also use $\pi=\pi_{\tau}: P \to  (P^{\langle \tau\rangle})^*$ to denote the natural projection. 
More preciously, 
\begin{equation}
\pi(x)=\pi_\tau(x) =\frac{1}{n} \sum_{i=0}^{n-1} \tau^i (x).
\end{equation}
By \cite{DLM2}, $V_P$ has a unique irreducible $\hat{\tau}$-twisted $V_P$-module, up to isomorphism.
Such a module $V_P[\hat{\tau}]$ was constructed in \cite{DL} explicitly; as a vector space,
$$V_P[\hat{\tau}]\cong M(1)[\tau]\otimes\bC[\pi(P)]\otimes T,$$
where $M(1)[\tau]$ is the ``$\tau$-twisted" free bosonic space, $\bC[\pi(P)]$ is the group algebra of $\pi(P)$ and $T$ is an irreducible module for a certain ``$\tau$-twisted" central extension of $P$. (see \cite[Propositions 6.1 and 6.2]{Le} and \cite[Remark 4.2]{DL} for detail).
Recall that $$\dim T=|P_\tau/(1-\tau)P|^{1/2}$$ and that the conformal weight $\phi(\tau)$ of $T$ is given by  
\begin{equation}
\phi(\tau):=\frac{1}{4n^2}\sum_{j=1}^{n-1}j(n-j)\dim \mathfrak{h}_{(j)},\label{Eq:rho}
\end{equation}
where $\mathfrak{h}_{(j)}=\{x\in\mathfrak{h}\mid \tau(x)=\exp((j/n)2\pi\sqrt{-1})x\}$.
Note that $M(1)[\tau]$ is spanned by vectors of the form
$$ x_1(-m_1)\dots x_s(-m_s)1,$$
where $m_i\in(1/n)\bZ_{>0}$ and $x_i\in\mathfrak{h}_{(nm_i)}$ for $1\le i\le s$.

In addition, the conformal weight of $x_1(-m_1)\dots x_s(-m_s)\otimes e^\alpha\otimes t\in V_P[\hat{\tau}]$ is given by 
$$
\sum_{i=1}^s m_i+\frac{(\alpha|\alpha)}{2}+\phi(\tau),
$$
where $x_1(-m_1)\dots x_s(-m_s)\in M(1)[\tau]$, $e^\alpha\in\bC[\pi(P)]$ and $t\in T$.
Note that $m_i\in(1/n)\bZ_{>0}$ and that the conformal weight of $V_P[\hat{\tau}]$ is $\phi(\tau)$.

Since $\sum_{j=1}^{n-1} j(n-j)= n(n^2-1)/6$,  we have 
\[
\phi(\tau)=\frac{1}{24}\sum_{i=1}^d a_i\frac{(n_i^2-1)}{n_i}
=\frac{1}{24}\{\sum a_in_i-\sum_{i=1}^d \frac{a_i}{n_i}\}=
1-\frac{1}{24}\sum_{i=1}^d\frac{a_i}{n_i}  
\]
if  $\tau\in O(P)$ has the the frame shape $\prod_{j=1}^d n_j^{a_j}$ by \eqref{Eq:rho}. 

\begin{remark} \label{Twg}
Let $v\in\bQ\otimes_\bZ P^\tau \subset \mathfrak{h}_{(0)}$.
Then $\exp(2\pi i v(0))$ has finite order on $V_P$ and commutes with $\hat{\tau}$.
Set $g= \hat{\tau} \exp(2\pi i v(0))$.  Then the unique irreducible $g$-twisted for $V_P$ is given 
by 
\[
V_P[g]\cong M(1)[\tau]\otimes\bC[-v+\pi(P)]\otimes T
\]
as a vector space \cite{BK}. 
In this case, the conformal weight of $V_P[g]$ is given by 
$$
\frac{1}2\min\{(\beta|\beta)\mid \beta\in -v+\pi(P)\}+\phi(\tau).
$$
\end{remark}

\section{Leech lattice and $\ell$-duality}

\subsection{$\ell$-duality}
In this subsection,  we
 study the properties of the lattice $\Lambda^\tau$ for 
 \[
 \tau \in  \CP_0= \{1A,2A,2C,3B,4C,5B,6E,6G,7B,8E,10F\}. 
 \]
 
\begin{lmm}
	For $\tau\in \CP_0$, we have $\det(\Lambda^{\tau})=\ell^{\rank(\Lambda^{\tau})/2}$.
\end{lmm} 

\pr  For $\tau\in \CP_0$, it is easy to check that $(1-\tau) (\Lambda^*_\tau )=\Lambda_\tau$. It follows that $\det(\Lambda^{\tau})=\det(\Lambda_\tau)=\prod m^{a_m}$, where $\prod m^{a_m}$ is the frame shape of $\tau$. The result then follows by a direct calculation (cf. Table 1 in Section \ref{sec:6}). 
\prend

Below is one of our key observations.  

\begin{thm}\label{dual}
If $\tau\in \CP_0$, then there is an isometry 
$$\varphi=\varphi_\tau: \sqrt{\ell}(\Lambda^{\tau})^{\ast} \to \Lambda^{\tau},$$
where $\ell=|\hat{\tau}|$. 
\end{thm}

\begin{remark}
It is known that $\ell=|\tau|$ or $2|\tau|$ and $\ell=2|\tau|$ if and only if 
$<\tau>$ contains a $2C$-element, see \cite{B}.  
\end{remark}

In \cite{HaLa},  Harada and Lang have determined the structure of 
$\Lambda^{\tau}$ for $\tau\in Co.0$. In particular, the Gram matrix for the lattice $\Lambda^\tau$ has been 
given explicitly. It is  straightforward to check that Theorem \ref{dual} holds for 
$$\tau\in \CP_0=\{1A,2A,2C,3B,4C,5B,6E,6G,7B,8E,10F\}.$$
Note that the Gram matrix of the dual lattice $L^*$ is equal to the inverse of the Gram matrix of $L$. Therefore, it suffices to check that the Gram matrix of $\Lambda^\tau$ is equal to  $\ell$ times the Gram matrix of $(\Lambda^\tau)^*$ for any $\tau\in \CP_0$.

\begin{remark}
	Let $\hat{\tau}$ be a standard lift of $\tau$ in $\Aut(V_\Lambda)$ and let $\phi(\tau)$ be the conformal weight of the irreducible $\widehat{\tau}$-twisted module of $V_\Lambda$. 
	By direct calculations, it is straightforward to verify that for any $\tau\in \CP_0$, we have 
	\[ 
	\phi(\tau^m)= 1-\frac{1}{|\widehat{\tau^m}|} \qquad  \mbox{ for all } m| |\tau|.
	\] 
	In other words,  $\widehat{\tau^m}$ is of type $0$ for any $\tau\in \CP_0$ and $m \in \bZ$ as defined in \cite{EMS}. 
\end{remark}

\subsection{Reverse automorphisms and associated Niemeier lattices}\label{S:4.2}

From now on, we assume that 
$\tilde{g}=\widehat{\tau}\exp(2\pi i\beta(0))\in \Aut(V_{\Lambda})$ is the  
reverse automorphism associated with an orbifold construction 
defined  by a W-element and assume that $\tau\in \CP_0$ 
and $\tau(\beta)=\beta$.

First we discuss some relations between  the lattice $L$ (see Section \ref{S:2}) associated with the roots of $V_1$ and the fixed point sublattice $\Lambda^\tau$ of the Leech lattice.  

\begin{ntn}
	Let $X$  be an even lattice and let $\beta\in \bQ \otimes_\bZ X$. 
Denote 	$X_{\beta} = \{ x\in X \mid \langle x, \beta \rangle \in \bZ\}$. Now suppose that $\langle \beta,\beta\rangle =2k/n$ for some positive integers $k$ and $n$ with $(k,n)=1$. 
	Let $\bar{X}=\{x\in X\mid \langle x,n \beta \rangle\in \bZ\}$. 
	For $x\in \bar{L}$, we
	define $x^{[\beta]} = x- m\beta$ if $\langle x, n\beta\rangle \equiv mk \mod n$ and $0\leq m\leq n-1$ and $\bar{X}^{[\beta]} = \{ x^{[\beta]}\mid  x \in \bar{X}\}$.
Then $\bar{X}^{[\beta]}$ is also an even lattice and $\det(\bar{X})=\det(\bar{X}^{[\beta]})$.  
\end{ntn}

Recall that  $\langle \alpha, \alpha\rangle  =2K_0/N_0$ and $N_0\alpha \in L^*$ for a W\clrblue{-}element $\alpha$ (cf. Proposition \ref{Comm}). Thus we have $\bar{L}=L$. 
Since $V^{[g]}\cong V_\Lambda$, we have  $V^g \cong (V_\Lambda)^{\tilde{g}}$ and  
 \[
V_{L_\alpha}= \Com(\Com( M(\widehat{\CH}),V^g), V^g)  \cong 
 \Com(\Com( M(\widehat{\CH}),(V_\Lambda)^{\tilde{g}}), (V_\Lambda)^{\tilde{g}}) = V_{\Lambda_\beta^\tau}.
 \] 
 It implies $L_\alpha\cong \Lambda^\tau_\beta$. Moreover, $\Lambda^\tau  = L^{[\alpha]}+ \bZ N_0\alpha$.

Now consider the irreducible $\tilde{g}$-twisted module 
\[
V_\Lambda[\tilde{g}]\cong \bC[-\beta +(\Lambda^\tau)^*]\otimes M(1)[\tau]\otimes T.
\] 
Since $(V_\Lambda[\tilde{g}])_\bZ \neq 0$ and $\phi(\tau)=1-1/|\widehat{\tau}|$ for $\tau\in \CP_0$, there exist\clrblue{s} $\beta'\in -\beta +(\Lambda^\tau)^*$ such that $\langle \beta', \beta'\rangle / 2\equiv 1/|\widehat{\tau}| \mod \bZ$.  Without loss, we may assume  $\beta = -\beta'$. In this case, $\beta \in L^*$.   
By a general result (cf. \cite{Ho, Lam20}), we also have  

(1) $(\mathcal{D}(L), q)\cong (\mathrm{Irr}(V_{\Lambda_\tau}^{\hat{\tau}}), -q')$ as quadratic spaces; 
	
(2) $\det(L) = \det(\Lambda^\tau) \times |\tau|^2$. 
	
The quadratic space structure of $(\mathrm{Irr}(V_{\Lambda_\tau}^{\hat{\tau}}), q')$ has also been determined in \cite{Lam20}. In particular, it has proved that the exponent of $L^*/L$ is $\ell =|\widehat{\tau}|$ and $q(L^*)\subset \frac{1}{\ell}\bZ$ .  Thus we have the following lemma.  

\begin{lmm} \label{Lbeta}
We have $\ell \beta \in L$ and $L = \Lambda_\beta^\tau + \bZ \ell \beta$.  
Moreover, $\sqrt{\ell} L^*$ is an even lattice. 
\end{lmm}


By Theorem \ref{dual}, there is an isometry $\varphi=\varphi_\tau: \sqrt{\ell} (\Lambda^\tau)^* \to \Lambda^\tau$ with $\ell =|\widehat{\tau}|$ and it induces an isometry from $\bC \Lambda^{\tau}  \to  \bC \Lambda^{\tau}$.

\begin{defn}\label{N=Lbeta}
	Set $\tilde{\beta}=\sqrt{\ell}\varphi(\beta)$ and $N:= \Lambda^{[\tilde{\beta}]}
	= \Lambda_{\tilde{\beta}} +\bZ\tilde{\beta}$. 
	
	By our assumption, $\langle \tilde{\beta}, \tilde{\beta}\rangle \in 2\bZ$ and $N$ is an even unimodular lattice. 
\end{defn}

\begin{thm} \label{indexn}
Let $n=|g|=|\tilde{g}|$. Suppose $m\varphi(\sqrt{\ell}\beta)\in \Lambda$. Then  we have $n|m$.  Moreover, $[ \Lambda^{[\tilde{\beta}]}: \Lambda_{\tilde{\beta}}]=n$.
\end{thm}

\pr 
If $m\varphi(\sqrt{\ell}\beta)\in \Lambda$, then $m\varphi(\sqrt{\ell}\beta)\in \Lambda^{\tau}=\varphi(\sqrt{\ell}(\Lambda^{\tau})^*)$, which is equivalent to 
 $m\sqrt{\ell}\beta\in \sqrt{\ell}(\Lambda^{\tau})^{\ast}$ (or equivalently,  
 $m\beta\in (\Lambda^{\tau})^{\ast}$). 

Since $[\Lambda^\tau: \Lambda^\tau_\beta]=n ( = [ (\Lambda^\tau_\beta)^* : (\Lambda^\tau)^*]) $ and $(\Lambda^\tau_\beta)^* = (\Lambda^\tau)^* +\bZ \beta$,  
$m\beta\in (\Lambda^{\tau})^{\ast}$ implies $n$ divides $m$. 

For the second statement, it suffices to show $n\sqrt{\ell} \varphi(\beta) \in \Lambda^\tau$. Let $k=|\tau|$. Then  $\widehat{\tau}^k = \exp(2\pi i\delta(0))$ for some $\delta\in \bC \Lambda^\tau$.  Since $n=|\tau| N_0=kN_0$
and $\tilde{g}^n= \widehat{\tau}^n\exp(2\pi i\beta(0))^n=1$,  
we have $n\beta -N_0 \delta \in \Lambda^\tau$. 

As $\widehat{\tau}$ is a standard lift, $\widehat{\tau}(e^\gamma)=e^\gamma$ for any $\gamma \in \Lambda^\tau$.  Thus, $\langle \delta, \Lambda^\tau\rangle \subset \bZ$ and $\delta \in (\Lambda^\tau)^*$.  By using the isometry $\varphi$, we have 
$n\sqrt{\ell} \varphi(\beta) - N_0 \sqrt{\ell} \varphi(\delta) \in \sqrt{\ell}  \varphi(\Lambda^\tau) < \sqrt{\ell}  \varphi((\Lambda^\tau)^*)= \Lambda^\tau$. 
Since $\delta \in (\Lambda^\tau)^*$, we have $\sqrt{\ell} \varphi(\delta) \in \Lambda^\tau$ and $n\sqrt{\ell} \varphi(\beta) \in \Lambda^\tau$ as desired.
\prend

\begin{thm}\label{duality}
Let $N= \Lambda^{[\tilde{\beta}]}
= \Lambda_{\tilde{\beta}} +\bZ\tilde{\beta}$.  Then $\varphi$ induces an isometry 
from $\sqrt{\ell}L^*$  to $N^\tau$. In particular, we have $N^\tau\cong \sqrt{\ell}L^*$.  
\end{thm}

\pr  Since $\tau$ fixes $\varphi(\sqrt{\ell}\beta)$, we may assume that $\tau$ acts on $N=\Lambda^{[\varphi(\sqrt{\ell}\beta)]}$ and  $N^\tau = \Lambda^\tau_{\tilde{\beta}} +\bZ \tilde{\beta}$. Note also that $\langle \beta, \beta \rangle \in (2/\ell) \bZ$.  

  By Lemma \ref{Lbeta}, we have $L= \Lambda_\beta^\tau +\bZ \ell \beta$. Thus, 
\[
\begin{split}
L^* & = \{ x\in (\Lambda_\beta^\tau)^*\mid \langle x, \ell \beta \rangle \in \bZ\}\\
&= \{ x\in (\Lambda^\tau)^*+\bZ \beta \mid \langle x, \ell \beta \rangle \in \bZ\}\\
& = < \{ x\in (\Lambda^\tau)^* \mid \langle x, \ell \beta \rangle \in \bZ\}, \beta >.
\end{split}
  \]

Then  
\[
\begin{split}
\sqrt{\ell} \varphi(L^*)  & = < \{ \sqrt{\ell} \varphi(x) \in \sqrt{\ell} \varphi(\Lambda^\tau)^*\mid \langle \sqrt{\ell} \varphi(x) , \sqrt{\ell} \varphi(\beta) \rangle \in \bZ\}, \sqrt{\ell} \varphi( \beta) >\\
&=<  \{ y\in \Lambda^\tau \mid   \langle y, \tilde{\beta} \rangle \in \bZ\}
, \tilde{\beta} >\\
&= \Lambda^\tau_{\tilde{\beta}} +\bZ \tilde{\beta} = N^\tau
\end{split}
  \]
as desired. \prend

\begin{lmm}\label{nL}
Under the above conditions, we have $\Lambda^{[\varphi(\sqrt{\ell}\beta)]}\not\cong \Lambda$.
\end{lmm}

\pr 
Since $\tau$ fixes $\varphi(\sqrt{\ell}\beta)$, we may assume that $\tau$ acts on $N=\Lambda^{[\varphi(\sqrt{\ell}\beta)]}$.  Then $\tau$ induces an isometry $\tau'$ of $N=\Lambda^{[\widetilde{\beta}]}$. 

Suppose $N \cong \Lambda$.   
Since all elements in $O(\Lambda)$ are determined by its frame shape up to conjugate, 
we may assume $\tau'$ is conjugate to $\tau$. 
In particular,  $N^{\tau'}$ is isometric to $\Lambda^{\tau}$ and so 
$\sqrt{\ell}(N^{\tau'})^{\ast} \cong N^{\tau'}$. 
By Theorem \ref{duality}, $N^{\tau'}\cong \sqrt{\ell}L^{\ast}$ and so 
$(N^{\tau})^{\ast}= \frac{1}{\sqrt{\ell}}L$.  Then $\Lambda^{\tau} \cong \sqrt{\ell}(\Lambda^{\tau})^{\ast}\cong 
\sqrt{\ell}(N^{\tau'})^{\ast}\cong L$. It is not possible since $\det(L)=\det(\Lambda^{\tau}). |\tau|^2$. 
\prend

As a consequence of the above lemma and Theorem \ref{indexn}, we have the following lemma. 
\begin{lmm}
Let $h$ be the Coxeter number of $N= \Lambda^{[\varphi(\sqrt{\ell}\beta)]}$ and $n= |g|=|\tilde{g}|$. Then we have $n \geq h$. 
\end{lmm}

\subsection{Roots of $V_1$}
Next we study the structure of the root lattice of $V_1$. 
Let $V_1=\oplus_{j=1}^t \CG_{j,k_j}$. If $u\in L^*$ is a root of $\CG_j$, then 
\[
\langle u,u\rangle=
\begin{cases}
\frac{2}{k_j} & \text{ if } u \text{ is a long root,}\\
\frac{2}{r_{j}k_j} & \text{ if } u \text{ is a short root.} 
\end{cases}
\]

As a corollary of $\ell$-duality, we have: 

\begin{lmm}\label{ell=n}
$|\tau|(K_0-N_0)$ divides $\ell$. 
\end{lmm}

\pr   
Let $ e^u \otimes t \in \CG_{j,k_j}$ be a short root vector. 
Then $\langle u,u\rangle=\frac{2}{r_jk_j}$ and we have 
$$\langle u,u\rangle=\frac{2}{r_jk_j}=\frac{2}{r_jh_j^{\vee}}
\times \frac{h_j^{\vee}}{k_j}
=\frac{2}{r_jh_j^{\vee}}\times \frac{N_0}{K_0-N_0}.$$
Since $\langle u,u\rangle \in \frac{\bZ}{\ell}$ and $\LCM(\{r_jh_j^{\vee}/N_0:j\})=|\tau|$, 
we have that $K_0-N_0$ divides $\ell/|\tau|$. 
\prend

\begin{remark}  \label{full0}
Let $ e^u \otimes t \in \CG_{j,k_j}$ be a root vector associated with a simple short root. 
Then $\langle \alpha,u\rangle=\frac{1}{r_jh_j^\vee}$ and  $e^u\otimes t$ belongs to $\tilde{g}^{s_j}$-twisted modules, where 
$s_jr_jh^{\vee}_j=n$. More precisely, it is belong to $s_j$-power of $\tilde{g}$-twisted module $T^1$ by fusion products. 
\end{remark}

\begin{lmm} \label{fullC}
	Suppose there is $j$ such that $r_jh^{\vee}_j=n=|\tau|N_0$. 
	Then there is a root $u\in L^*$ of $\CG_j$ such that $\langle u, u\rangle =2/\ell$. 
\end{lmm}
\pr Let  $u$ be a simple short root of $\CG_j$. Since $r_jh^{\vee}_j=n=|\tau|N_0$, we have    $\langle u, u\rangle = \frac{2}{|\tau|(K_0-N_0)}$ and $e^u \otimes t \in T^1_1$. It implies $\langle u,u\rangle/2 \in 1/\ell +\bZ$. 
Note also that $\ell/ |\tau|= 1$ or $2$ and $ |\tau|(K_0-N_0)$ divides $\ell$. 
Therefore, $\ell = |\tau|(K_0-N_0)$ and $\langle u,u\rangle =2/\ell$. 
\prend

\begin{defn}\label{fullCandSroots}
We call a root $\delta$ of $V_1$ satisfying 
$\langle \delta,\delta\rangle=\frac{2}{\ell}$ a shortest root and call a simple component $\CG_j$ of $V_1$ containing a shortest root 
a full component. Note that a shortest root exists if $r_jh_j^{\vee}=n$ for some $j$.	
\end{defn}

\section{Deep hole}
Let $V$ be a holomorphic VOA of central charge $24$ and 
$\alpha$ a W-element of $V_1$. Let 
$g=\exp(2\pi i\alpha(0))\in \Aut(V)$ and let 
$\tilde{g}=\hat{\tau}\exp(2\pi i\beta(0))\in \Aut(V_{\Lambda})$ be 
the reverse automorphism of $g$, where $\beta\in \bC\Lambda^{\tau}$.  

In this section, we assume that $\tau \in \CP_0$ and try to relate the automorphism 
$\tilde{g}=\hat{\tau}\exp(2\pi i \beta(0)\in \Aut(V_\Lambda)$ 
to a deep hole of the Leech lattice. 
In Section \ref{sec:6}, we will prove that $\tau\in \CP_0$.

Since $\tau\in \CP_0$, there is an isometry $\varphi: \sqrt{\ell} (\Lambda^\tau)^* \to \Lambda^\tau$.  
Set $\tilde{\beta}=\sqrt{\ell}\varphi(\beta)$ and  $N=\Lambda^{[\tilde{\beta}]}= \Lambda_{\tilde{\beta}}+\bZ \tilde{\beta}$. 
Note that $\langle \beta,\beta\rangle\in 2\bZ/\ell$ 
and $\langle \tilde{\beta},\tilde{\beta}\rangle \in 2\bZ$. 
Moreover, the root sublattice $R$ of $N$ is non-zero by Lemma \ref{nL}.  
 
One of the  main aims in this paper is to prove the following theorem. 

\begin{thm}\label{deephole}
The vector  $\tilde{\beta}=\varphi(\sqrt{\ell}\beta)$ 
is a deep hole of $\Lambda$. 
\end{thm} 


\subsection{Coinvariant sublattices $\Lambda_\tau$ and $N_\tau$} 

In this subsection, we will discuss another main observation of this article, which is related to the structure of the coinvariant sublattice $N_\tau$.

Since $\tau$ fixes $\tilde{\beta}=\varphi(\sqrt{\ell} \beta)$, we have $(\Lambda_{\tilde{\beta}})_\tau =\Lambda_{\tau}$ and $N_\tau > \Lambda_\tau $. Moreover, $[N_\tau: \Lambda_{\tau}]=|\tau|$ because $\det(N_\tau)=\det(N^\tau) =\det(\Lambda^\tau)/ |\tau|^2=\det(\Lambda_\tau)/ |\tau|^2$.  For $\tau\in \CP_0$, we note that $\Lambda_\tau \cong L_B(c)$ as defined in Appendix \ref{appA}.

\begin{lmm}
	For $\tau\in \CP_0\setminus \{2C\}$, we have $N_\tau \cong  L_A(C)$, where $C=\langle c\rangle$ is defined as in Table 2 in Appendix A.  
\end{lmm}

\pr  It follows from the fact that $\Lambda_{\tau} \cong L_B(C)$ and $\mathcal{D}(\Lambda_{\tau})$ has only one  non-zero singular element of order $|\tau|$, up to isometry if  $\tau\in \CP_0\setminus \{2C\}$ (see Appendix A).  
\prend
 
\medskip

Next we consider the case when $\tau=2C$. In this case, $\Lambda_{\tau}\cong \sqrt{2}D_{12}^+$. 
We use the standard model for root lattices of type $D$, i.e, 
\[
D_{12}=\left \{ (x_1, \cdots, x_{12})\in \bZ^{12}\mid \sum_{i=1}^{12} x_i =0 \mod 2\right\} 
\]
and $D_{12}^+ = \mathrm{Span}_\bZ\{D_{12}, \frac{1}2 (1, \cdots, 1) \}$. Note that $D_{12}^+$ is an odd lattice. 
\begin{lmm}
	For $\tau\in 2C$, there are two classes of non-zero singular elements, up to isometry; they correspond to vectors of the form $\sqrt{2}(1,0, \cdots, 0)$ and $\frac{\sqrt{2}}{2}(1^4 0^8)$, respectively.   
	There are two index $2$ overlattices of $\Lambda_{\tau}$. They are isometric to $L_A(1^{12})\cong \sqrt{2}D_{12}^*$ or  the overlattice $\mathrm{Span}_\bZ\{\sqrt{2}D_{12}^+, \frac{\sqrt{2}}{2}(1^4 0^8) \}$,  which has the root sublattice $A_1^4$. 
\end{lmm}

Let $X=\mathrm{Span}_\bZ\{\sqrt{2}D_{12}^+, \frac{\sqrt{2}}{2}(1^4 0^8) \}$. Then 
\[
\begin{split}
	X^* &= \{ \alpha\in \frac{\sqrt{2}}{2} D_{12}^+\mid \langle \alpha, \frac{\sqrt{2}}{2}(1^4 0^8)\rangle \in \bZ \}\\
	& =  \frac{\sqrt{2}}{2} \mathrm{Span}_\bZ\{D_{4}+D_8, \frac{1}2(1^4 0^8) \}\cong \frac{\sqrt{2}}{2} (\bZ^4+D_8). 
\end{split}
\]
Therefore,  $2X^* \cong A_1^4 + \sqrt{2}D_8$.  Then the quadratic form  of $2X^*$ is not isometric to $L$, which is not possible and hence we have the following lemma. 
 
\begin{lmm}
  For $\tau\in 2C$, 	$N_\tau\cong  L_A(1^{12})\cong \sqrt{2}D_{12}^*$.  
\end{lmm}

As a consequence, for any $\tau\in \CP_0$,  $N_\tau \cong L_A(c)$ as described in Appendix A. We also note that $\tau$ acts on $L_A(c)$ as $g_{\Delta,c}$ defined in 
\eqref{Eq:gre}. In particular, $\tau$ is contained in the Weyl group of $R$. 


\begin{remark}\label{Ntau}
If the frame shape of $\tau$ is $\prod k^{m_k}$, then 
$N_{\tau}$ contains $\oplus A_{k-1}^{m_k}$ and $\tau$ acts on $\oplus A_{k-1}^{m_k}$
as a product of the Coxeter elements. In particular, 
$\tau$ preserves any irreducible component of the root system of $N$. 
\end{remark}

\begin{cry}\label{N0=1}
If $N_0=1$, then the Coxeter number of $N$ is greater than or equal to 
$|\tau|$. 
\end{cry}

\pr 
For $\tau\in \CP_0$, it is easy to check that 
its frame shape of $\tau$ contains a positive power of $|\tau|$. 
Therefore, $N_{\tau}$ contains a root system $A_{|\tau|-1}$ whose Coxeter number is $|\tau|$. 
\prend

\subsection{Affine root system and shortest roots} 

One important fact that we will prove is the following: 

\begin{lmm} 
Let $\beta\in \bC\Lambda$.
Set $S=\{ \lambda\in \Lambda \mid \langle \lambda-\beta,\lambda-\beta\rangle
=2\}$. If $\tilde{S} =\{\lambda -\beta\mid \lambda\in S \}$ contains an affine fundamental root system, then $\beta$ is a deep hole.
\end{lmm}

\pr 
For $\lambda\in \Lambda$, let $\lambda^* = (\lambda, 1, 
\frac{\langle \lambda,\lambda\rangle}{2}-1)\in \Lambda+\Pi_{1,1}$, which is called 
a Leech root.  Set $\hat{\beta}= (\beta,1,\frac{\langle 
\beta,\beta\rangle}{2})\in \bQ (\Lambda+ \Pi_{1,1})$. Note that $\hat{\beta}$ is an isotropic vector. 
Then $\lambda^*-\hat{\beta} = (\lambda-\beta, 0 , \frac{\langle \lambda,\lambda\rangle}{2}-1-\frac{\langle\beta,\beta\rangle}{2})$ and we have 
\[
\begin{split}
\langle \lambda-\beta,\lambda-\beta\rangle
=\langle \lambda^*-\hat{\beta}, \lambda^*-\hat{\beta}\rangle =\langle \lambda^*, \lambda^* \rangle + \langle \hat{\beta},\hat{\beta}\rangle - 2\langle \lambda^*, \hat{\beta} \rangle 
=2-2\langle \lambda^*, \hat{\beta} \rangle.
\end{split}
\]
Thus,  
$\langle \lambda-\beta,\lambda-\beta\rangle=2$ if and only if 
$\langle \lambda^*, \hat{\beta} \rangle=0$.    

Suppose that $\beta$ is not a deep hole. Then there is a $\delta\in \Lambda$ such that $\langle \delta-\beta,\delta-\beta\rangle<2$. In this case, 
$\langle \delta^*, \hat{\beta} \rangle>0$.
Since $\delta^{\ast}\not\in \{\lambda^{\ast}\mid \lambda\in S\}$ and the minimal squared norm of $\Lambda$ is $4$,
$\langle \delta^{\ast},\lambda^{\ast}\rangle \leq 0$ for all $\lambda\in S$.
On the other hand, since $\tilde{S}$ contains an affine fundamental root system, there exists a subset $S'\subset S$ and positive real numbers $n_{\lambda}\in \bR_{\geq 0}$
such that $\sum_{\lambda\in S'}n_{\lambda}\lambda=  m \beta$, where $m=\sum_{\lambda\in S'} n_{\lambda}$. By  direct calculations, we have
\[
\begin{split}
\sum_{\lambda\in S'}n_{\lambda}\left( \frac{\langle\lambda, \lambda\rangle}2 -1\right) &=
\sum_{\lambda\in S'}n_{\lambda} \frac{\langle\lambda, \lambda\rangle -\langle \lambda-\beta, \lambda -\beta\rangle}2\\ &
= \sum_{\lambda\in S'}n_{\lambda} \left( \langle\lambda, \beta\rangle -\frac{\langle \beta,\beta\rangle}2\right) =\frac{m\langle \beta, \beta\rangle}2.  
\end{split}
\]
Thus, we also have 
\[\sum_{\lambda\in S'}n_{\lambda}\lambda^*= \left ( \sum_{\lambda\in S'}n_{\lambda}\lambda, \sum_{\lambda\in S'}n_{\lambda}, \sum_{\lambda\in S'}n_{\lambda}\left( \frac{\langle\lambda, \lambda\rangle}2 -1\right)\right)= m \hat{\beta}
\]
but it implies 
$\langle \delta^{\ast}, \hat{\beta}\rangle\leq 0$, which is a  contradiction.
\prend

Therefore, in order to prove Theorem \ref{deephole}, it is enough to show that the 
Coxeter number $h$ of $N$ is greater than or equal to $n=|g|$, since 
$R_1=\{\lambda-\tilde{\beta}\mid \langle \lambda-\tilde{\beta}, \lambda-\tilde{\beta}\rangle=2, \lambda\in \Lambda_{\tilde{\beta}} \}$ 
 will contain an affine fundamental root system in this case.

\subsection{Existence of shortest roots}
The purpose of  this subsection is to show the following proposition.

\begin{prop}\label{shortestroot}
There is a root $\delta$ of $V_1$ such that $\langle \delta,\delta\rangle
=\frac{2}{\ell}$. 
\end{prop}

\pr  By Lemma \ref{fullC}, it suffices to show that there is $j$ such that $r_jh_j^\vee=n$.   

Recall that ${\rm LCM}(r_jh^{\vee}_j \mid j)=n$. 
If $|\tau|$ is a prime power or $V_1$ is simple, then 
there is $j$ such that $r_jh_j^{\vee}=n$. 
Therefore, we may assume that 
$|\tau|$ is not a prime power and $\rank(V_1)>4$, that is, 
$\tau$ is $6E=1^22^23^26^2$ or $6G=2^36^3$ and $\mathrm{rank}(V_1)=8$ or $6$. 
We have already shown  $N_0| h_j^\vee$ for all $j$. By Corollary \ref{N0=1}, we may also assume $N_0\not=1$.  Suppose there is no $j$ such that $r_jh_j^\vee =n=6N_0$. Then there exist $k$ and $l$ such that  $r_kh_k^{\vee}=3N_0$ and $r_lh_l^\vee=2N_0$. 
If one of $\CG_j$ is of type $G_2$, then $N_0=4$ and there is also a component 
$\CG_i$ such that $r_ih_i^{\vee}=8$, that is, $\CG_i=A_7$, or $C_3$. 
Another possible component is $A_3$.  
Therefore, possible choices for $\tau$ and $V_1$ are $\tau=6E$ and $V_1=G_2C_3^2$, or $G_2C_3A_3$.  However, since $\dim V_1=120\not=14+2\times 21$ nor $14+21+15$, we have a contradiction. 
Therefore, there is no component of type $G_2$. In order to get $r_kh_k^{\vee}=3N_0$, $r_k=1$, $h_k^\vee=3N_0$ and $V_1$ contains a component of type  $A_{3N_0-1}$. 
Since $N_0\not=1$, $\mathrm{rank}(V_1)\leq 8$ and $V_1$ has a component with $r_ih_i^\vee =2N_0$,  we  have  
$N_0=2$ and $V_1=A_3+A_5$ (i.e., $\mathrm{rank}(V_1)=8$ and $\tau=6E$).  
However, $\dim V_1=72\not=15+35$, which is a contradiction.  
This completes the proof of Proposition \ref{shortestroot}.
\prend

As a corollary, $(\Lambda_{\tilde{\beta}}-\tilde{\beta})_2 \not=\emptyset$ and so 
we may choose $\beta$ so that $\langle \beta,\beta\rangle=\frac{2}{\ell}$. 
In particular, $\tilde{\beta}$ has squared length $2$.

\subsection{Proof of Theorem \ref{deephole} (part 2)}

\begin{Not}
	For an even lattice $K$, we use $R(K)$ to denote the sublattice generated by $K_2=\{a\in K\mid \langle a,a\rangle =2\}$. A component lattice of $R(K)$ means a sublattice generated by an irreducible component of the roots in $K_2$.  
\end{Not}

\begin{lmm}
Let $X\subseteq R(N^{\tau})$ be a component lattice and 
$\tilde{X}\subseteq R(N)$ be a component lattice containing $X$. 
If $X\not=\tilde{X}$, then $\rank(\tilde{X})\geq \rank(X)+2$ or 
$\tau$ acts on $\tilde{X}$ trivially. 
\end{lmm}

\pr
If $\tilde{X}\not=X$ and 
$\tau$ does not act on $\tilde{X}$ trivially, then there are $\mu_1\in X$ and $\mu_2\in \tilde{X}-(\tilde{X}^{\tau})^{\perp}$ such that 
$\langle \mu_1,\mu_2\rangle=-1$. Since $\mu_2\not\in N^{\tau}$ and 
$\rank(R(N_{\tau}))=\rank(N_{\tau})$ (see Remark \ref{Ntau}), there is a $\mu_3\in R(N_{\tau})$ such that 
$\langle \mu_2,\mu_3\rangle=-1$. Since $\mu_3\in \tilde{X}$, 
$\rank(\tilde{X})\geq \rank(X)+2$, as we desired. 
\prend

Now let's start the proof of Theorem \ref{deephole}. 
\medskip

Suppose Theorem \ref{deephole} is false and let $V$ be a counterexample. 
As we explained, it is enough to show that the Coxeter number of $N$ is 
greater than or equal to $n$.
By Proposition \ref{shortestroot}, there is a simple component $\CG_j$ of $V_1$ such that 
$r_jh^{\vee}_j=n$ and we may choose $\langle \tilde{\beta},\tilde{\beta}\rangle=2$, or 
equivalently $\langle \beta,\beta\rangle=\frac{2}{\ell}$.   
Suppose  there is a full component $\CG_j$ which is simply laced and let  $L_j$ be its root lattice. Then $\sqrt{\ell}\phi(L_j)\subseteq R^{\tau}$ and so 
$h\geq n$, which is a contradiction. 

Therefore, there are no simply laced full components in $V_1$ and  
$\GCD(6,|\tau|)\not=1$ since the lacing number divides $|\tau|$. 
We also  have $N_0\not=1$ by Corollary \ref{N0=1}. 
Furthermore, if $\rank(V_1)=4$, then we have shown $N_0=1$ in \cite{CLM} and 
we have a contradiction. Therefore, we may assume $|\tau|=2,3,4,6,8$. 
We will need a few lemmas.\\

\begin{lmm}\label{order2} 
If $|\tau|=2$, then $N_0$ is odd. 
\end{lmm}

\pr  
Suppose $N_0$ is even. Then  $K_0$ is odd since ${\rm GCD}(K_0,N_0)=1$. Since $K_0-N_0=1$ or $2$, $K_0=N_0+1$. 
Since $|\tau|=2$, there are no  components of type $G_2$.  
Moreover, $N_0$ is even and $n/N_0=2$; thus $4|n$ and  there are no full components of 
type $B$ and $F$. Therefore, the only full components of $V_1$ are of type $C_{N_0-1,1}$. 
Since $N_0$ divides the dual Coxeter numbers, the other components are of the type $A_{N_0-1}$ or $D_{N_0/2+1}$. 
Recall that $\rank(\Lambda^\tau)=12$ or $16$ and $N_0$ is even. By a direct calculation, it is easy to verify that the possible cases are only  
$V_1=C_{3,1}^a A_{3,1}^b$ ($N_0=4$) with $a+b=4$ and $C_7D_5$ ($N_0=8$). 
Since $\dim V_1=24(N_0+1)$, none of them is possible. Note that $\dim C_3=21, \dim A_3=15, \dim C_7=105,$ and $ \dim D_5=45$. 
\prend

\begin{lmm}\label{counterexample}
Let $V$ be a counterexample of Theorem \ref{deephole}. 
Then $\tau$ and $(V_1, N_0)$ will be one of the following: 
\begin{enumerate}
\item[(1)] $\tau=2A=1^82^8$ and $(V_1, N_0)=(C_{4}^4,5), (C_6^2B_4,7), (C_8F_4^2, 9), (C_{10}B_6, 11)$;   
\item[(2)] $\tau=2C=2^{12}$ and $(V_1, N_0)=(B_2^6, 5), (C_4A_4^2,5), 
(B_3^4,5), (B_4^3,7), (F_4A_8, 9)$, $(B_6^2,11)$, $(B_{12},23)$; 
\item[(3)] $\tau=4F=1^42^24^2$ and $(V_1, N_0)=(C_{7,1}A_{3,1}, 4)$;  
\item[(4)] $\tau=6E=1^22^23^26^2$ and $(V_1, N_0)=(C_5G_2A_1,2)$;  
\item[(5)] $\tau=6G=2^36^3$ and $(V_1, N_0)=(F_{4,6}A_{2,2},3)$. 
\end{enumerate}
\end{lmm}

\pr 
We may assume that $\CG_1$ is a full component. \\
\noindent
{\bf Case 1}: If $\CG_1=F_4$,  then $\dim F_4=52$ and $n=18$ and so  
$(\tau, \rank\Lambda^{\tau},\dim V_1)$ are $(2A,16,240), (2C,12,132)$, 
$(6G,6,60), (6E,8,72)$ and the possible choices for $(\CG_j, \dim \CG_j)$ are $(F_4,52)$,  $(C_8,136)$, $(A_8, 80)$ and $(A_2,8)$.   
Hence the possible cases are $\tau=2A$ and $V_1=F_4^2C_8$; $\tau=2C$ and $V_1=F_4A_8$; and $\tau=6G$ and $V_1=F_4 A_2$.  

From now on, we may assume that there are no full components of type $F_4$. \\ 
{\bf Case 2}: If $\CG_1=G_2$, then $n=12$ and $|\tau|=3$ $(N_0=4, \dim V_1=120)$ or 
$|\tau|=6$ ($N_0=2, \dim V_1=72$ or $60$). If $|\tau|=3$, then $\tau=3B$, $4|h^\vee_i$, $h_i^\vee< 12$ and $\rank(\Lambda^\tau)=12$. 
Then $V_1=G_2^a+A_3^b$ and $2a+3b=12, \dim V_1=14a+15b=120$, which has no solution. 
If $|\tau|=6$, then $V_1=G_2^a+C_5^b+A_1^c+A_3^d+A_5^e$. 
Using dimensions and rank$=2a+5b+c+3d+5e=6$ or $8$, the solution is only $V_1=C_5G_2A_1$. 

From now on, we may assume that 
all full components are of type $C_m$ or of type $B_m$. \\
{\bf Case 3}: If $\CG_1=C_8$, then $n=18$ and $|\tau|=2$ or $6$. \\
(3.1) If $|\tau|=2$, then $N_0=9$ and so possible components are 
$C_8, B_5, A_8$ and so $V_1=C_8^2, C_8A_8$ but $\dim(V_1)=24K_0=240$.  Neither one is possible.  \\ 
(3.2) If $|\tau|=6$, then $N_0=3$ and $V_1=C_8$, but $\dim V_1=96\not=\dim C_8$, a contradiction. \\
{\bf Case 4}: If $\CG_1=C_{2m}$ and $0<m\not=4$, then $n=2(2m+1)$.  
First we show that $m<7$. 
Suppose $m\geq 7$. Then $\rank(V_1)\geq \rank(C_{2m})\geq 14$ and so $\tau=2A$ and
$N_0=2m+1$ and $m=7$ or $8$.
In this case, $\dim V_1=24(2m+2)$ and $\dim C_{2m}=2m(4m+1)$.
If $m=7$, then there is no components $\CG_j$ of rank 2 (or less than 2) with
$h_j^\vee$ divisible by $N_0=15$.
So, $m=8$ and $V_1=C_{16}$; however, $\dim V_1=24(N_0+1)=24\times 18\not=
\dim C_{16}=16\times 33$.

For $m<7$ and $m\neq 4$,
we have $(2m+1,3)=1$, $N_0=2m+1$ and $|\tau|=2$. So, the possible components of $V_1$ 
are $C_{2m}, A_{2m}, B_{m+1}$.  
Since $K_0-N_0=1$ for $\tau=2A$ and $K_0-N_0=2$ for $\tau=2C$, 
$\dim V_1=24(2m+2)$ for $\tau=2A$ and $\dim V_1=12(2m+3)$ for $\tau=2C$. 
If $\tau=2A$, then by solving the relation  
$2ma+2mb+(m+1)c=16$ and $2m(4m+1)a+2m(2m+2)b+(m+1)(2m+3)c=24(2m+2)$ 
and $a>0$ and $b,c\geq 0$, 
we have  $V_1=C_{10}B_6, C_6^2B_4, C_4^4$. \\ 
If $\tau=2C$, then by solving the relation 
$2ma+2mb+(m+1)c=12$ and $2m(4m+1)a+2m(2m+2)b+(m+1)2m+3)=12(2m+3)$, 
we have $V_1=C_{4,2}A_{4,2}^2$ or $C_{2,2}^6$. \\
We next assume $\CG_1=C_{2m+1}$. Then since $2m+2$ is even, $|\tau|\not=2$. \\ 
{\bf Case 5}: If $\CG_1=C_3$, then $n=8$ and $|\tau|=4$ and $N_0=2$ and 
$\CG$ is a direct sum of $C_3$ or $A_1$, or $A_3$. 
Since $\dim C_3=21$, $\dim A_3=15$, $\dim A_1=3$ 
 and $\rank(V_1)=10$, it is impossible to have $\dim V_1=72$ and we have a contradiction. \\
{\bf Case 6}: If $\CG_1=C_5$, then $n=12$ and $(\tau,N_0)=(4F,3),(6,2)$. 
If $\tau=4F$, then the other possible  components are $A_2, D_4, A_5, C_2$ since there is no component of type $G_2$. Since $\rank(V_1)=10$ is even, it has another 
component of type $C_5$ or $A_5$, but since $\dim V_1=96$, $\dim C_5=55$, and  $\dim A_5=35$, 
we have a contradiction.   
If $|\tau|=6$, then $\rank(V_1)=8$ or $6$ and possible components 
are $A_1, A_3$. Since $\dim V_1=72$ or $60$ and $\dim A_1=3$, 
$\dim A_3=15$, we have a contradiction. \\
{\bf Case 7}: If $\CG_1=C_7$, then $n=16$ and $(\tau,N_0)=(4F,4),(8,2)$. 
If $N_0=2$, then $\dim V_1=72$ and $\dim C_7=105$, a contradiction. 
If $\tau=4F$, then $N_0=4$ and the only solution is $V_1=C_7A_3$. \\ 
{\bf Case 8}: If $\CG_1=C_{2m+1}$ and $m\geq 4$, then $\rank(V_1)\geq 2m+1\geq 9$ and so $\CG_1=C_9$. However, since $\rank(V_1)=10$ and 
$N_0=5$, we have a contradiction.  \\ 
{\bf Case 9}: We may assume that all full components are of type $B$. Say, $\CG_1=B_m$, then $n=2(2m-1)$.  
Since $2m-1$ is odd, $|\tau|=2, 6$.  

If  $|\tau|=6$, then  $N_0=1/3(2m-1)$ and $m\leq 8$. Since $N_0\neq 1$ is odd, the only possible values for $(N_0,m)$ are  $(3,5)$ and $(5,8)$.   
If $\tau=6G$, then $\rank V_1=6$ and $(N_0, m)=(3,5)$. The only possible choice  for $V_1$ is $B_5+A_1$, but since the coxeter number of $A_1$ does not divide $3$, a contradiction. \\
If $\tau=6E$, then $\rank V_1=8$, the possible choice of $V_1$ is $B_8$. 
Since $2C\not\in <6E>$, $K_0=N_0+1=4$ and $\dim V_1=24K_0=96\not=\dim B_8=8\times 17$, which is a contradiction. \\
Therefore, we have $|\tau|=2$ and $N_0=2m-1$. 
Since $2m-1$ is odd, the non-full components are all of type $A_{2m-2}$ 
and so $V_1=B_m^{\otimes a}\oplus A_{2m-2}^{\otimes b}$.  
We note $\dim B_m=m(2m+1)$ and $\dim A_{2m-2}=4m(m-1)$. 
If $K_0=2m$, then $am+b(2m-2)=16$ and $\dim V_1=48m=am(2m+1)+4bm(m-1)$ 
It is easy to see that there is no solution for $a\in \bZ_{> 0}$ and $b\in \bZ_{\geq 0}$. 
Hence $K_0-N_0=2$ and $\tau=2C$. In this case, 
$am+b(2m-2)=12$ and $\dim V_1=12(2m+1)=am(2m+1)+4bm(m-1)$. 
The solutions are $b=0$ and $am=12$; therefore, 
$V_1=(B_{2,2})^{\oplus 6}$, $(B_{3,2})^{\oplus 4}$, 
$(B_{4,2})^{\oplus 3}$, $(B_{6,2})^{\oplus 2}$, or  $B_{12,2}$, as we desired.  \\
This completes the proof of Lemma \ref{counterexample}.
\prend

We will show that none of the Lie algebras in the list in Lemma \ref{counterexample} 
satisfy the desired condition and we will get a contradiction.

First we note that $V_1$ contains a full component of type $C_m$ or $F_4$ unless $\tau=2C$. 

\subsubsection{The case $\tau\not=2C$}

We note the short root lattices of $C_m$ are $D_m$.  Namely, if $C_m$ is a full component of $V_1$, then $D_m$ is a sublattice of $N^{\tau}$. 
Let $X$ be a component of the root lattice $R(N)$ of $N$ containing the above $D_m$. 
For $m\geq 4$, a lattice containing $D_m$ is of type $D$ or $E$, that is, 
$X$ is of type $D$ or $E$.   
Recall that  $n=r_jh_j^{\vee}=2(m+1)$ for $C_m$ and the Coxeter numbers of 
$D_m$ and $E_k$ are $2(m-1)$ and $12$, $18$ and $30$. Since $V$ is a counterexample, the Coxeter number of $N$ is strictly less than $n$ and  
the possible choice of $X$ is $D_m$ or $D_{m+1}$.  

\begin{lmm} 
If $X=D_{m+1}$, then $\tau$ acts on $X$ trivially. 
In particular, if $V_1$ is a direct sum of full components, then 
$X=D_m$.  
\end{lmm}

If $\tau=2A$, then the list of lemma 5.15 says that $V_1$ is a direct sum of 
full components. Therefore, if $C_m$ is a full component, then $D_m$ is a connected component of $R(N)$.  
Therefore, the possible cases are as follows:
\begin{center}
\begin{tabular}{|c|c|c|c|c|}
\hline
$V_1$ & $C_{4,1}^4$ & $C_{6,1}^2B_{4,1}$ & $C_{8,1}F_{4,1}^2$ & $C_{10,1}B_{6,1}$ \\ \hline
$R(N)$ & $D_{4}^6$  & $D_6^4$   & $D_8^3$ & $D_{10} E_7^2$ \\ \hline
$R(N^\tau)$ contains &  $D_{4}^4$  & $D_6^2 A_1^4$   & $D_8 D_4^2$ & $D_{10} A_1^6$\\ \hline
\end{tabular}
\end{center}
It contradicts the fact that $N_\tau=L_A(1^8)$. 

If $\tau=4F$, then 
$R(N)$ contains $D_7\oplus A_3^4A_1^2$. 
The possible choice of $R(N)$ is $D_7E_6A_{11}$ or $D_8^3$, but 
neither of them contains $D_7A_3^4A_1^2$, which is a contradiction. 

If $\tau=6E$, then $R(N)$ contains $D_5A_2A_5^2A_2^2A_1^2$. 
The possible choice of $R(N)$ is $D_5^2A_7^2$ or $D_6^4$ or $D_6A_9^2$, 
but none of them contains $D_5A_5^2A_2^3A_1^2$, which is a contradiction. 

If $\tau=6G=2^36^3$, then $n=18$ and $V_1=F_{4,6}A_{2,2}$.  
Hence $X$ contains $D_4A_5^3A_1^3$. Since $h<18$, the possible choice of 
$R(N)$ is $D_4^6, D_5^2A_7^2, D_6A_9^2, D_6^4, E_6^4, A_{11}D_7E_6, D_8^3$, \\ $D_9A_{15}$. 
However, none of them contains $D_4A_5^3A_1^3$, which is a contradiction.

\subsubsection{The case $\tau=2C$}
By the previous lemma, $V_1=B_2^6, C_4A_4^2, B_3^4,B_4^3, F_4A_8, B_6^2, B_{12}$. \\
Since $\tau=2C$, we have $\ell=4$. 
Since $v+v^{\tau}\in \Lambda^{\tau}$ and 
$\{\frac{v+v^{\tau}}{2} \mid v\in \Lambda\}=\pi(\Lambda)=(\Lambda^{\tau})^{\ast}$ 
and $\sqrt{\ell}\phi( (\Lambda^{\tau})^{\ast})=\Lambda^\tau$, 
we have $2(\Lambda^{\tau})^{\ast}=\Lambda^\tau$, that is, we can take $\phi=1$ and 
$\tilde{\beta}=2\beta$. 

By the reverse construction, there is a Niemeier lattice $P$ such that 
\[
V^{[N_0\alpha]}=(V_{\Lambda})^{[(\widehat{\tau}\exp(2\pi i\beta(0)))^{2}]}=V_P. 
\]
Since $N_0\not=1$, $P\not=\Lambda$. 
Set $\widehat{\tau}^2=\exp(2\pi i\delta(0))$, which has order $2$.  
Therefore, 
\[
V_P=(V_{\Lambda})^{[\tilde{g}^2]}
=(V_{\Lambda})^{[\exp(2\pi i(\delta(0)+2\beta(0))]}. 
\]
Since $N_0$ is odd, 
$(\widehat{\tau}\exp(2\pi i\beta(0)))^{2N_0}=\exp(\delta+2N_0\beta)=1$, 
there is a $\mu\in \Lambda^{\tau}$ such that $\delta=2N_0\beta+\mu$ and so 
$V_P=(V_{\Lambda})^{[\exp(2\pi i (N_0+1)\tilde{\beta}(0)]}$.  Note also that 
$V= V_P^{[\widetilde{g}]}$ and $\widetilde{g}$ acts on $V_P$ with order $2$.

\noindent 
{\bf Case $V_1=B_m^{12/m}$:} 

We first note that all components have level $2$ and $V^{\exp(2\pi iN_0\alpha)}=(V_P)^{\widetilde{g}}$; thus $V_P$ contains the Lie subalgebras generated by the long roots.  Since the long roots of $B_{m}$ is of the type $D_{m}$,  $(V_P)_1$ contains a Lie subalgebra of type  $D_{m,2}^{\oplus (12/m)} $ and the fixed point sublattice $P^\tau \supseteq \sqrt{2} D_{m}^{\oplus 12/m}$.  Then $P$ must contain $D_{m,1}^{\oplus 24/m}$ as a sublattice. 
Since $N_0=2m-1$ and $|\widetilde{g}^2|=N_0$, the Coxeter number of $P$ is less than or equal to $2m-1$ and thus the root system of $P$ is $D_{m}^{\oplus 24/m}$. 
Since $2m-1$ is odd and $2\delta\in \Lambda$, there is $\mu\in \Lambda$ 
such that 
$\delta=2(2m-1)\beta+\mu$ and $4(2m-1)\beta\in \Lambda$. Thus, 
we have $V_P=(V_{\Lambda})^{[4\beta]}$, that is, 
every element $\mu \in P$ is written $\mu=\mu_1+4s \beta$ with 
$\langle\mu_1,4\beta\rangle\in \bZ$, that is, $\mu_1\in \Lambda_{4\beta}$. 
Since $\langle4\beta,2\beta\rangle\in \bZ$,  $\mu_1\in \Lambda_{2\beta}$  if $\langle \mu,2\beta\rangle\in \bZ$ (that is, $\mu\in \Lambda_{2\beta}$).

Since $V_P=V^{[\widetilde{g}]}$ can be constructed by a $\bZ_2$-orbifold construction, we have
\[
L^*  = \{ x\in (P^\tau)^* \mid \langle x, 2\beta \rangle \in \bZ\} +\bZ \beta .
\] 
by the same argument as in Theorem \ref{duality}.  

For each $B_m\subseteq V_1$, we have a fundamental root system 
$d_1,...,d_{m-1}, d_m$ with long roots $d_j$ $(j= 1, \dots, m-1)$, i.e, $\langle d_j,\alpha\rangle 
=\frac{1}{2m-1}$ for $j= 1, \dots, m-1$ and $\langle d_m,\alpha\rangle=\frac{1}{2(2m-1)}$. Without loss, we may assume $d_m=\beta$, which is a short root. Note also that $\tilde{d}_m=d_{m-1} +2d_m$ is also a long root and $\{d_1,...,d_{m-1}, \tilde{d}_m\}$ forms a fundamental root system for the long roots.  
Therefore, in $P$, we have a root sublattice of type $D_m^2$ with the fundamental  system 
$\{a_1,\dots,a_m\}$ and $\{b_1,\dots,b_m\}$,
$$\begin{array}{rclrcl}  
 a_1- \cdots - &a_{m-2}&-a_{m-1},\qquad    &b_1-\cdots - &b_{m-2}& -b_{m-1} \cr
        & | &                          &         &|&     \cr
        &a_m&                         &      &b_m&    \cr
\end{array}$$
such that  $\tau$ permutes $a_i$ and $b_i$ and 
$\frac{a_i+b_i}{2}=d_i$ for $i=1,...,m-1$ 
and $\frac{a_m+b_m}{2}=2\beta+d_{m-1}=2\beta+\frac{a_{m-1}+b_{m-1}}{2}$. 
Therefore $2\beta=\frac{a_m-a_{m-1}+b_m-b_{m-1}}{2}$ and $
\langle 2\beta, a_i \rangle\in \bZ, \langle 2\beta, b_i\rangle \in \bZ$ 
for all $i=1,...,m$, that is, $R(P)\subseteq N=\Lambda^{[2\beta]}$. 
Since $2\beta\in N$ and has norm $2$, $N$ contains $D_{2m}$ and so the 
Coxeter number of $N$ is greater than or equal to $2(2m-1)$, which contradicts 
the choice of $V$.

\medskip

\noindent
{\bf Case $V_1=C_{4,2}A_{4,2}^2$:} 

Since $V_1= C_{4,2}A_{4,2}^2$ and $\det(L)=2^{10}4^2$,  $L=\sqrt{2}(D_4+E_8)$ and $N^\tau\cong D_4+\sqrt{2}E_8$ in this case. In particular,  $N^\tau$ contains a root sublattice of type $D_4$. Therefore, the possible types for $N$ are $D_4^6$, $D_4A_5^4$, and 
$D_5^2 A_7^2$.   However, the lattices orthogonal to a sublattice of the type $D_4+\sqrt{2}E_8$  do not contain $L_A(1^{12})$. 

\medskip

\noindent
{\bf Case $V_1=F_{4,2}A_{8,2}$:} 

In this case, $n=18$ and $N_0=9$.  

Let $a_1, a_2, b_3, b_4$ be the fundamental roots for $F_4$ such that $a_1, a_2$ are short roots.
We may assume $\beta=a_1$. Since $|\tau|=2$, there is a
Niemeier lattice $P$ such that $V_P=(V_{\Lambda})^{[\tilde{g}^2]}$, which is also
equal to $V^{[9\alpha]}$. Set $s=g^9=\exp(2\pi i 9\alpha(0))$. Then $(F_{4,2}A_{8,2})^{<s>}$ contains $B_{4,2}A_{8,2}$ and the Coxeter number of $P$ is less than or equal to $N_0=9$. Therefore, the root system of $P$ is $A_8^3$. Since $V=V_P^{[\tilde{g}]}$, 
$\tau$ acts on one $A_8$ as a diagram automorphism and permutes the other two $A_8$'s.

Since $N_0=9$ is odd,  we have
\[
P=\Lambda_{4\beta}+\bZ 4\beta \quad \text{ and }\quad \tilde{\beta}=2\beta.
\]
by the similar argument as in the case $V_1=B_m^{12/m}$.  
By the choice of $\beta$, the two $A_8$'s permuted by $\tau$ is orthogonal to $\beta$; hence, 
$N=\Lambda^{[\tilde{\beta}]}=\Lambda_{2\beta}+\bZ 2\beta$ contains $A_8^2$.

Since $F_4$ is a full component, $N$ also  contains $D_4$ as a root sublattice, which
is orthogonal to the above $A_8^2$. Therefore, $N \supset D_4+A_8+A_8$. Moreover, the  Coxeter number  
of $N$ is strictly less than $n=18$ by our assumption.  Then $N$  has the type $A_9^2 D_6$. In this case, $N_\tau$, the sublattice orthogonal to  $D_4+ \sqrt{2}A_8$, does not contain a root sublattice of the type $A_1^{12}$ and we have a contradiction.

This completes the proof of Theorem \ref{deephole}. 
\qed





\section{Classification of $\tau$}\label{sec:6}
In this section, we will show that 
$\tau\in Co.0$ 
defined by $W$-elements of holomorphic VOAs of central charge 24 
are contained in 
\[
\CP_0=1A\cup 2A\cup 2C\cup 3B\cup 4C\cup 5B\cup 6E\cup 6G\cup 7B\cup 8E\cup 10F.
\]  
Combining the results in Section 3 - 5,  one can associated  a pair of $\tau\in \CP_0$ and 
a $\tau$-invariant deep hole $\tilde{\beta}$ with any holomorphic VOA $V$ 
of central charge $24$ with non-abelian weight one Lie algebra $V_1$.

Set 
$$\CP=\left\{ \tau\in Co.0\, \left | \begin{array}{l}
{}^\exists \beta\in \bQ\Lambda^{\tau} \mbox{ s.t. }
\widehat{\tau}\exp(2\pi i\beta(0)) 
\mbox{ can be realized as the reverse  }\\
\mbox{automorphism of  some orbifold construction 
given by a $W$-element} 
\end{array}\right. \right\}. $$
The main result of this section is the following. 
\begin{prop} \label{P0}
$\CP\subseteq \CP_0$. 
\end{prop}

Since $\dim ((V_{\Lambda})^{[\tilde{g}]})_1>24$,  it is easy to see that $\rank(\Lambda^\tau) \geq 4$ and  
$|\tau|\leq 15$. Therefore, $\tau$ is in one of the conjugacy classes in the following list (cf. \cite{CLM}).

\begin{longtable}{|c|rcllc|c|}
	\caption{ Conjugacy classes of $\tau\in Co.0$ with $\dim \Lambda^\tau \geq 4$} \label{Table:main} \\
	\hline
Class& frame shape& $\dim \CH_0$ & $\phi(\tau)$ & Power& $\Lambda^{\tau}$ & Check for $\CP_0$\\
\hline
$2A$ & $1^{8}2^{8}$ & $16 $ & $ 1-1/2$ & & & $\in\CP_0$ \\
$-2A$ & $2^{16}/1^{8}$ & $8$ & $1-0$ & & & Lemma \ref{2}\\
$2C$ & $2^{12}$ & $12$ & $ 1-1/4$ & & & $\in\CP_0$\\
\hline
$3B $ & $1^{6}3^{6}$ & $12 $ & $ 1-1/3$ & &  & $\in\CP_0$\\
$3C$ & $3^{9}/1^{3}$ & $6 $ & $ 1-0$ & $ $ & $3E_6^{-1}$ & Lemma \ref{2}\\
$3D$ & $3^{8}$ & $8 $ & $ 1-1/9$ & $ $ & $3E_8$ & Lemma \ref{6}\\
\hline
$-4A$ & $1^{8}4^{8}/2^{8}$ & $8 $ & $ 1-1/4$ & $-2A$ & $$ & Lemma \ref{2}\\
$4C$ & $1^{4}2^{2}4^{4}$ & $10 $ & $ 1-1/4$ & $ $ & $$ & $\in\CP_0$ \\
$-4C$ & $2^{6}4^{4}/1^{4}$ & $6 $ & $ 1-0$ & $2A$ &  & Lemma \ref{2}\\
$4D$ & $2^{4}4^{4}$ & $8$ & $ 1-1/8 $ & $2A $ & $2D_4+2D_4$ & Lemma \ref{5}\\
$4F$ & $4^{6}$ & $6 $ & $ 1-1/16 $ & $2C $ & $4I_6$ & Lemma \ref{5} \\
\hline
$5B$ & $1^{4}5^{4}$ & $8$ & $ 1-1/5$ & $ $ & $$ & $\in\CP_0$  \\
$5C$ & $5^{5}/1^{1}$ & $4$ & $1-0 $ & $ $ & $$ & Lemma \ref{2} \\
\hline
$6C$ & $1^{4}2^{1}6^{5}/3^{4}$ & $6 $ & $ 1-1/6$ & $-2A,3B $ & $$ & Lemma \ref{2}  \\
$-6C$ & $2^{5}3^{4}6^{1}/1^{4}$ & $6$ & $1-0$ & $2A, 3B $ & $6E_6^{-1}$ & Lemma \ref{0} \\
$-6D$ & $1^{5}3^{1}6^{4}/2^{4}$ & $6$ & $ 1-1/6$ & $2A, 3C $ & $3E_6^{-1}$ & Lemma \ref{2} \\
$6E$ & $1^{2}2^{2}3^{2}6^{2}$ & $8$ & $ 1-1/6$ & $ $ & $$ & $\in\CP_0$  \\
$-6E$ & $2^{4}6^{4}/1^{2}3^{2}$ & $4$ & $ 1-0$ & $ -2A$ & $$ & Lemma \ref{2} \\
$6F$ & $3^{3}6^{3}/1^{1}2^{1}$ & $4$ & $ 1-0$ & $ $ & $3D_4$ & Lemma \ref{3} \\
$6G$ & $2^{3}6^{3}$ & $6$ & $ 1-1/12$ &  &  & $\in\CP_0$  \\
$6I$ & $6^{4}$ & $4 $ & $ 1-1/36$ & $2C,3D $ & $6I_4$ & Lemma \ref{3} \\
\hline
$7B$ & $1^{3}7^{3}$ & $6 $ & $ 1-1/7$ & $ $ & $$ & $\in\CP_0$  \\ 
\hline
$8E$ & $1^{2}2^{1}4^{1}8^{2}$ & $6 $ & $ 1-1/8$ & $ $ & $$ & $\in\CP_0$ \\
\hline
$10D$ & $1^{2}2^{1}10^{3}/5^{2}$ & $4$ & $ 1-1/10$ & $-2A $ & $ $ & Lemma \ref{2} \\
$-10D$ & $2^{3}5^{2}10^{1}/1^{2}$ & $4$ & $ 1-0$ &  &  & Lemma \ref{3}  \\
$-10E$ & $1^{3}5^{1}10^{2}/2^{2}$ & $4$ & $ 1-1/10$ & $5C $ &  & Lemma \ref{2}  \\
$10F$ & $2^{2}10^{2}$ & $4$ & $ 1-1/20 $ & $ $ & $$ & $\in\CP_0$  \\
\hline
$-12E$ & $1^{2}3^{2}4^{2}12^{2}/2^{2}6^{2}$ & $4$ & $ 1-1/12$ & $-2A,-6E$ &  & Lemma \ref{2}  \\
$-12H$ & $1^{1}2^{2}3^{1}12^{2}/4^{2}$ & $4
$ & $ 1-1/12$ & $3C $ & $$ & Lemma \ref{3}  \\
$12I$ & $1^{2}4^{1}6^{2}12^{1}/3^{2}$ & $4
$ & $ 1-1/12$ & $-4C $ &  & Lemma \ref{3}  \\
$-12I$ & $2^{2}3^{2}4^{1}12^{1}/1^{2}$ & $4$ & $ 1-0$ &  &  & Lemma \ref{3}  \\
$12J$ & $2^{1}4^{1}6^{1}12^{1}$ & $4$ & $ 1-1/24$ & $2A,3B,6E$ & $$ & Lemma \ref{3} \\
\hline
$14B$ & $1^{1}2^{1}7^{1}14^{1}$ & $4 $ & $ 1-1/14$ & $2A,7B$ & & Lemma \ref{3}  \\
\hline
$15D $ & $ 1^1 3^1 5^1 15^1 $ & $ 4  $ & $1-1/15$ &  & & Lemma \ref{3}\\
\hline
\end{longtable}

First we treat the cases when $\rank(V_1)= 4$, which will eliminate many cases with $\dim \Lambda^\tau =4$. 

\begin{lmm}\label{3}
If $\rank V_1= 4$, then  $\tau \in \CP_0$.  
\end{lmm}

\pr 
Suppose $\rank(V_1) = 4$. By \cite{DM2}, we know that $\dim V_1>24$ and 
$h_j^{\vee}/k_j=(\dim V_1-24)/24$ for any component $\CG_j$; 
thus $V_1$ must be of the type $B_{4,14}, C_{4,10}, D_{4,36}$ or $G_{2,24}^{\oplus 2}$  and  
$|\tau|=14,10,6,12$, respectively. In these cases, $N_0=1$. 
If $|\tau|=14$, we have $\tau=14B$. Then $\ell=14$ and $\det(\Lambda^\tau)=  {14}^2$. In this case, $\det(L) =14^4$ and  $\det(\sqrt{14} L^*)=1$, which is not possible since $\sqrt{14} L^*$  is even (see Section \ref{S:4.2}). 
If $|\tau|=10$, then $V_1=C_{4,10}$ and $N_0=1$ and so 
	there are $\mu\in L^{\ast}$ such that 
	$\langle \alpha,\mu\rangle\equiv 1/10$ and $1/5$ $\pmod{\bZ}$.  
	Therefore, since we may choose $\beta$ from $L^{\ast}$, 
	both of $\tau$ and $\tau^2$ are of type zero and $\phi(\tau)<1$ and $\phi(\tau^2)<1$, which implies 
	$\tau=10F$ and $\tau\in \CP_0$. 
If $|\tau|=6$, $V_1=D_{4,36}$ and so there is $\mu\in L^{\ast}$ such that $\langle \tau,\mu\rangle\equiv 1/6 
\pmod{\bZ}$ and $\phi(\tau)<1$. Therefore, $\tau=6I$ and $N_0=1$. 
Since $6\beta\in \Lambda^{\tau}$ which has the Gram matrix $6I_4$,  
$\langle \beta,\beta\rangle\in 2\bZ/12$. However, $\phi(6I)=1/36$; thus $\tilde{g}=\widehat{\tau}\exp(2\pi i\beta(0))$ is not type zero  and 
we have a contradiction.
\prend

From now on, we may assume $\dim V^{\langle \tau \rangle} > 4$. 

\begin{lmm}\label{6}
$\tau\not\in 3D$. 
\end{lmm}

\pr  
Recall that  $\Lambda^{3D}\cong \sqrt{3}E_8$ (cf. \cite{HaLa}). If $\GCD(N_0,3)=1$, then 
$3N_0\beta\in \Lambda^{3D}$ and so $\langle \beta,\beta\rangle 
\in 6\bZ/9N_0^2=2\bZ/3N_0$, which contradicts 
$\phi(3D)+\frac{\langle \beta,\beta\rangle}{2}
\in 1-1/9+\frac{2\bZ}{3N_0}\subset \bZ$. 
Therefore, $3|N_0$. Since $\dim \bC\Lambda^{3D}=8$,  
$N_0|h^{\vee}_j$ for all $j$ and $|g|=N_0|\tau|=\LCM(\{r_jh_j^\vee\})$, the only possible case is $N_0=3$ and 
$V_1=A_8$. Since $K_0/N_0=\langle\alpha,\alpha\rangle/2= \dim V_1/(\dim V_1-24)$, we have $K_0/3= 80/56=10/7$, which is not possible.  
\prend

Suppose Proposition \ref{P0} is false and let $V$ be a counterexample such that $k=\rank(V_1)$ is maximal among all counterexamples. 
We also choose $V$ so that 
$\Comm(M(\CH),V)$ is largest among all counterexamples with $\rank(V_1)=k$.   

Let $\alpha$ be a $W$-element of $V_1$ and set 
$g=\exp(2\pi i\alpha(0))$. As we discussed, 
an orbifold construction from $V$ by $g$ gives $V_{\Lambda}$.    
Let $\tilde{g}=\widehat{\tau}\exp(2\pi i\beta(0))$ be 
the reverse automorphism of $g$ corresponding to the above orbifold construction. 
Since $V$ is a counterexample, $\tau\not\in \CP_0$. 
 
If $m$ divides $|\tau|$ and $m\not=1$, then  we have 
\[ \rank(V^{[g^m]}_1)>\rank V_1\quad  \text{ or } \quad \Comm(M(\CH),V^{[g^m]})> \Comm(M(\CH),V)
\]
by Proposition \ref{Comm}. 
Therefore, $V^{[g^m]}$ is not a counterexample for any 
$1\not=m \mid |\tau|$. 
Let $L$ be an even lattice such that  $\Comm(\Comm(M(\CH),V),V)\cong V_L$, where 
$\CH$ is a Cartan subalgebra of $V_1=\mathfrak{g}_{1, k_1} \oplus \cdots \oplus \mathfrak{g}_{r,k_r}$.
Let  $\alpha=\sum_{i=1}^{r} \rho_i/h_i^\vee$  be a W-element of $V_1$.   
Suppose  $\langle \alpha,\alpha\rangle=\frac{2K_0}{N_0}$ with 
$K_0,N_0\in \bZ$ and $\GCD(K_0,N_0)=1$.

\medskip

Now consider the VOA $V^{[g^{N_0}]}$ obtained by an orbifold construction 
from $V$ by $g^{N_0}=\exp(2\pi iN_0\alpha(0))$. 
Then $V^{[g^{N_0}]}$ also contains an abelian Lie subalgebra 
$\CH+\oplus_{j=0}^{|\tau|-1}U(jN_0\alpha)_1$ of rank 24. By viewing 
it as a Cartan subalgebra of $(V^{[g^{N_0}]})_1$,  
there is an even unimodular lattice $P$ of rank $24$ such that 
$V^{[g^{N_0}]} \cong V_P$. 
Furthermore, 
there is $\mu\in \bC\Lambda^{\tau}$ such that 
$(V_{\Lambda})^{[\exp(2\pi i\mu(0))]}=V_P$ since $(V^{[g^{N_0}]})^{[g]}=V_{\Lambda}$. 
In particular,  $\tau$ fixes $\mu$ and we may view $\tau$ as an isometry of $P$.

\begin{remark}  \label{Ptau}
Let $\tilde{g}'\in \Aut(V_P)$ be the reverse automorphism of $g^{N_0}\in \Aut(V)$. By our choice, $\CH+\oplus_{j=0}^{|\tau|-1}U(jN_0\alpha)_1$ 
is  a Cartan subalgebra of $V_P$ and thus, there is a $\delta\in 
\bC P^{\tau}$ and a standard lift $\widehat{\tau}'$ of $\tau$ 
such that $\tilde{g}'=\widehat{\tau}'\exp(2\pi i\delta(0))$ and $\langle N_0\alpha,\delta\rangle\equiv 1/|\tau| \pmod{\bZ}$. 
Then  $V=(V_P)^{[\widehat{\tau}'\exp(2\pi i\delta(0))]}$ and the actions of 
	$\widehat{\tau}$ and $\widehat{\tau}'$ on $U(jN_0\alpha)$ are 
	same; namely, they both act on $U(jN_0\alpha)$ as a multiple by a scalar 
	$e^{-2\pi ij/|\tau|}$.  
\end{remark}

Since $N_0\alpha\in L^{\ast}$,  $V^{<g^{N_0}>}$ contains $V_L$ and $L\subseteq 
\CH=\bC\Lambda^{\tau}$. We may also view $V_L$ as a subspace of $V_P^{<\widehat{\tau}' \exp(2\pi i\delta(0))>}$.
As a conclusion, we have the following.  

\begin{thm}\label{N_0} 
The VOA $V^{[e^{2\pi i N_0\alpha(0)}]}$ is isomorphic to a lattice VOA $V_P$  
	for an even unimodular lattice $P$ of rank $24$. 
	Furthermore, $P^{\tau}=L+\bZ N_0\alpha$ and $|g^{N_0}|= |\tau|$.
\end{thm}


As we have shown in Theorem \ref{N_0}, 
$V^{[\exp(2\pi iN_0\alpha(0)]}$ is a Niemeier lattice VOA $V_P$ and  
there is $\delta \in \bC P^{\tau}$ such that $\widehat{\tau}\exp(2\pi i\delta(0))$ is the reverse automorphism of $\exp(2\pi iN_0\alpha(0))$. 



\begin{lmm}
By taking a suitable standard lift $\widehat{\tau}$ of $\tau$, 
we may assume $\delta\in L^{\ast}$. 
\end{lmm}

\pr By the definition of W-element, 
there is $\mu\in L^{\ast}$ such that $\langle \alpha,\mu\rangle 
\equiv 1/N_0|\tau| \pmod{\bZ}$. 
Therefore, $\delta-\mu\in (P^{\tau})^{\ast}$ and 
there is $\xi\in P$ such that $\pi(\xi)=\delta-\mu$. 
Set $\xi=\xi'+\pi(\xi)$ with $\xi'\in \CH^{\perp}$
Since $\exp(2\pi i\xi(0))=1$ and $\tau$ acts on 
$(\CH)^{\perp}$ fixed point freely, there is $\tilde{h}\in 
\CH^{\perp}$ such that 
$$\begin{array}{l}
\exp(2\pi i\tilde{h}(0))^{-1}(\widehat{\tau}\exp(2\pi i\delta(0))\exp(2\pi i\tilde{h}(0)
=\exp(2\pi i\tilde{h}(0))^{-1}\widehat{\tau} 
\exp(2\pi i\tilde{h}(0))\exp(2\pi i\delta(0))\cr
=\widehat{\tau}\exp(2\pi i\xi'(0))\exp(2\pi i\delta(0))
=\widehat{\tau}\exp(2\pi i\xi(0))\exp(-2\pi i\pi(\xi)(0))\exp(2\pi i\delta(0))\cr
=\widehat{\tau}\exp(2\pi i\mu(0)).
\end{array}$$ 
Therefore, 
$\widehat{\tau}\exp(2\pi i\delta(0))=\widehat{\tau}
\exp(\mu(0))\exp(\pi \xi(0))$ 
is conjugate to $\widehat{\tau}\exp(2\pi i\mu(0)$. 
Replacing $\widehat{\tau}$ by $\exp(2\pi i\tilde{h}(0))^{-1}(\widehat{\tau}\exp(2\pi i\delta(0))\exp(2\pi i\tilde{h}(0))$, 
we may assume $\delta\in L^{\ast}$. \prend


By Theorem \ref{N_0}, we have $P^{\tau}=L+\bZ N_0\alpha$ and $rN_0\alpha\in L$ if and only if 
$|\tau|$ divides $r$ and $\langle \delta, N_0\alpha\rangle\equiv 
1/|\tau| \pmod{\bZ}$. Therefore,  
$m\delta\not\in\pi(P)$ if $|\tau|\not| m$. 
Hence if $|\tau|\not| m$, then the lowest weight of 
$\widehat{\tau}\exp(2\pi i\delta(0))$-twisted module 
is greater than $\phi(\tau)$. 
Therefore, we have:  

\begin{lmm}\label{twist}
If $\phi(\tau^m)\geq 1$, then $\widehat{\tau}^m\exp(2\pi im\delta(0))$-twisted module does not contain a weight one element. 
\end{lmm} 


\begin{lmm}\label{2}
$<\tau>$ does not contain $-2A, 3C, 5C$. 
\end{lmm}

\pr 
Let $m$ be a divisor of $|\tau|$. Suppose $|\tau^m|=p$ is a prime and 
$\phi(\tau^m)\geq 1$. 

Let  
$\tilde{V}:=V^{[g^{pN_0}]}$. From the property of orbifold construction and the choice of $\widehat{\tau}'$ and $\delta$, 
we have $\tilde{V}=(V_P)^{[(\widehat{\tau}')^m\exp(2\pi im\delta(0))]}$.  
From Corollary 4.9 in \cite{CLM}, 
$\tau\not=-2A, 3C, 5C$ and so $m\not=1$ and 
$\tilde{V}$ is not a counterexample. 
 
Let $\tilde{\alpha}$ be a W-element of $\tilde{V}_1$ and 
$\widehat{\sigma}\exp(2\pi i\mu(0))$ the reverse automorphism of  
$\exp(2\pi i\tilde{\alpha}(0))$. Then $\sigma\in \CP_0$ and we can use the results in \S 3 $\sim$ 5.

Next we will show that $\bC P^{\sigma}=\bC P^{\tau^m}$. 
As we explained in Section \ref{S:2}, 
$\CH+\oplus_{j=0}^{m-1}U(jpN_0\alpha)$ is a 
Cartan subalgebra of $\tilde{V}_1$. 
Since $\langle \delta,\alpha\rangle\equiv 1/|\tau| \pmod{\bZ}$, 
$\widehat{\tau'}$ acts on $U(jN_0\alpha)$ as a multiple by 
a scalar $e^{-2\pi ij/|\tau|}$ for all $j\in \bZ$. 
In particular, 
 $\bC P^{\tau^m}=\CH+\oplus_{j=0}^{|\tau|-1}U(jpN_0\alpha)_1$ and so $\bC P^{\sigma}=\bC P^{\tau^m}$. 
 
 We next show that if the order $p$ of $\tau^m$ is an odd prime, we have $|\widehat{\sigma}|=p$  
and if the order of $\tau^m$ is $2$ we have $|\widehat{\sigma}|=2$ or $4$. 

Since $\phi(\tau^m)\geq 1$, 
$\widehat{\tau'}^m\exp(2\pi im\delta(0))$-twisted module 
contains no weight one element and so we have 
$$\tilde{V}_1=
(V_P)^{<\widehat{\tau'}^m\exp(2\pi im\delta(0))>}_1.$$

Since $|\tau^m|=p$ is a prime,  the root vectors of 
$(V_P)^{<\widehat{\tau'}^m\exp(2\pi im\delta(0))>}_1$  are given by 

(1) $e^{\mu}$ with $\mu\in P^{\tau}$ or

(2) $e^{\mu}+\cdots +(\widehat{\tau'}\exp(2\pi im\delta(0)))^{p-1}(e^{\mu})$ with $\mu\in P$.  In this case, it corresponds to the root associated with $\pi(\mu)\in (P^\tau)^*$ and $\langle \pi(\mu), \pi(\mu) \rangle =2/\tilde{r}_j\tilde{k}_j$.  Note also that 
\[
\langle \pi(\mu), \pi(\mu) \rangle = \langle \pi(\mu), \frac{1}{p} \sum_{j=0}^{p-1} \tau^{mj} \mu \rangle =
\langle \mu, \frac{1}{p} \sum_{j=0}^{p-1} \tau^{mj} \mu \rangle \in \bZ/p 
\]
\noindent  
Therefore, if $p$ is an odd prime, then $\tilde{r}_j\tilde{k}_j=1$ or $p$ 
and if $p=2$, then $\tilde{r}_j\tilde{k}_j=1,2,4$.  

On the other hand, since $\sigma\in \CP_0$, we have shown  
$|\widehat{\sigma}|={\rm LCM}(\{\tilde{r}_j\tilde{k}_j\mid j\})$. 
If $\tilde{r}_j\tilde{k}_j=1$ for all $j$, then  we can easily get a contradiction since 
$\rank(\widetilde{V}_1)=8, 6, 4$, respectively.  

Hence if the order $p$ of $\tau^m$ is an odd prime, we have $|\widehat{\sigma}|=p$  
and if the order of $\tau^m$ is $2$ we have $|\widehat{\sigma}|=2$ or $4$. 
From the choice of $\sigma$, we have 
$\rank(\bC\Lambda^{\sigma})=\rank(\bC\Lambda^{\tau^m})$.  
Therefore, if $\tau^m=-2A, then \dim \Lambda^{-2A}=8$, but 
there is no $\sigma\in \CP_0$ of order $2$ or $4$ such that 
$\Lambda^{\sigma}=8$. 
If $\tau^m=3C$, then $\rank \Lambda^{3C}=6$, but there is no $\sigma\in \CP_0$ of order $3$ with 
$\rank \Lambda^{\sigma}=6$. 
If $\tau=5C$, then $\rank \Lambda^{5C}=4$, but there is no $\sigma\in \CP_0$ of order $5$ with 
$\rank \Lambda^{\sigma}=4$.  
\prend

\begin{lmm} \label{2.5}
$\tau\not\in -4C$.
\end{lmm}
\pr 
Suppose $\tau\in -4C$. 
We note $V=(V_P)^{[\widehat{\tau}\exp(2\pi i\delta(0))]}$. 
Since $\phi(\tau)\geq 1$, 
$\widehat{\tau}\exp(2\pi i\delta(0))$-twisted module does not contain elements of weight one. 
Recall that $\widehat{\tau}\exp(2\pi i\delta(0))$ is the reverse automorphism of 
$\exp(2\pi iN_0\alpha(0))$; the above statement means that  
there is no root $\mu\in V_1$ such that $|\tau|\langle \mu,N_0\alpha\rangle=1$. 
Since $|\tau|=4$, $\LCM(\{r_jh^{\vee}_j/N_0\mid j\})$ divides $2$, which 
contradicts $\LCM(\{ r_jh^{\vee}_j/N_0 \mid j\})=|\tau|$.
\prend

Therefore the remaining possibility of $\tau\not\in \CP_0$ are 
$4D, 4F, -6C$. Then since 
$4D^2=2A, 4F^2=2C, (-6C)^2=3B, (-6C)^3=2A$, 
non-trivial powers of $\tau$ are in $\CP_0$.  

Furthermore, $\phi(4D)=1-1/8$, $\phi(4F)=1/16$, 
$\phi(-6C)=1$. Namely, 
$\widehat{\tau}$ is not of type zero or $\phi(\tau)=1$. 

We note that since $V=V_{\Lambda}^{[\widehat{\tau}\exp(2\pi i\beta(0))]}$, 
$\tilde{V}=V^{[\widetilde{g}^m]}$ is a holomorphic VOA for 
any $m||\tau|$. 
Assume $m\not=1$. Unfortunately, $\widetilde{g}^m$ 
is not necessary to be a reverse automorphism for an automorphism 
defined by W-element in $\tilde{V}$. 
Let $\tilde{\alpha}$ be a W-element of $\tilde{V}_1$ and 
$\widehat{\sigma}\exp(2\pi i\mu(0))$ be a reverse automorphism 
of $\exp(2\pi i\tilde{\alpha}(0))$.

\begin{lmm}\label{conj2a2c3b}
Let $1\not=m||\tau|$. 
If $\tau^m=2A,2C,3B$, then $\sigma$ is conjugate to $\tau^m$. 
\end{lmm}

\pr 
If $\tau^m=2A,2C,3B$, then $\tilde{V}:=V_P^{[\widehat{\tau}^m\exp(2\pi im\delta(0))]}$ 
is not a counterexample. 
Let $\tilde{\alpha}$ be a W-element of $\tilde{V}_1$ and 
$\widehat{\sigma}\exp(2\pi i\mu(0))$ a reverse automorphism for 
$\exp(2\pi i\tilde{\alpha}(0))$. In particular, $\sigma\in \CP_0$.

As we have shown,  $\rank(\bC\Lambda)^{\tau^m}=\rank \tilde{V}_1=\rank
(\bC\Lambda)^{\sigma}$. 
Let $\tilde{\CH}$ be a Cartan subalgebra of $\tilde{V}$ and 
set $\Comm(\Comm(\tilde{\CH},\tilde{V}),\tilde{V})=V_{\tilde{L}}$ 
with an even lattice $\tilde{L}$. 
 
Since $\sigma\in \CP_0$, 
$\sqrt{|\widehat{\sigma}|}(\tilde{L})^{\ast}$ 
is an even lattice by Lemma 5.1. 
On the other hand, since $\tau^m\in \CP_0$ and 
$\tilde{V}=(V_P)^{[\widehat{\tau}^m\exp(2m\pi i\delta(0))]}$, 
we have $\phi(\tau^m)=1-\frac{1}{|\widehat{\tau}^m|}$ and 
$e^{m\delta}\otimes t\in \tilde{V}$ with $\wt(e^{m\delta})\equiv 
\frac{1}{|\widehat{\tau}^m|} \pmod{\bZ}$ from an orbifold construction 
by an automorphism with positive frame shape. 
Therefore  $|\widehat{\tau}^m|$ have to divide $|\widehat{\sigma}|$. 

If $\tau^m=2A$, then $\rank(\bC\Lambda)^{\tau^m}=16$ and 
$\sigma=2A$. 
If $\tau^m=2C$, then $\rank(\bC\Lambda)^{\sigma}=12$ 
and so $\sigma=2C$ or $3B$. If $\sigma=3B$, then $|\widehat{\sigma}|=3$ 
is not a multiple of $|\widehat{\tau^m}|=2$. 
If $\tau^m=3B$, then $\rank(\bC\Lambda^{\sigma})=12$ and so 
$\sigma=2C$ or $3B$.  
Since $|\widehat{\tau}^m|=3$ does not divide $|\widehat{\sigma}|=4$ if 
$\sigma=2C$. Hence $\sigma=3B$. 
\prend

As a corollary, for $\mu\in L^{\ast}$, 
if $\langle \mu, N_0\alpha\rangle 
\in \frac{\bZ}{m}$, then 
$S(\mu)\otimes e^{\mu} \in V^{[|\tau|g/m]}$ and so we have 
$|\widehat{\tau}^m|\langle \mu,\mu\rangle\in 2\bZ$.

\begin{lmm}\label{0} 
$\tau\not=-6C$. 
\end{lmm}

\pr 
Since $\phi(-6C)=1$, $\widehat{\tau}\exp(2\pi i\delta(0))$-twisted 
module does not contain weight one elements. 
Therefore, $r_jh^{\vee}_j/N_0=2$ or $3$ and the both have to appear. 
Recall that $(-6G)^3=2A$ and $(-6G)^2=3B$. 
If $ r_jh^{\vee}_j/N_0=2$, then $\exp(2\pi iN_0\alpha(0))^3$ fixes 
all elements in $\CG_j$ and so 
$\sqrt{3}L_j^{\ast}$ is an even lattice since $2A\in \CP_0$, where 
$L_j^{\ast}$ is the co-root lattice of $\CG_j$. 
If $r_jh^{\vee}_j/N_0=3$, then $\exp(2\pi iN_0\alpha(0))^2$ fixes 
all elements in $\CG_j$ and so 
$\sqrt{2}L_j^{\ast}$ is an even lattice. 
Therefore $\sqrt{6}L^{\ast}$ is an even lattice and 
$K_0-N_0=1$ and $h_j^{\vee}=k_jN_0$. 
If $N_0=1$, then $r_jh^{\vee}_j=2$ and $3$ and so 
$\CG_j=A_1$ or $A_2$, which contradicts 
to $\dim V_1>24$ since $\rank V_1=6$. 
If $N_0=2$, $r_jh^{\vee}_j=4$ and $6$ and so 
$\CG_1=A_3$ and $\CG_2=A_5$, which contradicts $\rank V_1=6$. 
If $N_0=3$, then one of $r_2h_2^{\vee}=9$ and so $\CG_2=A_8$, which  contradict $\rank V_1=6$. 
If $N_0=4$, then one of $r_1h_1^{\vee}=8$ and $\CG_2=G_2$. 
So, $\CG_1=A_7,C_3,D_5$.  Since $\rank V_1=6$ and $\CG_j\not=A_1$, 
we have a contradiction.  
If $N_0>4$, $r_1h_2^{\vee}\geq 10$ and $r_2h_2^{\vee}\geq 15$, 
which contradicts $\rank V_1=6$. 
 \prend
 
So the remaining cases are $\tau=4D$ or $4F$. 
Recall that  a short root in $\CG_j$ is called a shortest root if 
$r_jh_j^\vee= \LCM(r_ih_i^{\vee}|i)=  |\tau| N_0$ (or equivalently  $r_jk_j=\LCM(r_ik_i|i)$). 
Since $|\tau|=4$ and $\LCM(r_ih_i^{\vee}/N_0\mid i)=|\tau|=4$,
there is a shortest root. Say, $\CG_1$ contains a shortest root.

\begin{lmm}\label{lmm612}
If $\tau=4D$ or $4F$, then $\exp(2N_0\pi i\alpha(0))^2$ fixes a root vector associated with a shortest root.  
\end{lmm}

\pr  Since $\LCM(\{r_jh_j^{\vee}/N_0\mid i\}=|\tau|$ and $h_j^\vee$ is divisible by $N_0$, $r_j$ divides $|\tau|$. Therefore, $\CG_j\not=G_2$.  

If $\CG_1\not\cong B_m$, then there is a shortest root which is a sum of two shortest roots in the fundamental root system and thus $\exp(2N_0\pi i\alpha(0))^2$ fixes a root vector associated with a shortest root in $\CG_1$.  So we may assume $\CG_1=B_m$; however, $r_1h_1^{\vee}=2(2m-1)$ in this case, which contradicts that $|\tau|=4$ divides $r_1h_1^\vee$.   
\prend

\begin{lmm}\label{5}
$\tau\not=4D$ nor $4F$.
\end{lmm}

\pr 
Suppose false.  Set $\tilde{V}=V^{[\exp(4N_0\pi i\alpha(0))]}$. 
We may view that $\tau$ is an automorphism of $\tilde{V}$ of order $2$.
Let $\tilde{\CH}$ be a Cartan subalgebra of $\tilde{V}_1$. Then there is an even lattice $\tilde{L}$ such that 
$\Comm(\Comm(M(\tilde{\CH}),\tilde{V}),\tilde{V})\cong V_{\tilde{L}}$. 
Let $\tilde{\alpha}$ be a W-element of $\tilde{V}_1$ and $\widehat{\sigma}\exp(2\pi i\xi(0))$ the reverse automorphism of $\exp(2\pi i\tilde{\alpha}(0))$.  
Since $4D^2=2A$ and $4F^2=2C$, we have $\sigma=2A$ and $2C$, respectively by the same argument as in the proof of Lemma \ref{conj2a2c3b}. 

Since $\tilde{V}^{[\exp(2\pi i\alpha(0)]}=V_{\Lambda}$, 
there is no root orthogonal to $\alpha$ in $\tilde{V}$. 
Therefore, we can define positive roots by the condition $\langle \mu,\alpha\rangle>0$.  
Note that $\alpha$ is fixed by $\tau$. Therefore, if $\mu$ is a positive root, then so is $\tau^i(\mu)$ for any $i$. By Lemma \ref{lmm612}, there is a shortest root $\mu\in L^{\ast}$  
such that 
$s\otimes e^{\mu}\in 
V^{<\exp(2\pi iN_0\alpha(0))^2>}\subseteq 
\tilde{V}=V^{[\exp(4\pi i N_0\alpha(0))]}$. 

Since $\tau$ has positive frame shape, we may choose $\delta$ in $L^{\ast}$. 
Note that $\phi(\tau)=1-\frac{1}{2\ell}$ for $\tau\in 4C$ or $4D$. If $\langle \delta, \delta\rangle/2\in \frac{1}{\ell}\bZ$,   then  the irreducible $\widehat{\tau}\exp(2\pi i\delta(0))$-twisted module has no subspaces with integral weights, which is not possible. 
Therefore, it suffices to show that 
$\sqrt{\ell}L^{\ast}$ is an even lattice or equivalently, $\ell \langle \mu,\mu\rangle\in 2\bZ$ since $\mu$ is a shortest root.  
Without loss, we may assume $\langle \mu,\alpha\rangle>0$. 
Then we have one of the following two cases:\\
(1) $\mu$ is still a root in $\tilde{V}=V^{[\exp(4N_0\pi i\alpha(0))]}$ or \\
(2) $s\otimes e^{\mu}=s'\otimes e^{\mu'}+\tau(s')\otimes e^{\tau(\mu')}$ 
with a positive root $\mu'\in \bC\Lambda^{<\tau^2>}$ and  
$\mu=\frac{\mu'+\tau(\mu')}{2}$.  In particular, $\mu'+\tau(\mu')$ is not a root. 
Note that $\mu'$ and $\tau(\mu')$ are both positive roots of the same length and hence
$\frac{2\langle \mu',\tau(\mu')\rangle}{\langle \mu',\mu'\rangle} = 0, -1$ or $2$ but $-1$ is not possible since $\mu'+\tau(\mu')$ is not a root. We have $t=\frac{\langle \mu',\tau(\mu')\rangle}{\langle \mu',\mu'\rangle}\in \bZ$.

By Lemma \ref{conj2a2c3b},  $\tau^2$ is conjugate to $\sigma$ and thus  
$|\widehat{\tau}|=2|\widehat{\sigma}|$.  
Since $\tilde{V}$ is not a counterexample, 
$\sigma\in \CP_0$ and so 
$|\widehat{\sigma}|\langle \mu,\mu\rangle\in 2\bZ$ for the case (1) and 
$|\widehat{\sigma}|\langle \mu',\mu'\rangle\in 2\bZ$ for the case (2). 
For the case (1), 
$\ell \langle \mu,\mu\rangle\in 4\bZ$. 
For the case (2), 
$|\widehat{\sigma}|\langle \tau(\mu'),\mu'\rangle=t|\widehat{\sigma}|\langle \mu',\mu'\rangle\in 2\bZ$. 
Then we have 
$\ell \langle \mu,\mu\rangle=
2|\widehat{\sigma}| 
\langle \frac{\mu'+\tau(\mu')}{2}, \frac{\mu'+\tau(\mu')}{2}
\rangle
=|\widehat{\sigma}|\langle \mu'+\tau(\mu'), \mu'\rangle
\in 2\bZ$ as desired. 
\prend

This completes the proof of $\CP\subseteq \CP_0$.

\section{Reverse construction}\label{S:7}
As we have already shown, one can define a pair of $\tau\in \CP_0$ and 
a $\tau$-invariant deep hole $\tilde{\beta}$ for any holomorphic VOA $V$ 
of central charge $24$ with non-abelian $V_1$ by choosing a $W$-element $\alpha$ of $V_1$ and considering the reverse automorphism $\tilde{g}\in \Aut(\Lambda)=Co.0$ of $g=\exp(2\pi i\alpha(0))\in \Aut(V)$.  The pair $(\tau, \tilde{\beta})$ satisfies the following conditions:
\begin{itemize}
	\item[(C1)] $\tau\in \CP_0$ and $\tilde{\beta}$ is a $\tau$-invariant deep hole  of 
	Leech lattice $\Lambda$ with $\langle \tilde{\beta},\tilde{\beta}\rangle=2$; 
	\item[(C2)] the Coxeter number $h$ of $N=\Lambda_{\tilde{\beta}}+\bZ \tilde{\beta}$ is divisible by $|\tau|$; 
	\item[(C3)] $N_{\tau}=L_A(c_{\tau})$ with $c_{\tau}$ as defined in Appendix,
\end{itemize}
where $\Lambda_{\tilde{\beta}}=\{\lambda\in \Lambda \mid \langle \lambda, \tilde{\beta}\rangle\in \bZ\}$  and 
$N_{\tau}$ denotes an orthogonal complement of $N^{\tau}$ in $N$. We also use the same notation $\tau$ to denote the isometry of $N$ induced from $\tau$ on $\Lambda$. 

Let $\CT$ be the set of pairs satisfying the conditions  (C1) to (C3).  In this section, we will study the reverse construction.  Take a pair $(\tau, \tilde{\beta})\in \CT$. Assume that $\tau$  has the frame shape $\prod m^{a_m}$. Then $N_\tau > \oplus_{m\neq 1} A_{m-1}^{a_m}$ and the conformal weight of the $\widehat{\tau}$-twisted module of $V_\Lambda$ is given by $\phi(\tau)=1-\frac{1}{\ell}$, where $\ell=|\widehat{\tau}|$.   

Let $\beta=\frac{1}{\sqrt{\ell}}\varphi^{-1}(\tilde{\beta})$ and define $\tilde{g}=\widehat{\tau}\exp(2\pi i\beta(0))\in \Aut(V_{\Lambda})$. 
Note that $\langle \beta,\beta\rangle=\frac{2}{\ell}$.  
Set $N=\Lambda_{\tilde{\beta}}+\bZ \tilde{\beta}$ and let $h$ be the 
Coxeter number of $N$.  
Since $-\tilde{\beta}$ is a deep hole of $\Lambda$, 
$N^{(1)}$ contains an affine fundamental root system of rank $24$ and 
$N=\oplus_{k=0}^{h-1} N^{(k)}$, where $N^{(k)}=\Lambda_{\tilde{\beta}}-k\tilde{\beta}$.  

We may assume that $ \hat{\tau}^{|\tau|} =\exp(2\pi i \delta(0))$ for some $\delta$ fixed by $\tau$ and take $\delta=0$ if  $\hat{\tau}^{|\tau|} =1$.  
For $i\equiv k \mod 2|\tau|$, the irreducible $\tilde{g}^k$-twisted modules are as follows: 
\[
T^k =V_\Lambda[\tilde{g}^k]= 
\begin{cases}
\bC[-k\beta +(\Lambda^{\tau^i})^*]\otimes M(1)[\tau^i]\otimes T_{\tau^i} & \text{ if }   0  \leq  i < |\tau|, \\
\bC[-k\beta + \delta +(\Lambda^{\tau^i})^*]\otimes M(1)[\tau^i]\otimes T_{\tau^i} & \text{ if }   |\tau|  \leq  i < 2|\tau| 
 \\
\end{cases}
\] 
 (see \cite[Propositions 6.1 and 6.2]{Le} and \cite[Remark 4.2]{DL} for detail). 
In particular, $T^1$ contains a weight one element $e^{\beta}\otimes t$ with $t\in T_\tau$, 
which is the lowest weight element of $T^1$, and 
so $T^1_{\bZ}\not=0$ and the orbifold construction gives a holomorphic VOA $V:=V^{[\tilde{g}]}$ of central charge $24$ 
and $V_1$ is non-abelian since 
$V_1$ contains a root vector $e^{\beta}\otimes t$.  

One of the main purposes  of this section is to determine an equivalent relation $\sim $ on $\CT$ so that two pairs $(\tau, \tilde{\beta})$ and $(\tau', \tilde{\beta}')$ define isomorphic VOAs if and only if $(\tau, \tilde{\beta})\sim (\tau', \tilde{\beta}')$.  
It is clear that $\tau$ and $\tau'$ are conjugates in $O(\Lambda)$ and $\tilde{\beta}$ and $\tilde{\beta}'$ are equivalent deep holes of the Leech lattice $\Lambda$ if 
$(\tau, \tilde{\beta})$ and $(\tau', \tilde{\beta}')$ define isomorphic VOAs.
 
For two equivalent deep holes $\tilde{\beta}$ and $\tilde{\beta}'$, there are $\sigma\in O(\Lambda)$ and $\lambda\in \Lambda$ such that $\tilde{\beta}'= \sigma( \tilde{\beta}-\lambda)$.
Since $\tilde{\beta}'$ is $\tau'$-invariant, $\tilde{\beta}-\lambda$ is  $\sigma^{-1} \tau' \sigma$-invariant. Moreover, $\langle \tilde{\beta}, \lambda \rangle \in \bZ$ since $\langle \tilde{\beta}, \tilde{\beta}\rangle = \langle \tilde{\beta}', \tilde{\beta}'\rangle =2$. 
In this case,  $\Lambda_{\tilde{\beta}} +\tilde{\beta} = \Lambda_{\tilde{\beta}-\lambda} +(\tilde{\beta}-\lambda)$. Therefore, up to the action of $O(\Lambda)$, we may identify 
$N= \Lambda_{\tilde{\beta}} +\bZ \tilde{\beta}$ with $N' =\Lambda_{\tilde{\beta}'} +\bZ \tilde{\beta}'$.    We define an equivalent relation on $\CT$ as follows: 

$(\tau,\tilde{\beta})\sim (\tau',\tilde{\beta}')$ if and only if \\
(1)  
$\tilde{\beta}$ and $\tilde{\beta}'$ are equivalent deep holes of the Leech lattice $\Lambda$, i.e., there are $\sigma\in O(\Lambda)$ and $\lambda\in \Lambda$ such that $\tilde{\beta}'= \sigma( \tilde{\beta}-\lambda)$;\\
(2) $\tau$ is conjugate to $\sigma^{-1} \tau' \sigma$ in $O(N)$.\\
Note that $ (\sigma^{-1} \tau' \sigma,\tilde{\beta}-\lambda )\in \CT$ and $N= \Lambda_{\tilde{\beta}} +\bZ \tilde{\beta}= \Lambda_{\tilde{\beta}-\lambda} +\bZ(\tilde{\beta}-\lambda)$. Moreover,  $\tau$ and $\tau'$ are conjugate in $O(\Lambda)$ since they have the same frame shape by (2).  

\medskip
We will prove the following main theorem. 

\begin{thm}\label{Main}
	There is a one-to-one correspondence between the set of isomorphism classes of holomorphic VOA $V$ of central charge $24$ having non-abelian $V_1$ and 
	the set $\CT/\sim$ of equivalence classes  of  pairs $(\tau, \tilde{\beta})$ by $\sim$. 
\end{thm}

\subsection{W-element and the automorphism $\tilde{g}$} 
Let $(\tau, \tilde{\beta})\in \CT$. Set $\beta=\frac{1}{\sqrt{\ell}}\varphi^{-1}(\tilde{\beta})$ and define $\tilde{g}=\widehat{\tau}\exp(2\pi i\beta(0))\in \Aut(V_{\Lambda})$. 
Let $V=V^{[\tilde{g}]}$ be the holomorphic VOA obtained by an orbifold construction from $V_\Lambda$ and $\tilde{g}$.  We will show that there is a $W$-element $\alpha$ of $V_1$ such that $\tilde{g}$ can be viewed as a reverse automorphism of $\exp(2\pi i \alpha(0)) \in \Aut(V)$.

First we will prove the following theorem. 

\begin{thm}\label{order}
	Let $h$ be the Coxeter number of $N=\Lambda_{\tilde{\beta}} +\bZ\tilde{\beta}$. Then 
$|\tilde{g}|=h$. 
\end{thm}

We need  several lemmas. 

\begin{lmm} \label{hdividess} 
	$s\beta\in \pi(\Lambda)$ if and only if $h|s$. 
\end{lmm}

\pr Since  $\sqrt{\ell}\varphi((\Lambda^\tau)^*)= \Lambda^\tau$, we have $s\beta\in \pi(\Lambda)=(\Lambda^{\tau})^{\ast}$ if and only if 
$s\tilde{\beta}\in \Lambda^{\tau}$ if and only if $h|s$.
\prend 

\begin{lmm}
We have $\tilde{g}^{2h}=1$. Furthermore, $\tilde{g}^h=1$ if $2C\not\in <\tau>$.   
\end{lmm}

\pr
By (C3) and $\tau\in \CP_0$, $N_{\tau}$ contains a root sublattice $S\cong A_{|\tau|-1}$.  
For a root $\mu$ of $S$, there is $\lambda\in \Lambda$ and $m\in \bZ$ such that 
$\mu=\lambda+m\tilde{\beta} \in N_{\tau}\subseteq \bC\Lambda_{\tau}$. 
Then 
\[
0=\sum_{i=0}^{|\tau|-1} \tau^i \mu = \sum_{i=0}^{|\tau|-1} \tau^i \lambda +|\tau|m\tilde{\beta} \in \Lambda^\tau. 
\] 
Since $h$ is the smallest positive integer satisfying $h\tilde{\beta}\in \Lambda^{\tau}$ and $\tau$ is fixed point free on $S$ with order $|\tau|$, we have $m=h/|\tau|$.    
Moreover, we also have 
$\frac{h}{|\tau|}\tilde{\beta}\in \pi(\Lambda)\in (\Lambda^{\tau})^{\ast}$ and so 
$\langle \frac{h\tilde{\beta}}{|\tau|}, \Lambda^{\tau}\rangle\subseteq \bZ$. 
Hence  
$\bZ \supseteq \frac{h}{|\tau|}\langle \sqrt{\ell}\varphi(\beta), 
\sqrt{\ell}\varphi((\Lambda^{\tau})^{\ast})\rangle
=\langle \frac{h\ell}{|\tau|}\beta,(\Lambda^{\tau})^{\ast}\rangle$  
and so $\frac{\ell}{|\tau|}h\beta\in \Lambda$, which means   
$(\exp(2\pi \beta(0))^{\frac{\ell h}{|\tau|}}=1$.
 In particular, if $2C\not\in <\tau>$, then $\ell=|\tau|$ and 
$\exp(2\pi i\beta(0))^h=1$.  
\prend

From now on, we may assume $\tau=2C, 6G, 10F$. 
In particular, $|\tau|=2s$ for some $s=1,3$ or $5$. Note that 
$\sigma=\tau^s$ is a $2C$-element. 

\begin{lmm} \label{odd}
Let $\tau\in2C, 6G$, or $10F$. Let $\tilde{\beta}$ be a $\tau$-invariant deep hole satisfying (C1)-(C3). Let $h$ be the Coxeter number of $N=\Lambda_{\beta} +\bZ\beta$. Then $h/|\tau|$ is odd.    
\end{lmm}

\pr By our assumptions, $N_\tau\cong L_A(c_\tau)$ as defined in Appendix; namely, 
\[
N_\tau \supset  \bigoplus_{m_i| |\tau|, m_i\neq 1}  A_{m_i-1}^{a_i}
\]
as a full rank sublattice if the frame shape of $\tau$ is $ \prod m_i^{a_i}$.  Moreover, 
\[
[ N_\tau :  \bigoplus_{m_i| |\tau|, m_i\neq 1}  A_{m_i-1}^{a_i}  ]=|\tau|\quad  \text{ and } 
\quad R(N_\tau) = \bigoplus_{m_i| |\tau|, m_i\neq 1}  A_{m_i-1}^{a_i}.
\] 
Since there are only $16$ Niemeier lattices whose Coxeter number is even, it is straightforward to determine all Niemeier lattices $N$ such that  $L_A(c_\tau)\subset N$ as a direct summand. Indeed, $R(N)=A_1^{24}, D_4^6, A_5^4D_4$, $D_6^4, A_9^2D_6$, $D_8^3$, $A_{17}E_7$, $D_{12}^2$, or $D_{24}$ if $\tau\in 2C$; $R(N)= A_5^4D_4$ or $A_{17}E_7$ if $\tau\in 6G$; $R(N)= A_9^2D_6$ if $\tau\in 10F$. 

It turns out that $h/|\tau|$ is odd for all possible cases. \prend

The following lemma can be verified easily using the definition of $A_{m-1}$. 
\begin{lmm}\label{c2C}
	Let $m=2s$ be an even integer and let $\tau$ be a Coxeter element in $\mathrm{Weyl}(A_{m-1})$. Then the $(-1)$-eigenlattice of $\tau^s$ on $A_{m-1}$ is isometric to $A_1^s= \bZ x_1\oplus \cdots \oplus \bZ x_s$. Moreover, the vector $\frac{1}2( x_1+\cdots + x_s)\in s\gamma +A_{m-1}$, where $\gamma+ A_{m-1}$ is a generator of $A_{m-1}^*/A_{m-1}$. 
\end{lmm}



\begin{lmm}
For $\tau=2C, 6G, 10F$ and $v,u\in \Lambda^{\tau}$, 
$\langle v,u\rangle\in 2\bZ$. 
\end{lmm}

\pr 
We note $\Lambda^{\tau}\subseteq \Lambda^{\sigma}$. 
Since $\ell(\sigma)=4$, $\Lambda^{\sigma}=2(\Lambda^{\sigma})^{\ast}$. 
Hence $\frac{v}{2}\in (\Lambda^{\sigma})^{\ast}$
and $\langle v,u\rangle=2\langle \frac{v}{2},u\rangle\in 2\bZ$. 
\prend

Now now on, we set $p=h/|\tau|$, which is an odd integer. 

\begin{prop} \label{mod4}
For $\tau=2C, 6G, 10F$ and $v\in (\Lambda_{\tilde{\beta}}+sp\tilde{\beta})^{\tau}$, 
we have $\langle v,v\rangle \equiv 2 \pmod{4}$. 
\end{prop}
 
\pr  We first note that  $(\Lambda_{\tilde{\beta}}+\tilde{\beta})_2$ forms an affine fundamental system of roots of $N$.  Choose a fundamental root system $X$ from 
$\Lambda_{\tilde{\beta}}+\bZ\tilde{\beta}$ such that $\tilde{\beta}\in X$. 
Let $\rho$ be the corresponding Weyl vector and consider 
$\xi=\exp(2\pi i\rho(0)/h)\in \Aut(V_N)$. 
 
Although $\tau$ does not fix $X$, $\tau$ preserves $(\Lambda_{\tilde{\beta}}+\tilde{\beta})_2$
and $\xi$ acts on $\{e^\alpha \mid \alpha \in (\Lambda_{\tilde{\beta}}+\tilde{\beta})_2 \}$ with the same eigenvalue; hence the commutator 
$[\tau, \xi]=1$. Therefore we can induce $\tau$ to an automorphism of 
$(V_N)^{[\xi^j]}$ for any $j$.  
As it is well-known, $(V_N)^{[\xi]}\cong V_{\Lambda}$. 
Now consider the VOA obtained by an orbifold construction using  $V_N$ and $\xi^2$. 
Then  $(V_N)^{[\xi^{2}]} 
\cong V_M$ for some Niemeier lattice $M$  and the Coxeter number of $M$ is $2$, i.e,  $R(M)=A_1^{24}$. It is clear that  
\[
\Lambda_{\tilde{\beta}} + \bZ sp\tilde{\beta} = \Lambda_{\tilde{\beta}} \cup (\Lambda_{\tilde{\beta}} + sp\tilde{\beta}) \subset  M.
\]  
By Lemma \ref{c2C}, it is easy to check that 
\[L_A(c_{2C})\cong M_\sigma  <  \Lambda_{\tilde{\beta}} + \bZ sp\tilde{\beta} = \Lambda_{\tilde{\beta}} \cup (\Lambda_{\tilde{\beta}} + sp\tilde{\beta}) \subset  M. 
\]
Then $(\Lambda_{\tilde{\beta}} + sp\tilde{\beta})^\tau < M^\tau < M^{\tau^s}= (L_A(c_{2C}))^\perp$. 

Therefore, it suffices to check for the case $\tau\in 2C$ and $N\cong N(A_1^{24})$. 
In this case,  we may assume $N$ contains $A_1^{24}=\oplus_{i=1}^{24} \bZ x_i$ and $N/(A_1^{24})$ is the binary Golay code $G_{24}$ of length $24$ and $\tau(x_i)=x_i$ for $i=1,...,12$ and $\tau(x_i)=-x_i$ for $i=13,...,24$.  
We may also assume $\tilde{\beta}=x_1$. 
By (C3), $N_{\tau}=<x_{13},...,x_{24}, \frac{x_{13}+...+x_{24}}{2}>$.  
Then $\{13,...,24\}$ is a dodecad of $G_{24}$ and so 
$N^{\tau}=(N_{\tau})^{\perp}=<x_1,...,x_{12}, \frac{x_1+...+x_{12}}{2}>$. 

For a Weyl vector $\rho$ of $N$, we may choose $x_i$ such that 
$\exp(2\pi i\rho(0)/2)(x_i)=\sqrt{-1}x_i$ and so 
$\exp(2\pi i\rho(0)/2)(\frac{x_i+...+x_{12}}{2})=\sqrt{-1}^{12}(\frac{x_i+...+x_{12}}{2})=(\frac{x_i+...+x_{12}}{2})$. 
Namely, 
\[(\Lambda_{x_1})^{\tau}= 
<\frac{x_1+...+x_{12}}{2}, x_i+x_1\mid i=1,...,12>.
\] 
We note 
$\langle  \frac{x_1+...+x_{12}}{2}, \frac{x_1+...+x_{12}}{2}\rangle +2\langle \beta, \frac{x_1+...+x_{12}}{2}\rangle=6+2\equiv 0 \pmod{4}$ 
and 
$\langle x_1+x_i, x_1+x_i\rangle+2\langle \beta,  x_1+x_i\rangle=4+4 \equiv 0 \pmod{4}$. 
Since $(\Lambda_{x_1})^{\tau}$ is spanned by such elements and 
$\Lambda_{x_1}$ is self orthogonal modulo $2$, 
$\langle v,v\rangle\equiv 2 \pmod{4}$ for all 
$v\in \Lambda_{\tilde{\beta}}+\tilde{\beta}$, as we desired. 
\prend

We come back to the proof of Theorem \ref{order}. We still assume 
$\tau=2C, 6G$, or $10F$ and $s,p$ are odd. 
The main idea is to study the possibilities of weights in $T^1$ modulo $\bZ$.   Recall that 
\[
T^1 =V_\Lambda[\tilde{g}]= 
\bC[-\beta +(\Lambda^{\tau})^*]\otimes M(1)[\tau]\otimes T_{\tau} 
\] 
and the weights of $M(1)[\tau]\otimes T_{\tau}$ are in $1-\frac{1}{\ell}+\frac{1}{|\tau|} \bZ$.  Therefore, we will focus on the set $-\beta +(\Lambda^{\tau})^*$.  

Since $\Lambda_{\tilde{\beta}}+p\tilde{\beta}$ contains 
a root $\mu$ in $N_{\tau}$, $p\tilde{\beta}\in \pi(\Lambda)$ as we explained 
and we have  $|\Lambda^\tau /\Lambda_{\tilde{\beta}}^\tau |=| (\Lambda_{\tilde{\beta}}^\tau)^*/ (\Lambda^\tau)^*|= p$.  
For $\lambda \in \Lambda$, 
set $\mu=\sqrt{\ell}\varphi(\pi(\lambda))\in \Lambda^{\tau}$. 
We assume $\langle \mu, \tilde{\beta}\rangle=\frac{t}{p}$, that is, 
$\langle \pi(\lambda),\beta\rangle=\frac{t}{4sp}$.  
Since $p\mu \in \Lambda^{\tau}_{\tilde{\beta}}$, 
$$\langle p\mu+\tilde{\beta}, p\mu+\tilde{\beta}\rangle=
p^2\langle \mu,\mu\rangle+2p\langle \mu, \tilde{\beta}\rangle+2 
=\langle p \mu,p\mu\rangle+2t+2$$
and  $\langle p \mu,p\mu\rangle+2t\equiv 0 \mod 4$  by Proposition \ref{mod4}. 
Since $p$ is odd, we have $\langle \mu, \mu\rangle+2t=4m$  for some $m\in \bZ$; namely, 
$\langle \pi(\lambda),\pi(\lambda)\rangle =\frac{4m-2t}{4s}$.   
Hence 
$$
\langle \pi(\lambda)+\beta,\pi(\lambda)+\beta\rangle 
=\frac{4m-2t}{4s}+2\frac{t}{4sp}+\frac{2}{4s} =
\frac{2mp+t(1-p)+p}{2sp} $$
and so 
$\wt(e^{\pi(\lambda)+\beta})=\frac{2mp+t(1-p)+p}{4sp}$. 
Since $2mp+t(1-p)+p$ is always odd, the possible 
weights modulo $\bZ$ are $\frac{\mbox{odd}}{4ps}$, that is, 
there are $h=2sp$ distinct weights modulo $\bZ$ at most.  
That means $|\tilde{g}|\leq h$ and this 
completes the proof of Proposition \ref{order}. 
\prend

We next show that if we start the orbifold construction 
from $V$ by using a $W$-element, we will come back to the original 
$(\tau, \tilde{\beta})$ and $N$.

Let $V=V_\Lambda^{[\tilde{g}]}$. Then we have:
$$V=\bigoplus_{j=0}^{|\tilde{g}|-1} (T^j_{\bZ})^{<\tilde{g}>}.$$

Define $L$ by $\Comm(\Comm(\CH),V),V)\cong V_L$. 
Since $V$ is holomorphic, as a $V_L$-module, $V$ contains a submodule 
isomorphic to $V_{L+\mu}$ for all $\mu\in L^{\ast}$.  

Note that $\tilde{g}^{|\tau|} =\exp(2\pi i (\delta + |\tau|\beta)(0))$ is an inner automorphism. Then 
\begin{equation}\label{L}
L=\bigcup_{j=0}^{|\tilde{g}|/ |\tau| -1}  j(\delta +|\tau|\beta)+\Lambda_\beta^\tau
\end{equation}
Recall that $|\tilde{g}|=h$ and  then $  (h/|\tau|) (\delta +|\tau|\beta) \in \Lambda_\beta^\tau$. Since  $p=h/|\tau|$ is odd and  $2\delta \in \Lambda^\tau$, we have 
$\delta \equiv h\tilde{\beta} \mod \Lambda_\beta^\tau$.

\begin{prop}
$\sqrt{\ell} L^{\ast} \cong \Lambda_{\tilde{\beta}}^{\tau}+\bZ \tilde{\beta}=N^{\tau}$.   
\end{prop}
\pr Set $p=h/|\tau|=2k+1$. Then  $(\delta +|\tau|\beta) +\Lambda^\tau_{\beta} = (k+1)\ell \beta + \Lambda^\tau_{\beta}$. Since $2h\beta=(2k+1)\ell\beta \in \Lambda^\tau$, $
L= \bigcup_{j=0}^{2k}  j(k+1)\ell\beta+\Lambda_\beta^\tau=\Lambda^\tau_{\beta}+\bZ \ell \beta$. 
 By the same argument as in Theorem \ref{duality}, we have  $\sqrt{\ell}L^* = \Lambda_{\tilde{\beta}}^{\tau}+\bZ \tilde{\beta}=N^{\tau}$.  
\prend

\begin{lmm} \label{comm}
$\Comm_V (\CH) = V_{\Lambda_\tau}^{\hat{\tau}}$. 
\end{lmm}

\pr
Suppose false. Then there exits  $ 0<j< |\tilde{g}|=h$ such that  $V_{\Lambda_\beta^\tau} $ appears as a submodule 
of  $T^j_\bZ$. It implies   $j\beta \in (\Lambda^\tau)^*$. It contradicts that $h$ is the smallest integer such that $h\tilde{\beta}\in \Lambda^\tau$ (or equivalently, $h\beta\in \pi(\Lambda)=(\Lambda^{\tau})^{\ast}$).  
\prend

As a corollary, we have 
\begin{lmm}
	$\rank(V_1)=\rank(\Lambda^{\tau})$. 
\end{lmm}

\pr
Since $\CH\subseteq V_{\Lambda}^{<\tilde{g}>}$, we may assume 
$\CH\subseteq V_1$. By the lemma above, we have 
$\Comm_V (\CH)\cap V_1= 0$
. 
\prend

As we have shown, $(\Lambda_{\tilde{\beta}}+\tilde{\beta})_2$ is
a fundamental affine root system $\tilde{X}$ of $R(N)$. 
Note that  $\tilde{\beta}\in \tilde{X}$ by our convention. 
Let $\rho$ be a Weyl vector of $X$. Then  
$\exp(2\pi i\rho(0)/h)$ is an automorphism of $N$ of order $h$ and $N_{\rho/h}=\Lambda_{\tilde{\beta}}$. 
If $v$ is an extended root in $\Lambda_{\tilde{\beta}}+\tilde{\beta}$, 
then $\langle v,\rho\rangle=-(h-1)$.  Therefore, 
$\rho^{\perp}\subseteq \Lambda_{\tilde{\beta}}= \{x\in N\mid \langle \rho, x\rangle\in h\bZ\}$. 
Furthermore, 
for any $w\in R(N)\cap N^{(k)}$, $\langle \rho,w\rangle\cong k \pmod{h}$.

Set $\alpha=\sqrt{\ell}\varphi^{-1}(\pi(\rho)/h)$. 
Then 
\[
\langle \alpha,\beta\rangle
=\langle \sqrt{\ell}\varphi^{-1}(\pi(\rho)/h), \frac{1}{\sqrt{\ell}}\varphi^{-1}(\tilde{\beta})\rangle
= \frac{1}{h}\langle \pi(\rho), \tilde{\beta}\rangle=
\frac{1}{h}\langle \rho, \tilde{\beta}\rangle=\frac{1}{h}.
\] 
More generally, 
$\langle \alpha,L^{\ast}\rangle
=\langle \varphi^{-1}(\pi(\rho)/h), \varphi^{-1}(N^{\tau})\rangle
=\frac{1}{h}\langle \rho, N^{\tau}\rangle < \frac{\bZ}{h}$. 

Since $\CH$ is a Cartan subalgebra of $V_1$ and 
$V_1$ contains a root vector associated with $\beta$, i.e.,  a weight one element 
$e^{\beta}\otimes T_\tau\in T^1$ with 
$\langle \beta,\beta\rangle=\frac{2}{\ell}$. As we have shown in the previous 
sections, the isometry in $Co.0$ determined by  the reverse automorphism defined by a $W$-element of $V_1$ is $\tau$ itself and $\beta$ is one of shortest roots of $V_1$.

\begin{lmm} 
	Let $\gamma$ be a root of $V_1$. Then $\langle \alpha, \gamma\rangle \neq 0$. 
\end{lmm}

\pr 
Since $\rho^{\perp}\cap N^{\tau} \subseteq \Lambda^{\tau}$,  we have 
 $\alpha^{\perp}\cap L^{\ast}\subseteq \pi(\Lambda)$.
  and
we have the desired result. 
\prend

In other words, $\alpha$ is a regular element in $\CH$ and we can  define a positive root $\mu$ by $\langle \mu,\alpha\rangle >0$. 
Note that there is a correspondence between $L^*$ and $N^{\tau}$ through the map $\sqrt{\ell} \varphi_\tau$.  
Since $\langle \rho, v\rangle\in \bZ$ for all $v\in N$, 
$\langle \alpha, \mu\rangle\in \frac{1}h \bZ$ for all $\mu \in L^{\ast}$. 
Since $\langle \alpha,\beta\rangle=\frac{1}{h}$, we may assume 
$\beta$ is a simple short root of a full component.
Therefore, there is a $W$-element $\widetilde{\alpha}$ such that 
$\langle \widetilde{\alpha},\beta\rangle=\frac{1}{h}$. 
In this case, 
$\widehat{\tau}\exp(2\pi i\beta(0))$ is a reverse automorphism of 
$V_\Lambda$ for $\exp(2\pi i\widetilde{\alpha}(0))$ and 
we can recover the pair $(\tau, \tilde{\beta})$ and a Niemeier lattice 
$N=\Lambda_{\tilde{\beta}}+\bZ \tilde{\beta}$.  

 \medskip
\subsection{The relation $\sim$} 
Next we will study the relation $\sim$ on $\CT$. Let  $(\tau, \tilde{\beta})$ and $(\tau', \tilde{\beta}')$ be two pairs in $\CT$. Suppose $(\tau, \tilde{\beta}) \sim (\tau', \tilde{\beta}')$. As we mentioned, up to the action of $O(\Lambda)$, one can assume  $N=N'$ and $\tilde{\beta}'= \tilde{\beta}-\lambda$ for some $\lambda\in \Lambda_{\tilde{\beta}}$. 

\begin{lmm}\label{samecoset}
	Suppose $\tilde{\beta}'= \tilde{\beta}-\lambda$ is also fixed by $\tau$. Then the VOAs defined by $(\tau, \tilde{\beta})$ and $(\tau,\tilde{\beta}')$  are isomorphic. 
\end{lmm}

\pr
 By our assumption, $\lambda \in \Lambda^\tau$ and thus $\bar{\lambda}=\frac{1}{\sqrt{\ell}} \varphi^{-1}(\lambda)\in (\Lambda^\tau)^*$. Then  
$\widehat{\tau}\exp(2\pi i(\beta -\bar{\lambda})(0))=\tilde{g}\exp(-2\pi i \bar{\lambda}(0))$ is conjugate to $\tilde{g}$ by \cite[Lemma 4.5 (2)]{LS}.  Therefore, they define isomorphic VOAs by orbifold constructions. 
\prend

In other words, the choice of $\tilde{\beta}$ is somewhat ambiguous and  
we can choose any $\tau$-invariant root in the affine fundamental root system 
as the starting point.

\begin{prop}
If $(\tau, \tilde{\beta})\sim (\tau', \tilde{\beta}')$, then the VOAs obtained by an orbifold constructions from $V_\Lambda$ associated with  $\tilde{g}=\widehat{\tau}\exp(2\pi i\beta(0))$ and $\tilde{g}'=\widehat{\tau'}\exp(2\pi i\beta'(0))$ are isomorphic, 
where $\beta=\frac{1}{\sqrt{\ell}}\varphi^{-1}(\tilde{\beta})$ and 
 $\beta'=\frac{1}{\sqrt{\ell}}\varphi^{-1}(\tilde{\beta}')$. 
\end{prop}

\pr 
Up to the action of $O(\Lambda)$, one may assume  $N=N'$ and $\tilde{\beta}'= \tilde{\beta}-\lambda$ for some $\lambda\in \Lambda_{\tilde{\beta}}$. 
Since $\tau$ is conjugate to $\tau'$ in $O(N)$, there is a $\mu \in O(N)$ such that 
 $\tau = \mu\tau' \mu^{-1}$. Then $\mu\tilde{\beta}'$ is fixed by $\tau$.  
 Note that $h$ is still the smallest integer such that $h\mu \tilde{\beta}' \in \Lambda^{\tau}$.
 Since  $N= \cup_{j=1}^{h-1}(\Lambda_{\tilde{\beta}}+j\tilde{\beta})$, 
 $\mu \tilde{\beta}' \in (\Lambda_{\tilde{\beta}}+k\tilde{\beta})$ for some $k$ with 
$(k,h)=1$.  Therefore, $<\tilde{g}^k> = <\tilde{g}>$ and $V^{[\tilde{g}]}\cong V^{[\tilde{g}^k]}$.

Recall that  $N_\tau \cong L_A(c_\tau)>R=\oplus A_{m_i-1}^{a_i}$ and $\tau$ acts on $N_\tau$ as a Coxeter element of the root system $R$. Therefore, for $ (k, |\tau|)=1$, $\tau $ is conjugate to 
$\tau^k$ by an element $s\in Weyl(R)$.  In this case, $s\mu\tilde{\beta}' =\mu \tilde{\beta}'$ and 
$s\tau s^{-1} =\tau^k$. Therefore, $(\tau, \mu\tilde{\beta}')$ and $(\tau^k,\mu\tilde{\beta}')$ define isomorphic VOAs. %
Moreover,  by Lemma \ref{samecoset},  the VOAs obtained by an orbifold constructions from $V_\Lambda$ associated with  $\tilde{g}^k=\widehat{\tau}^k\exp(2\pi i k\beta(0))$ and $\widehat{\tau}^k\exp(2\pi i \mu\beta'(0))$ are isomorphic and we have the desired result. 
\prend 

Therefore, there is a one-to-one correspondence between the set of isomorphic 
classes of holomorphic VOA $V$ of central charge $24$ with non-abelian 
$V_1$ and the set of equivalent classes $(\sim)$ of pairs 
$(\tau, \tilde{\beta})$ with $\tau\in \CP_0$ and 
$\tau$-invariant deep hole $\tilde{\beta}$ of $\Lambda$ satisfying 
the conditions in Theorem \ref{Main}.

\begin{remark}
	Since there is a correspondence between deep holes, up to equivalence, and Niemeier lattices, each pair in $\CT$ will define a unique pair $(N, \tau)$,  where $N$ is a Niemeier lattice with $R(N)\neq \emptyset$ and $\tau \in O(N)$ with a positive frame shape.  Therefore, the classification of holomorphic VOAs of central charge $24$ with non-abelian 
	weight one Lie algebras can also  be reduced to a classification of the possible pairs of $(N, \tau)$, up to some equivalence.   
\end{remark}

\section{H\"ohn's observation and Lie algebra structure of $V_1$} \label{S:8}

In \cite{Ho}, H\"ohn observed that there is an interesting bijection 
between certain equivalence classes of cyclic subgroups of the glue codes of the Niemeier lattices with roots 
and the Lie algebra structures of the weight one subspace of a holomorphic vertex operator algebras 
of central charge $24$ corresponding to  the $69$ semisimple cases in 
Schellekens' list. In this section, we will provide an explanation for H\"ohn's observation using our main theorem. In particular,  we will give a pure combinatorial classification of holomorphic vertex operator algebras 
of central charge $24$ with non-trivial weight one spaces. Indeed, the orbit lattice $N(Z)$ described by H\"ohn is essentially  the lattice $L= \sqrt{\ell}(N^\tau)^*$ and the glue vector $v$ is the vector $\lambda_c$ as we defined in the appendix. The rescaling of the levels of the Lie algebras are handled uniformly using the $\ell$-duality map $\sqrt{\ell} \varphi_{\tau}$.

\subsection{Roots of $V_1$ and $N^\tau$} 

First  we study the relationship between the root system of $V_1$ and the lattice 
$N^\tau \cong \sqrt{\ell}L^*$. Recall from \cite{DW} that 
\begin{equation}\label{Q/Ql}
	L_{\g_i}(k_i,0)= \bigoplus_{j\in Q_i/k_iQ^i_{l} } V_{\sqrt{k_i}Q^i_{l} + \frac{1}{\sqrt{k_i}} \beta_j} \otimes M^{0, \beta_j}, 
\end{equation}
where $Q^i$ and $Q^i_{l}$ denote the root lattice and long root lattice of $\g_i$, respectively and 
$\beta_j$ is a representative of $j\in Q^i/k_iQ^i_{l}$.

By \eqref{Q/Ql}, the roots of the Lie algebra $V_1=\bigoplus_{i=1}^t\g_i$ can be represented by elements in $L^*$. Indeed, we can view them as roots of the lattice $\sqrt{\ell}L^*$. 

First we recall the definition for the root system of an even lattice $K$ (cf. \cite{SB,SV}). 
\begin{defn}
	A vector $v\in K$ is \textit{primitive} if the sublattice spanned by $v$ is a direct summand of $L$. A primitive vector $v$ is called a \textit{root} if $2\langle v, K\rangle/\langle v,v\rangle < \bZ$. The set of roots 
	\[
	R(K) = \{ v\in K\mid v \text{ is primitive},\ 2\langle v, K\rangle/\langle v,v\rangle < \bZ\}
	\]
	is called the \textit{root system} of $K$.
\end{defn}

\begin{defn}
	Let $K$ be an even lattice. The level $\ell$ of $K$ is the smallest positive integer $\ell$ such that $\sqrt{\ell}K^*$ is still even. 
\end{defn}

As in \cite{SB}, we denote the scaled root system as follows.
\begin{center}
	\begin{tabular}{|cc|l|}
		\hline
		${^\alpha}\!A_n , {^\alpha}\!D_n , {^\alpha}\!E_n$ && roots of length $2\alpha$ \\
		${^\alpha}\!B_n$ &&  short roots of length $\alpha$ \\
		${^\alpha}\!C_n$ &&  short roots of length $2\alpha$ \\
		${^\alpha}\!G_2$ &&  short roots of length $2\alpha$ \\
		${^\alpha}\!F_4$ &&  short roots of length $2\alpha$\\
		\hline 
	\end{tabular}
\end{center}
We also consider the reduced discriminant group of an arbitrary scaled root system $R$, 
which is a subgroup of the discriminant group of the sublattice generated by $R$ 
(see \cite[Definition 2.1]{SB}). 
For an irreducible root system  $^\alpha\! X$ of type $A$, $D$ or $E$, the reduced discriminant group of $^\alpha\! X$ is simply the discriminant group $\mathcal{D}(X)$. 

For non-simply laced root system, the corresponding root lattices and reduced discriminant groups are as follows. 
\begin{center}
	\begin{tabular}{|cc|l|c|}
		\hline
		Root system  && root lattice &  Reduced discriminant \\ \hline 
		${^\alpha}\!B_n$ &&  $\sqrt{\alpha} \bZ^n$  & $\bZ_2$\\
		${^\alpha}\!C_n$ &&  $\sqrt{\alpha} D_n$&   $\bZ_2$ \\
		${^\alpha}\!G_2$ &&  $\sqrt{\alpha} A_2\cong \sqrt{3\alpha} A_2^*$ & $1$\\
		${^\alpha}\!F_4$ &&  $\sqrt{\alpha} D_4\cong \sqrt{2\alpha} D_4^*$ & $1$\\
		\hline 
	\end{tabular}
\end{center}

For explicit description of the reduced discriminant groups  for 
the root systems of ${^\alpha}\!B_n$ and ${^\alpha}\!C_n$, we use the following standard
model for their root lattices, namely, 
\[
{^\alpha}\!B_n = \bigoplus_{i=1}^n \bZ e_i, \quad \text{ where }  (e_i,e_j)=\alpha \delta_{i,j},
\]
roots: $e_i$ of norm $\alpha$ and $\pm e_i\pm e_j$ of norm $2\alpha$; and 
\[
{^\alpha}\!C_n = \{ \sum_{i=1}^n x_i e_i \mid \sum_{i=1}^n x_i \equiv 0 \mod 2\} \quad \text{ where }  (e_i,e_j)=\alpha \delta_{i,j},
\]
roots: $\pm e_i\pm e_j$ of norm $2\alpha$ and $2e_i$ of norm $2\alpha$.

The reduced discriminant group of $^{\alpha}\! B_n$ is $0$ if $\alpha$ is odd and 
the reduced discriminant group of $R$  is given by  $\langle \bar{g} \rangle \cong \bZ_2$
where 
\[
\begin{cases}
	g=\frac{1}2 (e_1+\cdots +e_n) & \text{ if }  R= ^{2\alpha}\!\! B_n, \\
	g= e_1  & \text{ if }  R= ^{\alpha}\!\! C_n. 	
\end{cases}
\]
\begin{remark}[see \cite{SB}]
	An element $v\in K$ is a root if and only if  $v\in \langle v,v\rangle/2\cdot K^*$ and $v$ is primitive in $K$. In particular, $\langle v,v\rangle/2$ divides the exponent of $\mathcal{D}(K)=K^*/K$.   
\end{remark}

\begin{lmm}\label{roots}
	Let $v\in K$ such that $v\in \langle v,v\rangle/2\cdot K^*$. Assume that $v/n\in K$ for some $n\in \bZ$ with $n>1$. Then $n=2$. Moreover,  $\frac{v}2$ is a root and $K= \bZ(\frac{v}2) \perp  (K\cap \langle v\rangle^\perp)$. 
\end{lmm}
\pr  Since $v/n\in K$ and $v\in \langle v,v\rangle/2\cdot K^*$, we have 
\[
\langle \frac{v}n,v\rangle \in \frac{\langle v,v\rangle}2 \bZ,  \text{ or equivalently, }  2/n\in \bZ;
\]
hence,  $n=2$. Then $\frac{v}2$ is primitive in $K$ and  $\frac{2\langle v/2,x\rangle}{\langle v/2,v/2\rangle} = \frac{4\langle v,x\rangle}{\langle v,v\rangle}\in \bZ$ for any $x\in K$; hence, it is a root.  

Let $p_v: K \to  (\bZ \frac{v}2)^*$ be the natural projection.  For any $x\in K$, $p_v(x) = r_x v$  for some $r_x\in \bQ$. Then $ \frac{2\langle v,x\rangle}{\langle v,v\rangle} = 
\frac{2\langle v,r_xv\rangle}{\langle v,v\rangle} =2r_x\in \bZ$. That means $r_x\in \frac{1}2 \bZ$ and  $p_v(K)= \bZ(\frac{v}2)$ as desired. 
\prend

\medskip

The following results can be found in \cite{SB}. 
\begin{thm}
	Let $K$ be an integral lattice and $R$ a root system contained in $K$.
	Let $\mathcal{K} = K/ (\langle R\rangle \perp (R^\perp \cap K))$  be the glue code of $K$ over $R$.  Then the elements of $R$ are actually
	roots of $K$ if and only if $\mathcal{K}$ is contained in the reduced discriminant group of $R$, and furthermore $K$ contains no elements of the form $e/2$, where $e$ is a basis vector of an $A_1$-component.
\end{thm}



Recall that  $\ell$ is divisible by $|\tau|(K_0-N_0)=\LCM (\{r_ik_i\}_{i=1}^t)$ by Lemmas \ref{rk} and \ref{ell=n}. 

\begin{lmm}\label{Lem:B1roots}
	Let $\beta$ be a root of $\g_{i,k_i}$ and $M=\sqrt{\ell}L^*$. Then 
	\[
	v=\frac{\sqrt{\ell}}{\sqrt{k_i}} \beta \in \frac{\langle v,v\rangle }{2} M^*.
	\]
	In particular, $v=\frac{\sqrt{\ell}}{\sqrt{k_i}} \beta$ is a root of  $M$ if it is primitive in $M$ or $\bZ(v/2)$ is an orthogonal summand of $M$. 
\end{lmm}

\begin{proof}
	Let $M=\sqrt{\ell}L^*$ and $v=\frac{\sqrt{\ell}}{\sqrt{k_i}} \beta$. Then $M^* =\frac{1}{\sqrt{\ell}} L$ and 
	\[
	\frac{\langle v,v\rangle }{2}= 
	\begin{cases}
		\frac{\ell}{r_ik_i} & \text{ if $\beta$ is a short root,}\\
		\frac{\ell}{k_i} & \text{ if $\beta$ is a long root.}
	\end{cases}
	\] 
	Therefore, 
	\[
	v=\frac{\sqrt{\ell}}{\sqrt{k_i}} \beta= 
	\begin{cases}
		\frac{\langle v,v\rangle }{2} \frac{1}{\sqrt{\ell}} \sqrt{k_i} (r_i\beta)& \text{ if $\beta$ is a short root,}\\
		\frac{\langle v,v\rangle }{2} \frac{1}{\sqrt{\ell}} \sqrt{k_i} \beta & \text{ if $\beta$ is a long root.}
	\end{cases}
	\]
	Since $L > \sqrt{k_i}Q^i_{l}$ and $r_i\beta\in Q^i_{l}$ (resp. $\beta \in Q^i_{l}$) if $\beta$ is a short root (resp., a long root),  $v \in \frac{\langle v,v\rangle }{2} M^*$ as desired. 
\end{proof}

\begin{lmm}
	Let $\g_{i,k_i}$ be a simple Lie subalgebra of $V_1$. Suppose there is a  root $\beta$  of 
	$\g_{i,k_i}$ such that $v=\frac{\sqrt{\ell}}{\sqrt{k_i}} \beta$ is not a root of $N^\tau$.  
	Then $\beta$ is a long root and the long root lattice of $\g_i$ is of the type $A_1^r$, where $r=\rank(\g_i)$.  In particular, $\g_i$ is of type $A_1$ or $C_r$. 
\end{lmm}

\pr  
Without loss, we may assume $\rank(V_1)>1$. By Lemma \ref{roots},  $\bZ(v/2)$ is an orthogonal 
summand of $N^\tau$.   Suppose $\beta$ is a short root. Then for any short root $\beta'$, 
$v'=\frac{\sqrt{\ell}}{\sqrt{k_i}} \beta'$ is not a root of $N^\tau$. Therefore,  
$\frac{\sqrt{\ell}}{2\sqrt{k_i}} Q_i \cong \sqrt{s} A_1^r$; however, 
$\bZ \frac{\sqrt{\ell}}{2\sqrt{k_i}} \alpha$ is not an orthogonal summand of $\sqrt{s} A_1^r$  
for any long root $\alpha$.  Therefore,  $\beta$ is a long root and the long root lattice is of type $A_1^r$.  
\prend 

By the lemmas above, the lattice $N^\tau=\sqrt{\ell}L^*$ contains the information of the (scaled) root system of $\g$. 

\subsection{Orbit diagrams and Lie algebra structures of $V_1$} 

Let $V$ be a holomorphic VOA of central charge $24$ and 
$\alpha$ a W-element of $V_1$. Let 
$g=\exp(2\pi i\alpha(0))\in \Aut(V)$ and let 
$\tilde{g}=\hat{\tau}\exp(2\pi i\beta(0))\in \Aut(V_{\Lambda})$ be 
the reverse automorphism of $g$, where $\beta\in \bC\Lambda^{\tau}$.  
Let $\varphi_{\tau}: \sqrt{\ell}(\Lambda^\tau)^* \to \Lambda^\tau $ be the isometry described in  Theorem \ref{dual}. Recall that the vector  $\tilde{\beta}=\varphi(\sqrt{\ell}\beta)$ 
is a deep hole of $\Lambda$ and $N=\Lambda^{[\tilde{\beta}]} \ncong \Lambda$. 
Moreover, the Coxeter number $h$  of $N$ is equal to  $n=\LCM(r_ih_i^\vee)$ and $N^\tau \cong \sqrt{\ell}L^*$. 
 
 \medskip
 
Let $R$ be the set of roots in $N$ and set 
$ 
R_k=\{\mu \in -k\tilde{\beta}+ \Lambda_{\tilde{\beta}} \mid  \langle \mu, \mu\rangle =2 \}.
$ 
Then  $R= \cup_{k=1}^{h-1} R_k$. 
Since the Leech lattice does not contain a root, 
$\langle \mu, \nu \rangle \leq 0$ 
for $\mu, \nu \in R_k$ with $\mu\neq \nu$ for each $k$.  Indeed, $\langle \mu, \nu \rangle =0$ or $-1$  for any $\mu, \nu \in R_k$ with $\mu\neq \nu$.    

One  can associate a (simply laced) Dynkin diagram  with $R_k$ for each $k$. Namely, the nodes are labeled by elements of $R_k$ and two nodes $x$ and $y$ are connected if and only if $\langle x, y\rangle =-1$.  By abuse of notations, we often use $R_k$ to denote both the Dynkin diagram and the subset of roots.  Note that  $\tau$ acts on $R_k$ for each $k$  and acts as a diagram automorphism associated with the diagram defined by $R_k$. 

Since $\tilde{\beta}$ is a deep  hole of the Leech lattice $\Lambda$,  $R_1$ is a disjoint union of the affine diagrams associated with the root system of $N$. Moreover, $\tau$ acts on $R_1$. Since $N_\tau\cong L_A(c)$ and $\tau$ acts as $g_{\Delta,c}$ on a root sublattice of $N_\tau$,  $\tau\in\mathrm{Weyl}(R)$  and preserves  all irreducible components of $R(N)$.                                        Note also that  $g_{\Delta,c}$ induces an isometry in $O(\Lambda)$ if and only 
if $\lambda_c\in N$ \cite{BLS}. Therefore,  the vector $\lambda_c$ corresponds to a codeword of the glue code $N/R$. In particular, 
$P_{Q_i}(\lambda_c)\in (Q_i)^*$ for every irreducible root sublattice $Q_i$ of $N$.  Note also that $\tau$ has a positive frame when viewing as an isometry of $N$.  

Therefore, for each irreducible component, we can consider the quotient diagram as follows:  
we identify an orbit of nodes as one node and two nodes are connected if the nodes in the corresponding orbits are connected. By removing the node associated with the extended node, 
one obtain a usual Dynkin diagram (see Table \ref{affinediagram}).  
     
Note that the fixed sublattice is the (scaled) root lattice of the quotient diagram. A fixed node (or fixed simple root) corresponds to a simple short root of a full component.

{\small
	\begin{longtable}{|c|c|c|c|c|c|c|c|}
		\caption{Diagram automorphisms of affine diagrams}\label{affinediagram}
		\\ \hline 
		Type & $A_n$ & $D_{2k}$ &$D_{2k}$ &$D_{2k+1}$ &$D_{2k+1}$ &$E_6$& $E_7$  \\ \hline 
		Root subsystem  & $ (A_{\frac{n+1}{k}-1})^{k}$ &$A_1^{k}$&$A_1^2$& $A_3 A_1^{k-1}$&$A_1^2$ &$A_2^2$& $A_1^3$ \\ \hline
		Frame Shape & $1^{-1} (\frac{n+1}{k})^{k}$&$2^{k}$& $1^{2k-4}2^2$&  $1^{-1} 2^{k-1} 4$&$1^{2k-3} 2^2$ & $3^2$ & $1^1 2^3$ \\ \hline 
		Quotient diagram &  $A_{k-1}$&$B_{k} $ &$ C_{2k-2}$ &$ C_{k-1}$ &$C_{2k-1}$ & $G_2$& $F_4$ \\ \hline 
		Fixed sublattice & $\sqrt{\frac{n+1}{k}} A_{k-1}$& $A_1^{k}$ & $D_{2k-2}$ &$A_1^{k-1}$ &$D_{2k-1}$ & $A_2$ &$D_4$ \\ \hline
		Fixed simple roots &  $\emptyset $ & $A_1$ &  $A_{2k-3}$ &  $\emptyset$ &   $A_{2k-2}$ &  $A_1$ & $A_2$\\
		\hline 
	\end{longtable}
}

\begin{remark}
	Let $\CG_i$ be a simple Lie subalgebra of $V_1$ with $r_ih_i^\vee=n=h$.  Then $\ell =r_ik_i$ and 
	$k_i/h_i^\vee= \ell/h$. Therefore, the level $k_j$ of the simple Lie subalgebra $\CG_j$  is given by $k_j= \ell h^\vee_j/h$ for any $j$.  Note also that 
	the short roots of $\CG_i$ will correspond to an irreducible (connected) component $S_i$ of $N^\tau_2$.  Moreover,  $S_i\cap R_1$ corresponds to the simple short roots of $\CG_i$.  Therefore,  $S_i$  and $S_i\cap R_1$ determines the type of $\CG_i$ uniquely.
\end{remark} 
 
\begin{remark}
In \cite{MoS2},  a notion of generalized hole diagrams is introduced. It was also shown that  a generalized hole diagram determines a generalized deep hole up to conjugacy and that there are exactly $70$ such diagrams.  This notion of generalized hole diagrams essentially corresponds to the diagram associated with simple short roots of the full components (i.e., elements in  $R_1\cap N^\tau_2$).  
\end{remark}

\subsection{Possible pairs for $(N, \tau)$}
Next we will discuss the possible choices for the pair $(N, \tau)$ for each $\tau\in \CP_0$.

\subsubsection{$\tau\in 2A$} For $\tau\in 2A$ of $O(\Lambda)$, we have $\ell=|\hat{\tau}|=|\tau|=2$. 
In this case, $\Lambda_\tau\cong \sqrt{2}E_8$.  Let $N=\Lambda^{[ \tilde{\beta}] }$. Then the Coxeter number of $N$ is divisible by $2$
 and 
 \[
 N_\tau=\mathrm{Span}_\bZ\{ A_1^8, \frac{1}2 (\alpha_1 +\cdots+\alpha_8)\}, 
 \]
 where $\bZ \alpha_1 +\cdots+\bZ\alpha_8\cong A_1^8$. The vector $v= \frac{1}2 (\alpha_1 +\cdots+\alpha_8)$ corresponds to a codeword $c\in N/R$.  The possible choices for $(N, \tau)$, the codeword $c$ and the corresponding root systems and Lie algebra structures for $V_1$ are listed in Table \ref{T:2to6}.    

\begin{longtable}[c]{|c|c|c|c|c|c|}
	\caption{$(N,\tau)$ for the case $2A$} \label{T:2to6}\\
	\hline 
	Type & Codeword $c$ &Embedding & $R(N^\tau)$ & $V_1$ \\ \hline 
	\hline 
	$A_1^{24}$& $(1^8,0^8)$ &$A_1^8 \hookrightarrow A_1^8$ &$A_1^{16}$&$A_{1,2}^{16}$  \\ 
	$A_3^8$ & $(22022000$)&$(A_1^2)^4 \hookrightarrow A_3^4$&$A_3^4(\sqrt2A_1)^4$&$A_{3,2}^4A_{1,1}^4$ \\  
	$D_4^6$ &$(233200)$&$(A_1^2)^4 \hookrightarrow D_4^4$&$D_4^2C_2^4$&$D_{4,2}^2C_{2,1}^4$ \\ 
	$A_5^4D_4$&$(3300|1)$	&$(A_1^3)^2+ A_1^2 \hookrightarrow A_5^2+D_4$ &$A_5^2C_2(\sqrt2A_2)^2$&$A_{5,2}^2C_{2,1}A_{2,1}^2$
	\\ 
	$A_7^2D_5^2$& $(44|00)$ &$(A_1^4)^2 \hookrightarrow A_7^2$ & $D_5^2(\sqrt2A_3)^2$&$D_{5,2}^2A_{3,1}^2$\\ 
	$A_7^2D_5^2$& $(20|33)$ &$(A_1^4) + (A_1^2)^2 \hookrightarrow A_7+D_5^2$  &$A_7C_3^2(\sqrt2A_3)$&$A_{7,2}C_{3,1}^2A_{3,1}$ 	\\ 
	$D_6^4$ & $(2222)$ &$(A_1^2)^4 \hookrightarrow D_6^4$&$C_4^4$&$C_{4,1}^4$ \\ 
	$D_6^4$ & $(1230)$&$(A_1^2)+ (A_1^3)^2 \hookrightarrow D_6+D_6^2$ &$D_6C_4B_3^2$&$D_{6,2}C_{4,1}B_{3,1}^2$\\ 
	$A_9^2D_6$& $(05|3)$ &$(A_1^5)+ (A_1^3) \hookrightarrow A_9+D_6$  &$A_9(\sqrt2A_4)B_3$&$A_{9,2}A_{4,1}B_{3,1}$ 	\\ 
	$A_{11}D_7E_6$& $(620)$	&$A_1^6+ A_1^2 \hookrightarrow A_{11}+D_7$ &$E_6C_5(\sqrt2A_5)$&$E_{6,2}C_{5,1}A_{5,1}$\\ 
	$D_8^3$& $(033)$&$(A_1^4)^2 \hookrightarrow D_8^2$ &$D_8B_4^2$&$D_{8,2}B_{4,1}^2$\\ 
	$D_8^3$& $(221$)&$(A_1^2)^2+ A_1^4 \hookrightarrow D_8^2+D_8$ &$C_6^2B_4$ &$C_{6,1}^2B_{4,1}$ \\ 
	$A_{15}D_9$& $(80)$&$A_1^8 \hookrightarrow A_{15}$&$D_9(\sqrt2A_7)$&$D_{9,2}A_{7,1}$\\ 
	$E_7^2 D_{10}$& $(11|2)$ &$(A_1^3)^2 +A_1^2 \hookrightarrow E_7^2+D_{10}$ &$C_8F_4^2$&$C_{8,1}F_{4,1}^2$\\
	$E_7^2 D_{10}$ & $(01|1)$ &$A_1^3 +A_1^5 \hookrightarrow E_7+D_{10}$ &$E_7B_5F_4$&$E_{7,2}B_{5,1}F_{4,1}$\\ 
	$D_{12}^2$ & $(21)$ &$A_1^2 +A_1^6 \hookrightarrow D_{12} +D_{12}$&$C_{10}B_6$&$C_{10,1}B_{6,1}$\\
	$E_8D_{16}$ &$(01)$&$A_1^8 \hookrightarrow D_{16}$&$B_8E_8$&$B_{8,1}E_{8,2}$\\ 
	\hline 
\end{longtable}

\subsubsection{$\tau\in 3B$}
For $\tau\in 3B$ of $O(\Lambda)$, we have $\ell=|\hat{\tau}|=|\tau|=3$. 
In this case, $\Lambda_\tau\cong K_{12}$ is the Coxeter-Todd lattice of rank $12$. Let $N=\Lambda^{[ \tilde{\beta}] }$. Then the Coxeter number of $N$ is divisible by $3$
and $N_\tau=\mathrm{Span}_\bZ\{ A_2^6, (\gamma_1, \dots, \gamma_6)\}$,  
where $\gamma_i +A_2$ is a generator of $A_2^*/A_2$ for each $i$.  The vector $(\gamma_1, \dots, \gamma_6)$ corresponds to a codeword $c\in N/R$.  The possible choices for $(N, \tau)$, $c$ and the corresponding root systems and Lie algebra structures for $V_1$ are listed in Table \ref{table3to8}.

\begin{longtable}[c]{|c|c|c|c|c|c|} 
	\caption{$(N,\tau)$ for the case $3B$} \label{table3to8}\\
	\hline 
	Type & Codeword $c$ & Embedding & $R(N^\tau)$ & $V_1$ \\ \hline  
	\hline 
	$A_2^{12}$& $(1^60^6)$ &$A_2^6 \hookrightarrow A_2^6$&$A_2^6$&$A_{2,3}^{6}$ \\ 
	$A_5^4D_4$	& $(2220|0)$ &$(A_2^2)^3 \hookrightarrow A_5^3$ &$A_5D_4(\sqrt3A_1)^3$ &${A_{5,3}}{D_{4,3}}A_{1,1}^{3}$  \\ 
	$A_8^3$	& $(630)$ & $(A_2^3)^2 \hookrightarrow A_8^2$ &$A_8(\sqrt{3}A_2)^2$ & $A_{8,3}A_{2,1}^2$     \\ 
	$E_6^4$	& $(0111)$ &$(A_2^3)^3 \hookrightarrow E_6^3$&$E_6 G_2^3$ & $E_{6,3}{G_{2,1}}^3$  \\
	$A_{11}D_7E_6$	& $(401)$ &$A_2^4+ A_2^2 \hookrightarrow A_{11} E_6$  &$D_7(\sqrt3A_3) G_2$ 
	& ${D_{7,3}}{A_{3,1}}{G_{2,1}}$ \\  
	$A_{17}E_7$	& $(60)$ &$A_2^6 \hookrightarrow A_{17}$ &$E_7 {(\sqrt3A_5)}$ &$E_{7,3}A_{5,1}$ \\ \hline
\end{longtable}

\subsubsection{$\tau\in 5B$}
For $\tau\in 5B$ of $O(\Lambda)$, we have $\ell=|\hat{\tau}|=|\tau|=5$ and the Coxeter number of $N=\Lambda^{[ \tilde{\beta}] }$ is divisible by $5$. In this case, $N_\tau\cong \mathrm{Span}_\bZ\{ A_4^4, (\gamma_1, \gamma_2, \gamma_3,  \gamma_4)\}$, 
where $\gamma_i +A_4$ is a generator of $A_4^*/A_4$ for each $i$.  Again, the vector $(\gamma_1, \gamma_2, \gamma_3,  \gamma_4)$ corresponds to a codeword $c\in N/R$. That means $5=|\tau|$ divides $|N/R|$, also.  The possible choices for $(N, \tau)$, $c$ and the corresponding root systems and Lie algebra structures are listed in Table \ref{table5to6}.

\begin{longtable}[c]{|c|c|c|c|c|c|} 
	\caption{$(N,\tau)$  for the case $5B$} \label{table5to6}\\
	\hline 
	Type & Codeword $c$ & Embedding & $R(N^\tau)$ & $V_1$ \\ 
	\hline \hline 
	$A_4^6$	& $(123400)$ &$A_4^4 \hookrightarrow A_4^4$ &$A_4^2$ & $ A_{4,5}^2$ \\
	$A_9^2D_{6}$& $(24|0)$ & $(A_4^2)^2 \hookrightarrow A_9^2$ &$D_6(\sqrt{5}A_1^2)$  
	& $D_{6,5}A_{1,1}^2$ \\ \hline
\end{longtable}

\subsubsection{$\tau\in 7B$}

For $\tau\in 7B$ of $O(\Lambda)$, we have $\ell=|\hat{\tau}|=7$. Let $N=\Lambda^{[ \tilde{\beta}] }$. Then the Coxeter number of $N$ is divisible by $7$
and 
\[
N_\tau=\mathrm{Span}_\bZ\{ A_6^3, (\gamma_1,\gamma_2,  \gamma_3)\}, 
\]
where $\gamma_i +A_6$ is a generator of $A_6^*/A_6$ for each $i$.  The vector $(\gamma_1,\gamma_2,  \gamma_3)$ corresponds to a codeword $c\in N/R$.  The possible choices for $(N, \tau)$, $c$ and the corresponding root systems and Lie algebra structures are listed in Table \ref{table7to5}

\begin{longtable}[c]{|c|c|c|c|c|c|} 
		\caption{$(N,\tau)$  for the case $7B$} \label{table7to5}\\
	\hline 
	Type& Codeword $c$ & Embedding & $R(N^\tau)$ & $V_1$ \\ 
	\hline \hline  
	$A_6^4$& $(0124)$ &$A_6^3 \hookrightarrow A_6^3$&$A_6$ & $ A_{6,7}$  \\ \hline
\end{longtable}

\subsubsection{$\tau\in 2C$} For $\tau\in 2C$ of $O(\Lambda)$, we have $\ell=|\hat{\tau}|=2|\tau|=4$.

Let $N=\Lambda^{[ \tilde{\beta}] }$. Then the Coxeter number of $N$ is divisible by $2$
and 
\[
N_\tau=\mathrm{Span}_\bZ\{ A_1^{12}, \frac{1}2(\alpha,\dots,  \alpha)\}, 
\]
where $\bZ\alpha \cong A_1$.   The vector $\frac{1}2(\alpha,\dots,  \alpha)$ again corresponds to a codeword $c\in N/R$.  The possible choices for $(N, \tau)$, $c$ and the corresponding root systems and Lie algebra structures are listed in Table \ref{2c1}.

\begin{longtable}[c]{|c|c|c|c|c|c|c|} 
	\caption{$(N, \tau)$ for the case $2C$} \label{2c1}\\
	\hline 
		Type& Codeword $c$ & Embedding & $R(N^\tau)$ & $V_1$ \\ 
	\hline 	\hline 
	$A_1^{24}$& $(1^{12} 0^{12})$ &$A_1^{12} \hookrightarrow A_1^{12}$ &$A_1^{12}$  &$A_{1,4}^{12}$ \\
	$D_4^{6}$&$(111111)$ &$(A_1^{2})^{6} \hookrightarrow D_4^{6}$ & $B_2^6$ &$B_{2,2}^6$ \\
	$D_6^{4}$&$(2222)$ &$(A_1^{3})^{4} \hookrightarrow D_6^{4}$ & $B_3^4$ & $B_{3,2}^4$ \\	
	$D_8^{3}$&$(111)$ &$(A_1^{4})^{3} \hookrightarrow D_8^{3}$ & $B_4^3$ &$B_{4,2}^3$  \\
	$D_{12}^2$ & $(11)$ &$(A_1^{6})^{2} \hookrightarrow D_{12}^{2}$ & $B_2^6$ &$B^2_{6,2}$ \\
	$D_{24}$ & $(1)$ &$A_1^{12} \hookrightarrow D_{24}$ &$B_{12}$ & $B_{12,2}$ \\ 
	$A_5^4D_{4}$ & $(3333|0)$ &$(A_1^{3})^4 \hookrightarrow A_5^4$ &$D_4 \sqrt{2}A_2^4$ & $D_{4,4}A^4_{2,2}$ \\ 
	$A_9^2D_{6}$ & $(55|2)$ &$(A_1^{5})^2+A_1^2 \hookrightarrow A_9^2+D_6$ &$C_4 \sqrt{2}A_4^2$  &$C_{4,2}A^2_{4,2}$\\
	$A_{17}E_7$ &$(9|1)$ &$A_1^{9}+A_1^3 \hookrightarrow A_{17}+E_7$ &$F_4 \sqrt{2}A_8$  
	&$ A_{8,2}F_{4,2}$\\
	\hline 
\end{longtable}

\subsubsection{$\tau\in 4C$} For $\tau\in 4C$ of $O(\Lambda)$, we have $\ell=|\hat{\tau}|=|\tau|=4$.  Let $N=\Lambda^{[ \tilde{\beta}] }$. Then the Coxeter number of $N$ is divisible by $4$ and  the coinvariant lattice 
\[
N_\tau = \mathrm{Span}_\bZ\{ A_3^4 A_1^2, (\lambda, \lambda, \lambda, \lambda, \frac{1}{2}\alpha, \frac{1}{2}\alpha)\},  
\]
where $\lambda= 1/2(1,1,1,-3)\in A_3^*$ and $\bZ\alpha= A_1$. Note that $N^\tau$ contains  $A_3^4 A_1^2$ as an index $4$ sublattice. 
The possible choices for $(N, \tau)$, $c$ and the corresponding root systems and Lie algebra structures are listed in Table \ref{4Ctable}.

	\begin{longtable}{|c|c|c|c|c|}  
\caption{$(N, \tau)$ for the case $4C$} \label{4Ctable} \\		\hline 
		Type& Codeword $c$ & Embedding &  $R(N^\tau)$ & $V_1$ \\ \hline \hline
		$A_3^8$ & $(32001011)$ 	& $A_3^4+A_1^2 \hookrightarrow A_3^4 +A_3$&  $A_3^3 \sqrt{2}A_1$ &  $A_{3,4}^3A_{1,2}$ 	\\ \hline	
		$A_7^2D_5^2$ & $(02|13)$ &$A_3^2+ (A_3A_1)^2 \hookrightarrow A_7 +D_5^2$  &  
		$A_72A_1 A_1^2$ &$A_{7,4}A_{1,1}^3$ \\ \hline
		$A_7^2D_5^2$ & $(22|20)$ &$A_3^2+ A_3^2+ A_1^2 \hookrightarrow A_7 + A_7+D_5$  &   
		$D_5 C_3 2A_1^2$ & $D_{5,4} C_{3,2} A_{1,1}^2$ \\ \hline
		$A_{11} D_7 E_6$ &$(310)$& $A_3^3+ A_3A_1^2 \hookrightarrow A_{11} +D_7$  &  
		$E_6 B_2 2A_2$ &  $E_{6,4}B_{2,1}A_{2,1}$ \\ \hline
		$A_{15} D_9$ & $(4|2)$ &$A_3^4+ A_1^2 \hookrightarrow A_{15}+D_9$  &  
		$C_7 2A_3$ & $C_{7,2}A_{3,1}$ \\ \hline
	  	\end{longtable}

\subsubsection{$\tau\in 6E$} For  $\tau\in 6E$ of $O(\Lambda)$, we have $\ell=|\hat{\tau}|=6$.
 Let $N=\Lambda^{[ \tilde{\beta}] }$. Then the Coxeter number of $N$ is divisible by $6$ and the coinvariant sublattice 
\[
N_\tau \cong \mathrm{Span}_\bZ \{ A_5^2A_2^2A_1^2 , (\beta, \beta, \gamma, \gamma, \frac{1}2 \alpha,  \frac{1}2 \alpha)\},  
\]
where $\beta= \frac{1}6(1^5, -5)\in A_5^*$, $\gamma= \frac{1}3 (1,1,-2)\in A_3^*$ and $\langle \alpha, \alpha\rangle=2$.  Note that $N^\tau$ is an index $6$ sublattice of $E_8\perp E_8$. 
 The possible choices for $(N, \tau)$, $c$ and the corresponding root systems and Lie algebra structures are listed in Table \ref{6Etable}. 
 
	\begin{longtable}{|c|c|c|c|c|}
		\caption{$(N, \tau)$ for the case $6E$} \label{6Etable} \\ \hline 
		Type& Codeword $c$ & Embedding &  $R(N^\tau)$ & $V_1$ \\ \hline 
		$A_5^4 D_4$ & $(0255|1)$ &$A_5^2+ A_2^2+ A_1^2 \hookrightarrow A_5^2 +A_5+ D_4$  & $A_5 
		\sqrt{3}A_1 B_2$ & $A_{5,6}B_{2,3}A_{1,1}$   \\ \hline
		$A_{11} D_7 E_6$ & $(222)$ &$A_5^2+ A_1^2+ A_2^2  \hookrightarrow A_{11} +D_7+ E_6$  & $\sqrt{6}A_1 C_5 G_2$ & $C_{5,3}G_{2,2}A_{1,1}$ \\
		\hline  
	\end{longtable}

\subsection{$\tau\in 8E$}  
For $\tau\in 8E$, $\ell=|\hat{\tau}|=8$.  Let $N=\Lambda^{[ \tilde{\beta}] }$. Then the Coxeter number of $N$ is divisible by $8$ and the coinvariant lattice  $N_\tau$ is an index $8$ over-lattice of the lattice $A_7^2A_3A_1$. 
More precisely,  
\[
N_\tau= \mathrm{Span}_\bZ\{ A_7^2A_3A_1, (\gamma_3, \gamma_1, \beta, \alpha) \},
\]
where $\gamma_3=\frac{1}8 (3^5, -5^3)$, $\gamma_1=\frac{1}8 (1^7, -7)$ are in $A_7^*$, 
$\beta\in \frac{1}2(1,1,1)\in A_3^* $ and  $\alpha =\frac{1}2 (1,1)\in A_1^*$. 
 
 The possible choices for $(N, \tau)$, $c$ and the corresponding root systems and Lie algebra structures are listed in Table \ref{6Etable}. 
 
\begin{center}
	\begin{tabular}{|c|c|c|c|c|}
		\hline 
		Type& Codeword $c$ & Embedding &  $R(N^\tau)$ & $V_1$ \\ \hline 
		$A_7^2 D_5^2$ & $(37|10)$ &$A_7^2+A_3A_1 \hookrightarrow A_7^2 +D_5$  & $D_5A_1$ & $D_{5,8} A_{1,2}$   \\ \hline
		\end{tabular}
\end{center}

\subsection{$\tau \in 6G$}  
For $\tau\in 6G$, $\ell =|\hat{\tau}| =12$.  Let $N=\Lambda^{[ \tilde{\beta}] }$. Then the Coxeter number of $N$ is divisible by $6$ and $N_\tau$ contains $A_5^3 A_1^3$ as an index $6$ sublattice. 
 The possible choices for $(N, \tau)$, $c$ and the corresponding root systems and Lie algebra structures are listed in the following table.

\begin{center}
	\begin{tabular}{|c|c|c|c|c|}
		\hline 
		Type & Codeword $c$ & Embedding &  $R(N^\tau)$ & $V_1$ \\ \hline 
		$A_5^4 D_4$ & $(31110)$ &$A_5^3+A_1^3 \hookrightarrow A_5^3 +A_5$  & $D_4\sqrt{2}A_2$ & $D_{4,12} A_{2,6}$   \\ \hline
		$A_{17} E_7$ & $(3|1)$ &$A_5^3+A_1^3 \hookrightarrow A_{17} +E_7$  & $F_4\sqrt{6}A_2$ & $F_{4,6} A_{2,2}$   \\ \hline
	\end{tabular}
\end{center}

\subsection{$\tau \in 10F$} 
For $\tau \in 10F$, $\ell =|\hat{\tau}|=20$. The Coxeter number of $N=\Lambda^{[ \tilde{\beta}] }$ is divisible by $10$ and $N_\tau$ contains $A_9^2 A_1^2$ as an index $10$ sublattice. 
 The possible choices for $(N, \tau)$, $c$ and the corresponding root systems and Lie algebra structures are listed in the following table.

\begin{center}
	\begin{tabular}{|c|c|c|c|c|}
		\hline 
		Type & Codeword $c$ & Embedding &  $R(N^\tau)$ & $V_1$ \\ \hline 
		$A_9^2 D_6$ &$(79|2)$ & $A_9^2+A_1^2 \hookrightarrow A_9^2 +D_6$  & $C_4$ & $C_{4,10}$   \\ \hline
		\end{tabular}
\end{center}

\begin{remark}
Since $N_\tau= L_A(c)$ with $c$ as  a codeword of the glue code $N/R$, we can recover the same information as in \cite[Table 3]{Ho}. In particular, there are exactly $46$ possible Lie algebra structures for $V_1$ if  $0< \rank(V_1)< 24$. This gives an alternative proof for the Schellekens list.  
\end{remark}

\appendix
\section{Properties about the lattice $\Lambda_\tau$ for $\tau\in \CP_0$}\label{appA}

In this appendix, we will review some  properties about the coinvariant sublattice $\Lambda_\tau$ for $\tau\in \CP_0$. Let $\Lambda$ be the Leech lattice and let $$\tau \in \CP_0=\{1A,2A, 2C, 3B, 5B, 7B, 4C, 6E, 6G, 8E, 10F\}.$$
Let $\Lambda_\tau$ be the coinvariant lattice of $\tau$. Then $\tau$ is fixed-point free on $\Lambda_\tau$.

The following can be verified by MAGMA.

\begin{lmm}\label{Lem:conjclass0} 
Let $\tau\in \CP_0$. Then 
\begin{enumerate}
\item $(1-\tau)\Lambda_{\tau}^*=\Lambda_\tau$.

\item The quotient group $C_{O(\Lambda_\tau)}(\tau)/\langle \tau \rangle$ acts faithfully on $\mathcal{D}(\Lambda_\tau)$. 
\end{enumerate}
\end{lmm}

For $\tau\in \CP_0$, the coinvariant lattice $\Lambda_\tau$ can be constructed by the so-called generalized ``Construction B".  First, we  review the construction.  

\paragraph{Generalized ``Construction B".}

Let $R_i$ $(1\le i\le t)$ be a copy of the root lattice of type $A_{k_i-1}$, where  $k_i\in\bZ_{\ge 1}$ for $1\le i\le t$. Let $R=R_1\perp R_2\perp\dots\perp R_t$.
Then $\mathcal{D}(R_i)\cong\bZ_{k_i}$ and $\mathcal{D}(R)\cong\bigoplus_{i=1}^t\bZ_{k_i}$.
Let $\nu:R^*\to R^*/R=\mathcal{D}(R)\cong\bigoplus_{i=1}^t\bZ_{k_i}$ be the canonical surjective map.
For a subgroup $C$ of $\bigoplus_{i=1}^t\bZ_{k_i}$, let $L_A(C)$ denote the lattice defined by 
\begin{equation}
L_A(C)=\nu^{-1}(C)=\{\alpha\in R^*\mid \nu(\alpha)\in C\};\label{Eq:ConstA}
\end{equation}
we call $L_A(C)$ the lattice constructed by \emph{Construction A} from $C$.
Note that $L_A(\{\mathbf{0}\})=R$, where $\mathbf{0}$ is the identity element of $\bigoplus_{i=1}^t\bZ_{k_i}$.

We now fix a base $\Delta_i$ of the root system of $R_i$, which is of type $A_{k_i-1}$.
Then $\Delta=\bigcup_{i=1}^t\Delta_i$ is a base of the root system of $R$.
For $x=(x_i)\in\bigoplus_{i=1}^t\bZ_{k_i}$, denote 
\begin{equation}
\lambda_x=(\lambda_{x_1}^1,\dots,\lambda_{x_t}^t)\in R_1^*\perp\dots\perp R_t^*=R^*,\label{Eq:lambdac}
\end{equation}
where $x_i$ is regarded as an element of $\{0,\dots,k_i-1\}$ and $\{\lambda_{j}^i\mid 1\le j\le k_i-1\}$ is the set of fundamental weights in $R_i^*$ with respect to $\Delta_i$.
The following lemma is immediate from the definition of $L_A(C)$.

\begin{lmm}\label{L:genA} For a generating set $\mathcal{C}$ of $C$, 
the set $\{\lambda_c\mid c\in\mathcal{C}\}$ and $R$ generate $L_A(C)$ as a lattice.
\end{lmm}

Set 
\begin{equation}
\chi_\Delta=(\frac{\rho_{\Delta_1}}{k_1},\dots,\frac{\rho_{\Delta_t}}{k_t})\in \bQ\otimes _\bZ R\label{Eq:chi},
\end{equation}
where $\rho_{\Delta_i}$ is the Weyl vector of $R_i$ with respect to $\Delta_i$.

Define $
L_B(C)=\{\alpha\in L_A(C)\mid (\alpha|\chi_\Delta)\in\bZ\}
$;  
we call $L_B(C)$ the lattice constructed by \emph{Construction B} from $C$.

\begin{remark} Up to isometry,  $L_B(C)$ does not depend on the choice of a base $\Delta$.
\end{remark}

By definitions, it is easy to show  the following results (see \cite{LS21}).  

\begin{lmm}\label{Lem:enorm} Let $x=(x_i)\in \bigoplus_{i=1}^t\bZ_{k_i}$.
	Then $(\lambda_x|\lambda_x)\in 2\bZ$ if and only if $(\lambda_x|\chi_\Delta)\in\bZ$.
\end{lmm}

Set $n=\mathrm{LCM}(\{k_1,\dots,k_t\})$. 

\begin{lmm}\label{L:indexn}
	$|L_A(C):L_B(C)|=n$ if and only if $\chi_\Delta\in (1/n)L_A(C)^*$.
\end{lmm}

Next we consider some isometry of $R$. Recall that $R_i\cong A_{k_i-1}$ is a
root lattice of type $A_{k_i-1}$. Let $\Delta_i=\{\alpha^i_1, \dots, \alpha^i_{k_i-1}\}$ be a set of simple roots and let $ \alpha^i_0=-\sum_{j=1}^{k_i-1} \alpha^i_{j}$ be the negative of the highest root. Then the map $g_{\Delta_i}(\alpha^i_j)=\alpha^i_{j+1}$ if $1\le j\le k_i-2$ and $g_\Delta(\alpha^i_{k_i-1})=\alpha^i_0$ defines an isometry on $R_i$, which is a Coxeter element of the Weyl group of $R_i$. In particular, $g_{\Delta_i}$ acts on $\tilde{\Delta}_i=\Delta_i \cup \{\alpha^i_0\}$ as a cyclic permutation of order $k_i$.

For $e=(e_i)\in\bigoplus_{i=1}^t\bZ_{k_i}$, set 
\begin{equation}\label{Eq:gre}
	g_{\Delta,e}=((g_{\Delta_1})^{e_1},\dots,(g_{\Delta_t})^{e_t})\in O(L_A(C)).
\end{equation}

\begin{lmm}\label{L:gNc} $g_{\Delta,e}\in O(L_B(C))$ if and only if $\lambda_e\in L_A(C)^*$.
\end{lmm}
\begin{proof} 
	By definition, it is easy to see that 	
	\begin{equation}
		g_\Delta(\lambda_j)=\lambda_j-\sum_{i=1}^j\alpha_i,\quad g_{\Delta}(\rho_{\Delta})=\rho_{\Delta}-k\lambda_{1}.
		\label{Eq:grho}
	\end{equation}
	Then we have  
	$   
		g_{\Delta,e}(\chi_\Delta)\in \chi_\Delta+(e_1\lambda_1^1,\dots,e_t\lambda_1^t)+R=\chi_\Delta+\lambda_e+R.
	$  
	By the definition of $L_B(C)$,  
	$g_{\Delta,e}\in O(L_B(C))$ if and only if 
	$\lambda_e\in L_A(C)^*$.
\end{proof}
\begin{lmm}\label{L:fpf2} The isometry $g_{\Delta,e}$ is fixed-point free and of order $n$ if and only if $\gcd(e_i,k_i)=1$ for all $1\le i\le t$
\end{lmm}

It turns out that the coinvariant lattice $\Lambda_\tau$ for $\tau\in \CP_0$ can be constructed as $L_B(C)$ with $C$ generated by a single glue vector $c$. Moreover, $\tau$ can be identified with $g_{\Delta,c}$ as defined above.

\begin{prop}[\cite{BLS}] For any $\tau\in \CP_0$, the coinvariant lattice $\Lambda_\tau$ of the Leech lattice can be constructed as $L_B(C)$ with $C$ generated by a single glue vector $c$. Moreover, $\tau|_{L_B(C)}=g_{\Delta,c}$ as defined above.  
\end{prop}

Some properties of $\Lambda_\tau$ are summarized in Table \ref{glue}; the structures of $O(\Lambda_g)$ and $C_{O(\Lambda_g)}(g)$ are computed by using MAGMA.
The symbol $\prod{a_i}^{b_i}$ for $\Lambda_\tau^*/\Lambda_\tau$ means the abelian group $\bigoplus(\bZ/a_i\bZ)^{b_i}$. For the notations of groups, see \cite{ATLAS}. 

{\small
\begin{longtable}{|c|c|c|c|c|c|}
\caption{Coinvariant lattices $\Lambda_\tau$ for $\tau\in \CP_0$}\label{glue}
\\ \hline 
Class&
$R$ & $c$&$\Lambda_{\tau}^*/\Lambda_\tau$&$O(\Lambda_\tau)$&$C_{O(\Lambda_\tau)}(\tau)$ \\ \hline
$2A$ & $A_1^8$& $(11111111)$& $ 2^{8}$ & $2^{1+8}.GO_8^+(2)$ & $2^{1+8}.GO_8^+(2)$ \\
$2C$ & $A_1^{12}$& $(1^{12})$& $ 2^{12}$ & $2^{11}. \mathrm{Sym}_{12}$ & $2^{11}. \mathrm{Sym}_{12}$ \\
$3B$ & $A_2^6$& $(111111)$& $ 3^{6}$ & $3^{1+6}PSU_4(2)$ & $3^{1+6}PSU_4(2)$\\
$5B$ & $A_4^4$& $(1234)$& $ 5^{4}$ & $(Frob_{20} \times GO^+_4(5))/2$ &  $5\times 2.(Alt_4\times Alt_4).2$\\
$7B$ & $A_6^3$& $(124)$& $ 7^{3}$ & $7.3.2.L_2(7).2$ & $7\times 2.L_2(7).2$\\
$4C$&
 $A_3^4A_1^2$ & $(1 1 1 1|11)$& $2^24^4$&
$2^{10+3}.Sym_6$&$2^{9+3}.Sym_6$\\
$6E$&
 $A_5^2 A_2^2 A_1^2$& $(11|11|11)$&
$2^43^4$&$6.(GO^+_4(2)\times GO^+_4(3)).2$&$6.(GO^+_4(2)\times GO^+_4(3))$\\
$ 6G$&
$A_5^3 A_1^3$& $(111|1 1 1) $&$2^63^3$&
$6.(GO_3(3)\times PSO^+_4(3)).2$&$6.(GO_3(3)\times PSO^+_4(3))$\\
$8E$&
$A_7^2A_3A_1$ & $(13|1|1)$&$2.4.8^2$&$2^6.(Dih_8\times Sym_4)$&$2^5.(4\times Sym_4)$\\
$10F$&
 $A_9^2A_1^2$& $(13|11)$&$2^45^2$ &$(2\times AGL_1(5)). Dih_8^2$&$10. Dih_8^2$\\
\hline 
\end{longtable}
}


For $\tau\in O(\Lambda)$ and $k\in \bZ_{>0}$,  set
\begin{equation}
\mathcal{L}_{\tau,k}=\{ \lambda+\Lambda_g\in \mathcal{D}(\Lambda_\tau)\mid q(\lambda+\Lambda_\tau)=0,\ o(\lambda+\Lambda_\tau)=k\}, \label{Def:Lg2}
\end{equation}
where $o(\lambda+\Lambda_\tau)$ is the order of $\lambda+\Lambda_\tau$ in $\mathcal{D}(\Lambda_\tau)$.
The following lemmas can be verified by using MAGMA \cite{BLS}.

\begin{lmm}
	Let $\tau$ be an isometry of $\Lambda$ whose conjugacy class is $2A$, $3B$, $5B$ or $7B$. Then  $C_{O(\Lambda_\tau)}(\tau)/\langle\tau \rangle$ acts transitively on the set of all non-zero singular elements. 
\end{lmm}

\begin{lmm}\label{Lem:conjclass1} Let $\tau$ be an isometry of $\Lambda$ whose conjugacy class is $4C$, $6E$, or $8E$.
Let $k$ be a divisor of $|\tau|$.
Then $C_{O(\Lambda_\tau)}(\tau)/\langle \tau\rangle$ acts transitively on $\mathcal{L}_{\tau,k}$. 
\end{lmm}

\begin{lmm}\label{Lem:conjclass6G10F}
Let $\tau$ be an isometry of $\Lambda$ whose conjugacy class is $6G$ or $10F$.
Then, the group $C_{O(\Lambda_\tau)}(\tau)/\langle \tau\rangle$ acts transitively on $\mathcal{L}_{\tau, \, |\tau|/2}$.
\end{lmm}



\begin{thebibliography}{99}	
\bibitem{ATLAS}
J.H. Conway, R.T. Curtis, S.P. Norton, R.A. Parker and  R.A. Wilson,
ATLAS of finite groups. 
Clarendon Press, Oxford, 1985.	\\[-9mm]

\bibitem{B} R.E.~Borcherds, Vertex algebras, Kac-Moody algebras, and the Monster, Proc.
Natl. Acad. Sci. USA, 83 (1986), 3068-3071. \\[-9mm]

\bibitem{BDM}
K.~Barron, C.~Dong, G.~Mason, 
Twisted sectors for tensor product vertex operator algebras associated to 
permutation groups,  Comm. Math. Phys. 227 (2002), no. 2, 349--384. \\[-9mm]


\bibitem{BK}B. Bakalov, and V.\,G.\,Kac, Twisted modules over lattice vertex algebras, 
\emph{Lie theory and its applications in physics V}, 3--26, World Sci. Publ., River Edge, NJ, 2004. \\[-9mm]

\bibitem{BLS} K. Betsumiya, C.H. Lam, H. Shimakura, Automorphism groups of cyclic orbifold vertex operator algebras associated with the Leech lattice and some non-prime isometries,
to appear in \emph{Israel Journal of Mathematics}. \\[-9mm]

\bibitem{CMi}
S.~Carnahan and M.~Miyamoto,  
{\it Rationality of fixed-point vertex operator algebras};  arXiv:1603.05645. \\[-9mm]

\bibitem{CLM}
N.~Chigira, C.H.~Lam and M.~Miyamoto, 
{\it Orbifold construction and Lorentzian construction of Leech lattice vertex operator algebra}, \emph{J. Algebra}, \textbf{593} (2022), 26-71.
\\[-9mm]


\bibitem{ConSl}
J.H.~Conway, N.J.A.~Sloane, 
{\it Sphere packings, lattices and groups. Third edition} Springer-Verlag, New York, 1999. \\[-9mm]

\bibitem{Dong} C.~ Dong, Vertex algebras associated with even lattices, \emph{J. Algebra} {\bf 161}, no. 1, 245--265 (1993). \\[-9mm]

\bibitem{DGM}
L.\ Dolan, P.\ Goddard and P.\ Montague, Conformal field theories, representations and lattice constructions, {\em Comm. Math. Phys.} {\bf 179} (1996), 61--120. \\[-9mm]

\bibitem{DL} C.\ Dong and J.\ Lepowsky,
The algebraic structure of relative twisted vertex operators,
{\it J. Pure Appl. Algebra} {\bf 110} (1996), 259--295.   \\[-9mm]

\bibitem{DLM2}
C.\ Dong, H.\ Li, and G.\ Mason, Modular-invariance of trace functions
in orbifold theory and generalized Moonshine, 
\emph{Comm. Math. Phys.} {\bf 214} (2000), 1--56.  \\[-9mm]

\bibitem{DM2}
C.~Dong, G.~Mason,  
Holomorphic vertex operator algebras of small central charge,
Pacific J. Math., vol 213, No.2, (2004). \\[-9mm]

\bibitem{DN}
C.\ Dong and K.\ Nagatomo, 
Automorphism groups and twisted modules for lattice vertex operator algebras, {\it in} Recent 
developments in quantum affine algebras and related topics (Raleigh, NC, 1998), 117--133,
{\it Contemp. Math.}, {\bf 248}, Amer. Math. Soc., Providence, RI, 1999. 
\\[-9mm]  


\bibitem{DW}
C. Dong and Q. Wang, 
The structure of parafermion vertex operator algebras: general case,
\emph{Comm. Math. Phys.} \textbf{299} (2010), 783--792.
\\[-9mm]


\bibitem{ELMS}
J. van Ekeren, C.H. Lam, S.\ M\"oller and H. Shimakura,  Schellekens' list and the very strange formula, \emph{Adv. Math.} \textbf{380} (2021), 107567.  \\[-9mm] 

\bibitem{EMS}
J. van Ekeren, S.\ M\"oller and N.\ Scheithauer, Construction and Classification of Holomorphic 
Vertex
Operator Algebras, \emph{J. Reine Angew. Math.} \textbf{759} (2020),  61--99.    \\[-9mm] 



\bibitem{EMS1}
J. van Ekeren, S.\ M{\"o}ller and N.R. Scheithauer, Dimension formulae in genus zero and uniqueness of vertex operator
algebras, \emph{Int. Math. Res. Not.}, 2020(7):2145--2204, 2020.  \\[-9mm]

\bibitem{FLM}
I.\ Frenkel, J.\ Lepowsky and A.\ Meurman, Vertex Operator Algebras and the Monster, Pure and Appl.\ Math., Vol.134, Academic Press, Boston, 1988. \\[-9mm]

\bibitem{HaLa}
K.~Harada and M.\,L.~Lang, 
On some sublattices of the Leech lattice, \emph{Hokkaido Mathematical Journal}, \textbf{19} (1990),  435--446.   \\[-9mm] 

\bibitem{Ho}
G.~H\"{o}hn,  On the Genus of the Moonshine Module, arXiv:1708.05990  \\[-9mm]

\bibitem{KLL}
K.\ Kawasetsu, C.H.\ Lam and X.\ Lin, $\mathbb{Z}_2$-orbifold construction associated with $(-1)$-isometry and uniqueness of holomorphic vertex operator algebras of central charge 24, \emph{ Proc. Amer. Math. Soc.}  \textbf{146} (2018), 1937--1950.  \\[-9mm]

\bibitem{Lam}
C.H.\ Lam, On the constructions of holomorphic vertex operator algebras of
central charge $24$, {\em Comm. Math. Phys.} {\bf 305} (2011), 153--198. \\[-9mm]

\bibitem{Lam20} C.H.\ Lam, Cyclic orbifolds of lattice vertex operator algebras having group-like fusions, \emph{Lett. Math. Phys.} \textbf{110} (2020), 1081-1112. \\[-9mm]

\bibitem{LL}
C.~H. Lam and X. Lin, A holomorphic vertex operator algebra of central charge 24 with the
weight one Lie algebra {$F_{4,6}A_{2,2}$}, \emph{J. Pure Appl. Algebra}, 
\textbf{224(3)}, 1241--1279, 2020. \\[-9mm]


\bibitem{LS}
C.H.\ Lam and H.\ Shimakura, Quadratic spaces and holomorphic framed vertex operator algebras of central charge 24, {\em Proc. Lond. Math. Soc.} {\bf 104} (2012), 540--576.  \\[-9mm]

\bibitem{LS2} C.H.\ Lam and H.\ Shimakura, Classification of holomorphic framed vertex operator algebras of central charge 24, {\em Amer. J. Math.} {\bf 137} (2015), 111--137. \\[-9mm]

\bibitem{LS3} C.H.\ Lam and H.\ Shimakura, Orbifold construction of holomorphic vertex operator algebras associated to inner automorphisms, {\em Comm. Math. Phys.} {\bf 342} (2016), 803--841.
\\[-9mm]

\bibitem{LS4} C.H.\ Lam and H.\ Shimakura, A holomorphic vertex operator algebra of central charge 24 whose weight one Lie algebra has type $A_{6,7}$, \emph{Lett. Math. Phys.} \textbf{106} (2016), 1575--1585.
\\[-9mm]


\bibitem{LS19}
C.~H. Lam and H. Shimakura, Reverse orbifold construction and uniqueness of holomorphic vertex
operator algebras, {\em Trans. Amer. Math. Soc.}, \textbf{372(10)}, 7001--7024, 2019. \\[-9mm]

\bibitem{LS20}
C.~H. Lam and H. Shimakura, 
Inertia groups and uniqueness of holomorphic vertex operator
algebras, \emph{Transformation groups}, \textbf{25 (4)}, 1223-1268, 2020.  \\[-9mm]

\bibitem{LS5} C.H.\ Lam and H.\ Shimakura, On orbifold constructions associated with the Leech 
lattice vertex operator algebra, \emph{Mathematical Proceedings of the Cambridge 
	Philosophical Society}, \textbf{168} (2020), 261-285.  \\[-9mm]

\bibitem{LS21}
C.~H. Lam and H. Shimakura, Extra automorphisms of cyclic orbifolds of lattice vertex operator algebras, arXiv:2103.08085.  \\[-9mm]

\bibitem{Le} J.\ Lepowsky, Calculus of twisted vertex operators,  \emph{Proc.
Natl. Acad. Sci. USA} {\bf 82} (1985), 8295--8299. \\[-9mm]

\bibitem{MC} R. McRae, On semisimplicity of module categories for finite non-zero index vertex operator subalgebras, \emph{Lett. Math. Phys.} \textbf{112}, Article number: 25 (2022). \\ [-9mm]

\bibitem{Mi3}
M.\ Miyamoto, A $\bZ_3$-orbifold theory of
lattice vertex operator algebra and
$\bZ_3$-orbifold constructions,
{\it in} Symmetries, Integrable Systems and Representations, 319--344,
{\em Springer Proc. Math. Stat.} {\bf 40}, Springer, Heidelberg, 2013. \\[-9mm]

\bibitem{M}
M.~Miyamoto, 
A $C_2$-cofiniteness of cyclic orbifold models,  \emph{Comm. Math. Phys.} \textbf{335 No. 3}  (2015), 1279-1286. \\[-9mm]

\bibitem{MoScheit}
S.~Moller, N.R.~Scheithauer, 
Dimension formulae and generalised deep holes of the Leech lattice vertex operator algebra;  arXiv:1910.04947. \\[-9mm] 

\bibitem{MoS2}
S.~Moller, N.R.~Scheithauer, 
A geometric classification of the holomorphic vertex operator algebras of central charge $24$,  
arXiv:2112.12291.  \\[-9mm]

\bibitem{Nie}
H.V.~Niemeier, Definite quadratische Formen der Dimension 24 und Diskriminante 1, J. Number Theory, 5 (1973), 142--178. \\[-9mm]

\bibitem{SS}
D.\ Sagaki and H.\ Shimakura, Application of a $\mathbb{Z}_{3}$-orbifold construction to the lattice vertex operator algebras associated to Niemeier lattices, {\em Trans. Amer. Math. Soc.} {\bf 368} (2016), 1621--1646.  \\[-9mm]

\bibitem{SB} R.\ Scharlau and B.\ Blaschke, Reflective integral lattices, \emph{J. Algebra} {\bf 181} (1996), 934--961. \\[-9mm]


\bibitem{SV} R. Scharlau and B.B. Venkov, Classifying lattices using modular forms- a preliminary  report; in: M. Ozeki, E. Bannai, M. Harada (eds.): Codes, Lattices, Modular Forms and Vertex Operator Algebras, Conference Yamagata University, October 2 - 4, 2000 (Proceedings 2001). 
\\[-9mm]

\bibitem{Schel}
A.N.~Schellekens, Meromorphic $c=24$ conformal field theories, Comm. Math. Phys., 153 (1993), 159--185. \\[-9mm]
\end{thebibliography}
\end{document}